\documentclass[11pt,reqno]{amsart}

\usepackage{enumerate}
\usepackage{dsfont}
\usepackage{bbm}
\usepackage{color}
\usepackage{a4wide}
\usepackage[all]{xy}
\usepackage[colorlinks=true,linkcolor=blue]{hyperref}
\usepackage{parskip}
\usepackage[framemethod=tikz]{mdframed}

\hypersetup{%
    ,urlcolor=blue
    ,citecolor=red
    ,linkcolor=blue
    }

\allowdisplaybreaks

\definecolor{mycolor}{rgb}{0.122, 0.435, 0.698}

\newmdenv[innerlinewidth=0.5pt, roundcorner=4pt,linecolor=mycolor,innerleftmargin=6pt,
innerrightmargin=6pt,innertopmargin=6pt,innerbottommargin=6pt]{mybox}

\makeatletter
\renewcommand*{\eqref}[1]{%
  \hyperref[{#1}]{\textup{\tagform@{\ref*{#1}}}}%
}
\makeatother

\def\R{\mathbb{R}}
\def\1{\mathbbm{1}}

\newcommand{\s}{\scriptsize}

\newtheorem{theorem}{Theorem}[section]
\newtheorem{lemma}[theorem]{Lemma}

\theoremstyle{definition}

\newtheorem*{example}{Example}

\theoremstyle{remark}
\newtheorem{remark}[theorem]{Remark}

\theoremstyle{Proposition}
\newtheorem{proposition}[theorem]{Proposition}

\theoremstyle{Corollary}
\newtheorem{corollary}[theorem]{Corollary}

\numberwithin{equation}{section}

\begin{document}
	\newpage
	\title[Asymptotic behavior of the multilevel type error for L\'evy driven SDE]{Asymptotic behavior of the multilevel type error for  SDEs driven by a pure jump L\'evy process }
	
	
\author{Mohamed BEN ALAYA}
\address{Mohamed Ben Alaya, Laboratoire De Math\'ematiques Rapha{\"{e}}l Salem, UMR 6085, Universit\'e De Rouen, Avenue de L'Universit\'e Technop\^ole du Madrillet, 76801 Saint-Etienne-Du-Rouvray, France}
\curraddr{}
\email{mohamed.ben-alaya@univ-rouen.fr}
\thanks{}

\author{Ahmed KEBAIER}
\address{Ahmed, Kebaier, Université Sorbonne Paris Nord, LAGA, CNRS, UMR 7539,  F-93430, Villetaneuse, France}
\curraddr{}
\email{kebaier@math.univ-paris13.fr}
\thanks{This research is supported by Laboratory of Excellence MME-DII, Grant no. ANR11LBX-0023-01 (\url{http://labex-mme-dii.u-cergy.fr/}).  Ahmed Kebaier benefited from the support of the chair Risques Financiers, Fondation du Risque.}

\author{Thi Bao Tram NGO}
\address{Thi Bao Tram NGO, Université Sorbonne Paris Nord, LAGA, CNRS, UMR 7539,  F-93430, Villetaneuse, France}
\curraddr{}
\email{ngo@math.univ-paris13.fr}
\thanks{}
	
	
	
	\date{25/04/2021}
	
	\maketitle
	\begin{abstract}	
		Motivated by  the multilevel Monte Carlo method introduced by Giles \cite{Gil}, we study the asymptotic behavior of the normalized error process $u_{n,m}(X^n-X^{nm})$  where $X^n$ and $X^{nm}$ are respectively Euler approximations with time steps $1/n$ and $1/nm$ of a given stochastic differential equation $X$ driven by a pure jump L\'evy process.  In this paper, we prove that this normalized multilevel error converges to different non-trivial limiting processes with various sharp rates $u_{n,m}$ depending on the behavior of the Lévy measure around zero.  Our results are consistent with those of Jacod \cite{a} obtained for the normalized error $u_n(X^n-X)$, as when letting  $m$ tends to infinity, we  recover the same limiting processes. For the multilevel error, the proofs of the current paper are challenging since unlike \cite{a} we need to deal with $m$ dependent triangular arrays instead of one. 
	\end{abstract}
	
	\section{Introduction}	
Suppose that we are in the probability space $(\Omega, \mathcal{F}, (\mathcal{F}_t)_{0\leq t\leq T}, \mathbb{P})$ endowed with the filtration $\mathcal{F}_s=\sigma(Y_u,u\leq s)$, where $Y$ is a L\'evy process with characteristics $(b,c,F)$ with respect to the truncation function $h(x)=x\1_{\{|x|\leq1\}}$, meaning   
\begin{displaymath}
\mathbb{E}(e^{\texttt{i}uY_t})=\exp\Big\{t \big(\texttt{i}ub-\frac{cu^2}{2}+\int F(dx)(e^{\texttt{i}ux}-1-\texttt{i}ux\1_{\{|x|\leq 1\}}\big)\Big\}.
\end{displaymath}
We consider the L\'evy driven stochastic differential equation (SDE) 
	\begin{equation}\label{eq:1.1}
	X_t=x_0+\int_0^tf(X_{s-})dY_s,\hskip 1cm t\in[0,T],\hskip 0.1cm T>0
	\end{equation}
	where $x_0\in\R$, $f\in\mathcal{C}^3$ (a three-times-differentiable function). Without loss of generality, we assume that $T=1$. In what follows, we consider the continuous Euler approximation  \begin{align}\label{eq:Xn}dX_t^n=f(X^n_{\eta_n(t-)})dY_t,\hskip 1cm t\in[0,1]\end{align} with time step $1/n$, where $\eta_n(t)=\frac{[nt]}{n}$.
	
For the error process $X^n-X$,  Jacod and Protter \cite{b}  proved that the sharp  rate is $\sqrt{n}$  when the characteristic triplet  corresponds to $(b,c,F)$ with $c>0$.  Let us precise that a rate is called sharp if the normalized error converges to a non-trivial  limiting process. Then, Jacod \cite{a} established new sharp rates of convergence different from the classical  $\sqrt{n}$ rate for several cases  with a  L\'evy characteristic triplet  $(b,0,F)$ and $F(\mathbb{R})=\infty$. Those cases correspond to different behaviors of the Lévy measure $F$ around zero.  More recently,  Wang \cite{m} extended the results of Jacod  \cite{a}  for the case of general It\^o semimartingales.  
	
In the current paper, motivated by the multilevel Monte Carlo method  introduced by Giles  \cite{Gil}, we study instead the error between two Euler schemes with different time steps. In particular, we  are interested in determining  sharp rates $u_{n,m}$ for the weak convergence of the error between two consecutive Euler approximations
	$X^n-X^{nm}$ and identifying the corresponding limiting processes.  Here, $X^n$ and $X^{nm}$  stand for the Euler schemes with respectively  time steps  $1/n$  and $1/mn$ that are build  on the same L\'evy paths. In the literature, several papers studied this multilevel type error. Indeed,  when the characteristic triplet of $Y$ is $(b,1,0)$, Ben Alaya and Kebaier \cite{b}  proved  that 
		\begin{displaymath}
		\;\;(Y,	\sqrt{\frac{mn}{(m-1)}}(X^{nm}-X^n))\stackrel{stably}{\Rightarrow}(Y,U),  \mbox{ as } n\to \infty
		\end{displaymath}
		with $m\in\mathbb{N}\backslash\{0,1\}$, $$U_t=\int_0^tf'(X_s)U_sdY_s+\frac{1}{\sqrt{2}}\int_0^tf(X_s)f'(X_s)dW_s$$ and 
			 $W$ is a new standard Brownian motion independent of $Y$.
	 When the characteristic triplet of $Y$ is $(b,c,F)$ with $c\neq0$, Dereich and Li \cite{DereichLi} proved  a similar result with a sharp rate $\sqrt{n}$ under some regularity condition on $F$ around zero with an explicit limiting process. 	
	 
	 Therefore, to fill the gap in the literature for the analysis of this type of error, we consider $Y$ as a L\'evy process with characteristics $(b,0,F)$ where $F$ is  an infinite measure.  More precisely,  for the same cases studied by Jacod \cite{a} we consider in the current work the multilevel type error between two consecutive Euler approximations defined by \begin{align}\label{eq:err}U^{n,m}=X^n_{\eta_{nm}}-X^{nm}_{\eta_{nm}}.\end{align} 
	For this multilevel type error, we use  triangular arrays technics to find the sharp rate of convergence  that turns out to be  faster than   $\sqrt{n}$ which is the usual rate when $Y$ incorporates a continuous Gaussian part. It is worth noticing that the work of \cite{a} when studying the error $X^n_{\eta_{n}}-X_{\eta_{n}}$ needs only to treat one main contributing triangular array. 
However, in the current work, the technical challenge we faced when proving the convergence of the correctly normalized multilevel error \eqref{eq:err}  consists in studying the asymptotic behavior of the joint probability distribution of $m$ triangular arrays. These dependent triangular arrays appear naturally when studying the multilevel error  obtained between the finer discretization with time step $1/nm$ and the coarse one with time step $1/n$.  To overcome this problem, we develop new treatments and proofs to handle these $m$ dependent terms that contribute in the limit in different ways depending on the  assumptions taken on the original Lévy measure $F$ around zero. In more details, besides using the  ``subsequences principle" trick (see e.g.\ Jacod and Protter \cite{JacodProtter2012}), we use arguments of  Sato \cite[Ex.12.8-12.10]{d} that let us avoid complicated calculations of multi-dimensional integrals and rather focus on the pairwise asymptotic behavior of the $m$  marginals and we conclude using technical criteria of Kallenberg \cite[Theorem 15.14 and Corollary 15.16]{e} to prove the weak convergence to the limiting process.

The rest of the paper is organized as follows. In Section \ref{sec:2}, using similar notations, we recall from Jacod \cite{a} some assumptions and estimates on the Lévy measure and also the semimartingale decomposition. Here, in the spirit of Jacod \cite{a}, we precise our consideration for five specific cases. In Section \ref{sec:3}, we introduce  and  prove our main results namely Theorem \ref{thm:tight} for the tightness, Theorem \ref{thm:main1} and Theorem \ref{thm:main} the functional limit theorems for the couple of normalized errors.  Section \ref{error} gives the details of the error analysis to prove our main results with specifying the main and rest terms for each cases and the study of the asymptotic behaviors of the joint distribution of the main terms. The rest terms are treated in appendix \ref{app:A}. Appendices \ref{app:B} and \ref{appC} are dedicated to recall some technical tools that we use throughout the paper.  
	\section{General settings and notations}\label{sec:2}
Let $f$ denotes a real-valued function $f$ satisfying 
	\begin{equation}\label{f}\tag{\textbf{H$_f$}} \textrm{$f\in \mathcal{C}^3$ and globally Lipschitz.}
	 \end{equation}
It is well known that assumption  \eqref{f} guarantees that $\eqref{eq:1.1}$ has an unique non-exploding solution. The crucial factor to find the sharp rates of the multilevel type error \eqref{eq:err} is the behavior of the L\'evy measure $F$ near $0$, which will be expressed through the following functions on $\mathbb{R}_+$:
\begin{align*}
\theta_+(\beta)&=F((\beta,+\infty)),\hspace{1cm}\theta_-(\beta)=F((-\infty,-\beta)),\\
\theta(\beta)&=\theta_+(\beta)+\theta_-(\beta).
\end{align*}	
 Note that from now on, we denote $C$ as a generic constant which may change from line to line. We keep the same framework as in  Jacod \cite{a} and we introduce the four main classes of Lévy measures $F$  that we are  interested in:
	\begin{equation}\label{h1}\tag{\textbf{H$_1^\alpha$}} \textrm{We have } \theta(\beta)\leq\frac{C}{\beta^{\alpha}} \textrm{ for all } \beta\in(0,1].   
	 \end{equation}
	\begin{equation}\label{h2}\tag{\textbf{H$_2^\alpha$}}\textrm{We have } \beta^{\alpha}\theta_+(\beta)\rightarrow\theta_+ \textrm{ and } \beta^{\alpha}\theta_-(\beta)\rightarrow\theta_- \textrm{ as } \beta\rightarrow 0 \textrm{ for some constants } \theta_+\textrm{, } \theta_-\geq0. 
	\end{equation}
	$\textrm{ We set } \theta=\theta_++\theta_->0 \textrm{ and } \theta'=\theta_+-\theta_-,\textrm{ as  } \beta\rightarrow0\textrm{, } \theta(\beta)\sim\frac{\theta}{\beta^{\alpha}}.$
	\begin{equation}\label{h3}\tag{\textbf{H$_3$}}\textrm{The measure }F \textrm{ is symmetrical about }0.
	\end{equation}
	\begin{equation}\label{h4}\tag{\textbf{H$_4$}}\textrm{We have }b=0.
	\end{equation}
We note that \eqref{h1}  is weaker than  \eqref{h2}. Here, the Hypothesis $(\mathbf{H_1^2})$ always holds because the L\'evy measure $F$ integrates $x\mapsto |x|^2\wedge1$. That is 
	\begin{displaymath}
	|\beta|^2\theta(\beta)=\int_{|x|>\beta}(|\beta|^2\wedge1) F(dx)\le\int_{|x|>\beta}(|x|^2\wedge1) F(dx)<+\infty.
	\end{displaymath}
	Now, let us give an example of a process in finance which satisfies the first two hypotheses. 
	\begin{example}
		We consider the $CGMY$ process with L\'evy density
		\begin{align*}
			F(dx)=\left\{\begin{array}{ll}
				Ce^{-Mx}/x^{1+Y}&x>0\\Ce^{Gx}/|x|^{1+Y}&x<0
			\end{array}\right.
		\end{align*} 
		where $C, G, M>0$ and $0<Y<2$. For some sequence $\beta_n\in (0,1]$ and $\beta_n\rightarrow0$ as $n\rightarrow\infty$, this process always satisfies Hypothesis \eqref{h2} and therefore it satisfies Hypothesis \eqref{h1}. More precisely, we have 
		\begin{align*}
			\beta_{n}^Y\theta(\beta_{n})
			=\frac{2C}{Y}+C	\beta_{n}^Y\int_{x>\beta_{n}}\frac{e^{-Mx}+e^{-Gx}-2}{x^{1+Y}}dx.
		\end{align*}
		So, noticing that 
		\begin{align*}
			C	\beta_{n}^Y\int_{x>\beta_{n}}\frac{e^{-Mx}+e^{-Gx}-2}{x^{1+Y}}dx&\leq \frac{C}{Y}(e^{-M\beta_{n}}+e^{-G\beta_{n}}-2)\underset{n\rightarrow\infty}{\rightarrow}0.
		\end{align*}
		we deduce that $\beta_{n}^Y\theta(\beta_{n})\rightarrow \theta$ as $n\rightarrow\infty$, where $\theta=\frac{2C}{Y}$.
	\end{example}
In the same  spirit  as in Jacod  \cite{a}, we prove sharp rate $u_{n,m}$ for our multilevel error $U^{n,m}$  \eqref{eq:err}  with pointing out the technical choice for the sequence $\beta_n$ tending to $0$ that truncates the small jumps. To do so, we consider five different cases depending on some reasonably general circumstances. 
\begin{enumerate}
\renewcommand{\labelenumi}{\textbf{\theenumi}}
\renewcommand{\theenumi}{ (C\arabic{enumi})}
	\item\label{1}  If  \eqref{h1} is valid for some $\alpha<1$ and $d:=b-\int_{|x|\leq1}xF(dx)\neq0$, then we choose $u_{n,m}=\frac{nm}{m-1}$ and $\beta_n=\frac{(\log{n})^2}{n}$. (See the Remark \ref{rmk:2} for the finiteness of $d$)
	\item\label{2} If  \eqref{h2} is valid  for some $\alpha<1$ and hypotheses \eqref{h3} and \eqref{h4} are also valid then we choose $u_{n,m}=\left[\frac{mn}{(m-1)\log{n}}\right]^{1/{\alpha}}$ and $\beta_{n}=\left(\frac{\log{n}}{n}\right)^{1/{\alpha}}$.
	\item\label{3} If  \eqref{h2} is valid  for  $\alpha=1$ and $F$ is non-symmetric then we choose $u_{n,m}=\frac{mn}{(m-1)(\log{n})^2}$ and $\beta_{n}=\frac{\log{n}}{n}$. 
	\item\label{4} If  \eqref{h2} is valid  for  $\alpha=1$  and hypothesis \eqref{h3} is also valid then we choose $u_{n,m}=\frac{mn}{(m-1)\log{n}}$ and $\beta_{n}=\frac{\log{n}}{n}$.
	\item\label{5} If  \eqref{h2} is valid  for some $\alpha>1$ then we choose $u_{n,m}=\left[\frac{mn}{(m-1)\log{n}}\right]^{1/{\alpha}}$ and $\beta_{n}= \frac{\log{n}}{n^{1/(2\alpha)}}$.
\end{enumerate}
   
  	\subsection{Some estimates on Lévy measure}\label{subsec:not}
	In what follows we consider the same notations as in  Jacod  \cite{a} and  for $\beta>0$, we denote
\begin{align}
	&c(\beta)=\int_{|x|\leq\beta}|x|^2F(dx),\nonumber\\&d_+(\beta)=\int_{x>\beta}|x|F(dx),\hskip 1cm d_-(\beta)=\int_{x<-\beta}|x|F(dx),\nonumber\\&\rho_+(\beta)=\int_{x>\beta}|x|^{\alpha}F(dx),\hskip 1cm \rho_-(\beta)=\int_{x<-\beta}|x|^{\alpha}F(dx),\label{3.1}\\&\delta(\beta)=d_+(\beta)+d_-(\beta),\hskip 1cm \rho(\beta)=\rho_+(\beta)+\rho_-(\beta),\nonumber\\&d'(\beta)=d_+(\beta)-d_-(\beta),\hskip 1cm b'=b+\int_{|x|>1}xF(dx),\nonumber\\&d(\beta)=b'-d'(\beta)\nonumber.
\end{align}
Note that $d(\beta)=b-\int_{\beta<|x|\leq1}xF(dx)$ if $\beta<1$ and $d(\beta)=b$ if $\beta=1$.\\\\
\noindent{$\bullet$} Without loss of generality we can reduce ourselves to study the case where we have bounded jumps and coefficient $f$ with compact support. Thus, from now on we assume that 
\begin{enumerate}
\renewcommand{\labelenumi}{\textbf{\theenumi}}
\renewcommand{\theenumi}{(A)}
\item \label{A} $f\in\mathcal{C}^3$ with compact support and $|\Delta Y|\leq p$  for some integer $p\geq1$, which amounts to say that $\theta(p)=0$. 
\end{enumerate}
 Indeed, adapting  the same arguments as in Proposition 2.4 in Jacod in \cite{a} to the multilevel error setting, we can easily recover our main results namely Theorem \ref{thm:tight}, Theorem \ref{thm:main1} and Theorem \ref{thm:main}  for non-bounded jumps and coefficient $f$ without a compact support.  
\begin{remark}\label{rmk:1} 
	Note that under \ref{A} the quantity  $\int_{\mathbb R} |x|^{a} F(dx)$, $a\ge2$ is finite. 
\end{remark} 
In this part, we recall from Jacod in \cite{a} under assumption \ref{A} some useful estimations on the above quantities introduced in \eqref{3.1}. We provide some details for the proofs of the following lemmas in appendix \ref{app:B}.
\begin{lemma}\label{lem:estimation1}
	Since $\theta(p)=0$, under \eqref{h1}, we have for any $\beta\in(0,1]$
	\begin{align}\label{eq:3.3}
		\left\{\begin{array}{l}
			c(\beta)\leq C\beta^{2-\alpha}, \hskip 1cm \rho(\beta)\leq
			C\log\left(\frac{1}{\beta}\right)\\\int_{|x|>\beta}|x|^{\alpha/2}F(dx)\le\frac{C}{\beta^{\alpha/2}}\\
			\delta(\beta)+|d(\beta)|+d_+(\beta)+d_-(\beta)+|d'(\beta)|\leq C s(\beta)\end{array}\right.
		\mbox{where }s(\beta)=\left\{\begin{array}{ll}1,&\alpha<1\\\log(\frac{1}{\beta}),&\alpha=1\\\beta^{1-\alpha},&\alpha>1\end{array}\right.
	\end{align}
\end{lemma}
\begin{remark}\label{rmk:2}
	Note that under \eqref{h1} with $\alpha<1$, we have  $|\delta(\beta)|<C$ for all $\beta\in(0,1]$, then by the monotone convergence theorem  $\int |x|F(dx)<\infty$.
	\end{remark}
\begin{lemma}\label{lem:estimation2}
	If further \eqref{h2} holds, then we obtain the following equivalences or convergences as $\beta$ goes to $0$,
	\begin{align}
		\left\{\begin{array}{lll}\label{eq:3.5}c(\beta)\sim\frac{\alpha\theta}{2-\alpha}\beta^{2-\alpha},& &\\
			\frac{\rho_+(\beta)}{\log{1/{\beta}}}\rightarrow\alpha\theta_+& \frac{\rho_-(\beta)}{\log{1/{\beta}}}\rightarrow\alpha\theta_-, &\\
			\beta^{\alpha-1}d_+(\beta)\rightarrow\frac{\alpha\theta_+}{\alpha-1}& \beta^{\alpha-1}d_-(\beta)\rightarrow\frac{\alpha\theta_-}{\alpha-1}& \textrm{if } \alpha>1 \\
			\frac{d_+(\beta)}{\log{(1/\beta)}}\rightarrow \theta_+&	\frac{d_-(\beta)}{\log{(1/\beta)}}\rightarrow \theta_-& \textrm{if } \alpha=1\\
			d_+(\beta)\rightarrow d_+&d_-(\beta)\rightarrow d_-&\textrm{if }\alpha<1 \end{array}\right.
	\end{align} 
	with some positive constants $d_+$ and $d_-$.
\end{lemma}
\begin{lemma}\label{lem:estimation3}
	When $\alpha=1$, under assumption \eqref{h2}, we have for every $b>0$ and as $\beta\rightarrow0$
	\begin{align}\label{eq:3.6}
		\frac{1}{(\log{(1/\beta)})^2}\int_{\beta<|x|\le b}(x\log{|x|})F(dx)\rightarrow -\frac{\theta'}{2}.
	\end{align}	
\end{lemma}
Besides, for a given truncating sequence  $(\beta_{n})_{n\ge 0}$  that tends to zero as $n$ tends to infinity, we introduce 
\begin{align*}
	\boxed{c_{n}:=c(\beta_{n}), \;\;d_{n} :=d(\beta_{n}), \;\;d'_{n} :=d'(\beta_{n}),\;\;\rho_{n}:=\rho(\beta_{n}) \;\;\mbox{ and }\delta_{n}:=\delta(\beta_{n}),}
\end{align*}
to make the notations less cluttered.
We deduce easily from \eqref{eq:3.3} that under \eqref{h1}, we have
	\begin{align}\left\{\begin{array}{l}\label{eq:3.13}
			c_{n}\leq C\beta_{n}^{2-\alpha},\\
			{d'}_n+|d_{n}|+\delta_{n}\leq Cs(\beta_n),\\
			\rho_n\leq
			C\log\left(\frac{1}{\beta_n}\right).\\
		\end{array}\right.
	\end{align}
\subsection{Semimartingale decomposition}\label{1.2}
Now, we give a decomposition of the process $Y$.\\\\
\noindent{$\bullet$} For a predictable real function $\delta$ on ${\Omega}\times\R_+\times\R$  and a real measure $m$, we denote $\delta\ast m$ the stochastic integral process given by $\delta\ast m_t=\int_0^t\int_{\R^d}\delta(s,x)m(ds,dx)$ for $t\ge0$.
Let $\mu$ denotes the jump measure of our driving Lévy process $Y$ and $\nu(ds,dx)=ds\otimes F(dx)$ is its predictable compensator. For $\beta>0$, we can write 
	\begin{align}
	Y&=Y^{\beta}+N^{\beta}\label{eq:3.7}, \;\mbox{ where }\, Y^{\beta}=A^{\beta}+M^{\beta}\;\mbox{ with }\\
	A_t^{\beta}&=d(\beta)t,\;\;\;M_t^{\beta}=x\1_{\{|x|\leq\beta\}}\ast(\mu-\nu)_t \;\;\mbox{ and } \;
	N_t^{\beta}=x\1_{\{|x|>\beta\}}\ast\mu_t.\notag
	\end{align}
Then $M^{\beta}$ is a square-integrable martingale with predictable bracket $<M^{\beta},M^{\beta}>_t=c(\beta)t$. Moreover under assumption \ref{A}, for $\beta\ge p$ we have $N^{\beta}=0$ and then $Y=A^{\beta}+M^{\beta}$ with $A_t^{\beta}=b't$, whereas for $\beta=1$ we have $A^1_t=bt$.\\\\
\noindent{$\bullet$} In the context of the multilevel type error  \eqref{eq:err}, we consider two time discretization grids. The coarse grid with time step $1/n$ and with associated times $t_i:=\frac{i-1}{n}$ for all $i\in\{1,\hdots,n+1\}$. The finer grid with time step $1/nm$ and with associated times 
$$t_i^k:=\frac{m(i-1)+k-1}{nm},$$
with $i\in\{1,\dots,n\}$, $k\in\{1,\dots,m+1\}$ and $m\in\mathbb N\setminus\{0,1\}$.
Note that $t_1^1=t_0$ and $t_{i}^1=t_i$ corresponding to the coarser grid. We note also that the point of the coarse grid  can be written either as a final point $t_{i-1}^{m+1}$ or as the next point $t_{i}^1$ on the same grid.
Now, for a  given truncating sequence $(\beta_n)_{n\ge0}$ that tends to zero as $n\to\infty$, we denote $$ \boxed{M^{\beta_n}_{t_i^k,t}=M_{t}^{\beta_{n}}-M_{t_i^k}^{\beta_{n}} }\;  \mbox{ for }  t\in I(nm,i,k):= (t_i^k,t_i^{k+1}],$$
with $(i,k)\in\{1,\dots,n\}\times\{1,\dots,m\}$. Further, let {$\Big(T^{\beta_n}_p({ t_i^k})\Big)_{p\ge 1}$  denotes the sequence of the successive jump times of $Y$ after time $t_i^k$ and of size bigger than or equal to $\beta_n$. Let also $K({\s t_i^k})$ denotes the random number of jumps occurring in the time interval $I(nm,i,k)$ that satisfies $T^{\beta_n}_{ \s K({\s t_i^k})}{( t_i^k)}\le t_i^{k+1}< T^{\beta_n}_{ \s K({\s t_i^k})+1}{( t_i^k)}$. Note that the random number $K(t_i^k)$ is well-defined as we use the cut-off of size $\beta_n$. Then, the following two main properties hold:
\begin{enumerate}
	\renewcommand{\labelenumi}{\textbf{\theenumi}}
	\renewcommand{\theenumi}{(P\arabic{enumi})}
\item\label{I} Conditionally on $\mathcal{F}_{t_i^k}$, the random variables   $(\Delta Y_{T^{\beta_n}_{ \s p}( t_i^k)})_{p\geq1}$, $K({\s t_i^k})$  and $\{M^{\beta_n}_{t_i^k,t},t\in I(nm,i,k)\}$ are independent.  Each $\Delta Y_{T^{\beta_n}_{ \s p}( t_i^k)}$ has the density  $\frac{1}{\theta(\beta_{n})}1_{|x|>\beta_{n}}F(dx)$ and $K({\s t_i^k})$ has Poisson law with parameter $$\boxed{\lambda_{n,m}=\frac{\theta(\beta_{n})}{nm}.}$$
	\item\label{M} The process $\{M^{\beta_n}_{t_i^k,t},t\in I(nm,i,k)\}$ is a L\'evy process, independent of $\mathcal{F}_{t_i^k}$ and satisfying for $t\in I(nm,i,k)$, for $v\in\mathbb{R}$
	\begin{displaymath}
	\mathds{E}(e^{\texttt{i}vM^{\beta_n}_{t_i^k,t}})=\exp\left\{\left(t-{t_i^k}\right)\int_{|y|\leq\beta_{n}}(e^{\texttt{i}vy}-1-\texttt{i}vy)F(dy)\right\}.
	\end{displaymath}
\end{enumerate}
We also note that under Hypothesis \eqref{h1}, with some choice $\beta_n$ going to $0$, we have
$\lambda_{n,m}\rightarrow0$ as $n\rightarrow\infty$. 
\section{Main results}\label{sec:3} Our main results are to prove the convergence in law of 
	$u_{n,m}U^{n,m}$ to the  limit process with the above choices of the rate $u_{n,m}$ corresponding to those cases. First of all, we assume that function $f$ always satisfies assumption \eqref{f}. The theorem below is about the tightness which can be easily deduced by Lemma \ref{lem:tight1}, Lemma \ref{lem:rest1}, Lemma \ref{lem:tight2}, Lemma \ref{lem:rest2}, Lemma \ref{lem:tight3}, Lemma \ref{lem:rest3}, Lemma \ref{lem:tight5}, Lemma \ref{lem:rest5}, Corollary \ref{Cor:tight}  and Lemma \ref{lem:tightrest}  in appendix \ref{appC}.
\begin{theorem}\label{thm:tight}
	Assume that hypothesis \eqref{h1} holds for some $\alpha \in (0,2)$. Then, with the above choice of $u_{n,m}$ in the previous section, the sequence $(u_{n,m}U^{n,m})$ is tight.
\end{theorem}
 Let $\overline{Y}^{n}$ be the discretized process associated with $Y$, that is $\overline{Y}^{n}_t=Y_{\eta_{n}(t)}$. We observe that the sequence $\overline{Y}^{n}$ converges pointwise to the process $Y$ for the Skorohod topology. The following limit theorem considering the error between two consecutive levels Euler approximations is covered by \ref{1}.
\begin{theorem}	\label{thm:main1}
For case \ref{1}, the sequence 
	\begin{displaymath}
	(\overline{Y}^{n},u_{n,m}{U}^{n,m})\stackrel{\mathcal{L}}{\longrightarrow}(Y,U), \textrm{ when }n\rightarrow\infty
	\end{displaymath}
	where $U$ is the unique solution of the linear equation 
	\begin{equation}\label{eq:1.6}
		U_t=\int_0^tf'(X_{s-})U_{s-}dY_s-Z_t, \quad t\in[0,1]
	\end{equation}
	and  when letting $n\rightarrow\infty$,
		\begin{multline}\label{eq:01}
		Z_t=d\sum\limits_{\substack{n:R_n\le t}}\left(ff'(X_{R_n-})\Delta Y_{R_n}\frac{\lfloor m\Upsilon_n\rfloor}{m-1}+(f(X_{R_n})-f(X_{R_n-}))(1-\frac{\lfloor m\Upsilon_n\rfloor}{m-1})\right)\\+\frac{d^2}{2}\int_0^tf(X_{s-})f'(X_{s-})ds
	\end{multline}
where $d=b-\int_{|x|\leq1}xF(dx)$, $(R_n)_{n\geq1}$ denotes an enumeration of the jump times of $Y$ (or of $X$) and $(\Upsilon_n)_{n\ge1}$ is a sequence of i.i.d. variables, uniform on $[0,1]$ and independent of $Y$.
\end{theorem}
\begin{proof}
This proof uses some results in Section \ref{error}. In particular, we begin with the decomposition of $u_{n,m}U^{n,m}$ from \eqref{decUn}. From Theorem \ref{thm:noise1} and  Lemma \ref{lem:rest1},  we have $(\overline{Y}^{n},u_{n,m}{Z}^{n,m})\stackrel{\mathcal{L}}{\longrightarrow}(Y,Z)$. Then the result is straight-forward according to Theorem \ref{thm:final}.
\end{proof}
\begin{remark}\label{rmk:3}
	From Theorem \ref{thm:main1}, if we let  $m\rightarrow\infty$, we will recover the same limit as the case of Euler (Case 3a in Jacod's paper \cite{a}), that is
	\begin{align*}
		Z_t=&d\sum_{n:R_n\leq t}([f(X_{R_n})-f(X_{R_n-})]\Upsilon_n+f'(X_{R_n-})\Delta X_{R_n}(1-\Upsilon_n))+\frac{d^2}{2}\int_0^tf(X_{s-})f'(X_{s-})ds,
	\end{align*}
	where $d=b-\int_{|x|\leq1}xF(dx)$ and $(\Upsilon_n)_{n\geq1}$ is a sequence of i.i.d. variables, uniform on $[0,1]$ and independent of $Y$, and $(R_n)_{n\geq1}$ is an enumeration of the jump times of $Y$ (or of $X$).
	\end{remark}
The following limit theorem is of the rest cases and with stronger assumption \eqref{h2}.
\begin{theorem}	\label{thm:main}
	We have that the sequence 
	\begin{displaymath}
	(\overline{Y}^{n},u_{n,m}U^{n,m})\stackrel{\mathcal{L}}{\longrightarrow}(Y,U), \textrm{ when }n\rightarrow\infty
	\end{displaymath}
	where $U$ is the unique solution of the linear equation $\eqref{eq:1.6}$ 
	and the process $Z$ is described as follows:
	\begin{enumerate}[(a)]
		\item In \ref{2} and \ref{4}, we have $\eqref{eq:1.7}$, where $V$ is another L\'evy process, independent of $Y$ and characterized by 
		\begin{equation}
		\label{eq:1.10}
		\mathds{E}(e^{iuV_t})=\exp\left(\frac{t\theta^2\alpha}{4}\int\frac{1}{|x|^{1+\alpha}}(e^{iux}-1-iux\1_{\{|x|\leq1\}})dx\right).
		\end{equation}
		Hence, $V$ is a symmetric stable process with index $\alpha$.
			\item In \ref{3}, 
		\begin{equation}\label{eq:1.9}
		Z_t=-\frac{{\theta'}^2}{4}\int_0^tf(X_{s-})f'(X_{s-})ds,
		\end{equation}
		and we even have that $u_{n,m}U^{n,m}$ converges to $U$ in probability (locally uniformly in time).
			\item  In \ref{5}, 
		\begin{equation}\label{eq:1.7}
		Z_t=\int_0^tf(X_{s-})f'(X_{s-})dV_s
		\end{equation} 
		where $V$ is a L\'evy process, independent of $Y$ and characterized by
		\begin{equation}
		\label{eq:1.8}
		\mathds{E}(e^{iuV_t})=\exp\left(\frac{t\alpha}{2}\int\left\{\left[(\theta_+^2+\theta_-^2)\1_{\{x>0\}}+2\theta_+\theta_-\1_{\{x<0\}}\right]\times\frac{1}{|x|^{1+\alpha}}(e^{iux}-1-iux)\right\}dx\right).
		\end{equation}
		Hence, $V$ is stable with index $\alpha$.
	\end{enumerate}
\end{theorem}
\begin{proof}
This proof also uses some of the results in Section \ref{error}. In particular, we begin with the decomposition of $u_{n,m}U^{n,m}$ from \eqref{decUn}. From Theorem \ref{thm:noise24}, Theorem \ref{thm:noise3}, Theorem \ref{thm:noise5}, Lemma \ref{lem:rest2}, Lemma \ref{lem:rest3} and Lemma \ref{lem:rest5}, we have $(\overline{Y}^{n},u_{n,m}{Z}^{n,m})\stackrel{\mathcal{L}}{\longrightarrow}(Y,Z)$. Then the result is straight-forward according to Theorem \ref{thm:final}.
\end{proof}
\section{Error analysis}\label{error}
As mentioned in subsection \ref{subsec:not}, with no loss of generality, we develop our error analysis under assumption \ref{A}.  For $t\in[0,1]$, we first recall that $\eta_n(t)=\frac{[nt]}{n}$. The error between two consecutive levels Euler approximations $U^{n,m}_t:=X^n_{\eta_{nm}(t)}-X^{nm}_{\eta_{nm}(t)}$ is given by
\begin{align*}
	U^{n,m}_{\eta_n(t)}=&\int_0^{\eta_{n}(t)}\Big(f(X_{\eta_{nm}(s-)}^n)-f(X_{\eta_{nm}(s-)}^{nm})\Big)dY_s-\int_0^{\eta_{n}(t)}\Big(f(X_{\eta_{nm}(s-)}^n)-f(X_{\eta_{n}(s-)}^{n})\Big)dY_s\\
	=&\int_0^{\eta_{n}(t)}\Big(f(X_{\eta_{nm}(s-)}^{nm}+U^{n,m}_{s-})-f(X_{\eta_{nm}(s-)}^{nm})\Big)dY_s\\&\hspace{6.7cm}-\int_0^{\eta_{n}(t)}\Big(f(X_{\eta_{nm}(s-)}^n)-f(X_{\eta_{n}(s-)}^{n})\Big)dY_s.\\
\end{align*}
Let $G(x,y):=f(x+yf(x))-f(x)$, using the interpolation Euler scheme   $X^{n}_{\eta_{nm}(s-)}=X^{n}_{\eta_n(s-)}+f(X_{\eta_{n}(s-)}^{n})(Y_{\eta_{nm}(s-)}-Y_{\eta_n(s-)})$, we deduce that
\begin{align}\label{func:G}
	U^{n,m}_{\eta_n(t)}&=\int_0^{\eta_{n}(t)}\Big(f(X_{\eta_{nm}(s-)}^{nm}+U^{n,m}_{s-})-f(X_{\eta_{nm}(s-)}^{nm})\Big)dY_s\notag\\&\hspace{5cm}-\int_0^{\eta_{n}(t)}  G(X^{n}_{\eta_n(s-)}, Y_{\eta_{nm}(s-)}-Y_{\eta_n(s-)})dY_s.
\end{align}
\begin{remark}\label{rmk:4}
Under assumption \ref{A},	 using  Taylor's expansion we write 
	\begin{displaymath}
	G(x,y)=yff'(x)+y^2k(x,y),
	\end{displaymath}
	where $k$ is a $\mathcal{C}^1$ function that vanishes outside $K\times\mathbb{R}$ for some compact subset $K\subset \mathbb{R}$. Also, we note that  $ff'$ has compact support.
\end{remark}
Consequently,  given  a deterministic rate of convergence $u_{n,m}$, we write
\begin{align}\label{decUn}
	u_{n,m}U^{n,m}_{\eta_n(t)}=&u_{n,m}\int_0^{\eta_{n}(t)}\Big(f(X_{\eta_{nm}(s-)}^{nm}+U^{n,m}_{s-})-f(X_{\eta_{nm}(s-)}^{nm})\Big)dY_s-u_{n,m}Z^{n,m}_t,
\end{align}
where
\begin{multline}\label{Znm}
	Z^{n,m}_t:=\int_0^{\eta_{n}(t)}ff'(X_{\eta_n(s-)}^n)(Y_{\eta_{nm}(s-)}-Y_{\eta_n(s-)})dY_s\\+\int_0^{\eta_{n}(t)}k(X^n_{\eta_n(s-)},Y_{\eta_{nm}(s-)}-Y_{\eta_n(s-)})(Y_{\eta_{nm}(s-)}-Y_{\eta_n(s-)})^2dY_s.
\end{multline}
Now, recalling the notation $\overline{Y}^{n}_t=Y_{\eta_{n}(t)}$, for $t\in[0,1]$, we follow the same arguments as in    \cite[Theorem 9.3 page 40]{h} to prove the following result.
\begin{proposition}\label{thm:final}
For $n\to\infty$,  if  $(\overline{Y}^{n},u_{n,m}Z^{n,m})\stackrel{\mathcal{L}}{\rightarrow}(Y,Z)$ then $(\overline{Y}^{n},u_{n,m}U^{n,m}_{\eta_n(.)})\stackrel{\mathcal{L}}{\rightarrow}(Y,U)$, where the limiting process $U$ is solution to  $\eqref{eq:1.6}$.
\end{proposition}
\begin{proof}
	By \eqref{decUn}, we have
	\begin{multline*} 
	\hspace{-0.4cm}	u_{n,m}U^{n,m}_{\eta_n(t)}=\int_0^{\eta_{n}(t)}\frac{f(X_{\eta_{nm}(s-)}^{nm}+U^{n,m}_{s-})-f(X_{\eta_{nm}(s-)}^{nm})}{U^{n,m}_{s-}}(u_{n,m}U^{n,m}_{s-})\1_{\{U^{n,m}_{s-}\neq0\}}dY_s-u_{n,m}Z^{n,m}_t.
	\end{multline*}
	Let $T^{n,a}=\inf{\{t>0:|u_{n,m}U^{n,m}_{t}|>a\}}$. Then the sequence $(u_{n,m}U_{t\wedge T^{n,a}}^{n,m})$ is relatively compact, and any limit point will satisfy $\eqref{eq:1.6}$ on $[0,T^a]$, where $T^a=\inf{\{t>0:|U_t|>a\}}$. But $\lim_{a\uparrow\infty}T^a=\infty$ a.s., so $u_{n,m}U^{n,m}_{\eta_n(.)}\stackrel{\mathcal{L}}{\rightarrow} U$.
\end{proof}
\begin{corollary}\label{Cor:tight} The tightness of the sequence $(\overline{Y}^{n},u_{n,m}U^{n,m}_{\eta_n(.)})$  is a straight-forward consequence of  the tightness of the sequence $(\overline{Y}^{n},u_{n,m}Z^{n,m})$. 
\end{corollary}
Thus, the aim now is to study the asymptotic behavior of  the couple $(\overline{Y}^{n},u_{n,m}Z^{n,m})$. To do so, on the one hand,  we set  
\begin{equation}\label{dec:M+R}
u_{n,m}Z^{n,m}_t=:\mathcal{M}^{n,m}_t+\mathcal{R}^{n,m}_t,
\end{equation}
where $\mathcal{M}^{n,m}$ stands for the main term contributing in the limit behavior and 
$\mathcal{R}^{n,m}_t$ stands for the rest term that will tends to zero. In the sequel, the expression of the above decomposition  has to be specified for each  case \ref{1}-\ref{5}. 
It is worth noticing that the second term in \eqref{Znm} will not contribute  in the limit and will be considered as a part of $\mathcal{R}^{n,m}$ except for  \ref{1} where it will be considered as a part of ${\mathcal{M}}^{n,m}$. On the other hand,  we also need to rewrite the process $\overline{Y}^{n}$  in a triangular array form. To do so, recalling our notations given in subsection \ref{1.2}, with the formula \eqref{eq:3.7} and taking into account the number of jumps $K(t_i^k)$ occuring in the time interval $I(nm,i,k)$,  with a truncating sequence $(\beta_n)_{n\ge 0}$, we write
\begin{align}\label{Y}\begin{array}{cl}
\overline{Y}^{n}:=\overline{Y}^{n}(1)+\overline{Y}^{n}(2),&\mbox{where}\\
\overline{Y}^{n}_t(1)=\sum_{i=1}^{[nt]}y^{n,m}_i(1),&y^{n,m}_i(1):=\sum_{k=1}^m(M^{\beta_n}_{t_i^k,t_i^{k+1}}+\sum_{j\geq2}\Delta Y_{T^{\beta_n}_{ \s j}( t_i^k)}\1_{\{K({\s t_i^k})\geq j\}}),\\
\overline{Y}^{n}_t(2)=\sum_{i=1}^{[nt]}y^{n,m}_i(2),&y^{n,m}_i(2):=\sum_{k=1}^m\left(\frac{d_{n}}{nm}+\Delta Y_{T^{\beta_n}_{ \s 1}( t_i^k)}\1_{\{K({\s t_i^k})\geq 1\}}\right).\end{array}
\end{align}
 In particular, $(\overline{Y}^n(1))_{n\ge0}$ converges uniformly in probability to zero for all cases except case \ref{3} and $(\overline{Y}^n(2))_{n\ge0}$ is tight for all cases. 
 Each subsection below is dedicated to study separately each case. Note that from now on, we denote $C$ as some positive generic constant that can be changed from line to line. Moreover, by the notation $\Gamma^n\stackrel{\mathds{P}}{\longrightarrow}0$, we mean that $\sup_{s\leq t}|\Gamma_s^n|$ goes to $0$ in probability for all $t$ as $n$ tends to infinity.
\subsection{Asymptotic behavior of the couple $(\overline{Y}^{n},u_{n,m}Z^{n,m})$ for case \ref{1}.}\label{section1}	
For this subsection, we first need to introduce some complementary notations. Following subsection \ref{subsec:not}, 
let  $(T^{\beta_n}_{ \s p}(t_i^1, t_i^k))_{p\ge 1}$ denotes the sequence of jump times and $K({\s t_i^1,t_i^k})$ the random number of jumps that occur on the interval $(t_i^1,t_i^k]$, where we recall that $t_i^1=\frac{i-1}{n}$ and $t_i^k=\frac{m(i-1)+k-1}{nm}$.
Then,  each $\Delta Y_{T^{\beta_n}_{ \s p}( t_i^1,t_i^k)}$ has the density  $\frac{1}{\theta(\beta_{n})}\1_{\{|x|>\beta_{n}\}}F(dx)$ and $K({\s t_i^1,t_i^k})$  has a Poisson distribution with parameter $(k-1)\lambda_{n,m}$.	By the representation formula  $\eqref{eq:3.7}$, we obtain a decomposition for the normalized error term $u_{n,m}Z^{n,m}_t$. More precisely, from \eqref{Znm}, we write
	\begin{align*}
		u_{n,m}Z^{n,m}_t	
		&=\sum_{i=1}^5\int_0^{\eta_n(t)}ff'(X_{\eta_n(s-)}^n)d\Gamma_s^n(i)\\&+\overline{\Gamma}^n_t(1)+\int_0^{\eta_n(t)}k(X^n_{\eta_n(s-)},Y_{\eta_{nm}(s-)}-Y_{\eta_n(s-)})d\overline{\Gamma}^n_s(3)+\overline{\Gamma}^n_t(4)+\overline{\Gamma}^n_t(5),
	\end{align*}
	where
	\begin{align*}\left\{\begin{array}{rl}
		\Gamma^n_t(1)=&u_{n,m}\int_0^t(A^{\beta_n}_{\eta_{nm}(s-)}-A^{\beta_n}_{\eta_{n}(s-)})dA^{\beta_n}_s,\\
		\Gamma^n_t(2)=&u_{n,m}\int_0^{t}(A^{\beta_n}_{\eta_{nm}(s-)}-A^{\beta_n}_{\eta_{n}(s-)})dN^{\beta_n}_s+u_{n,m}\int_0^t(N^{\beta_n}_{\eta_{nm}(s-)}-N^{\beta_n}_{\eta_{n}(s-)})dA^{\beta_n}_s,\\
		\Gamma^n_t(3)=& u_{n,m}\int_0^t(M^{\beta_n}_{\eta_{nm}(s-)}-M^{\beta_n}_{\eta_{n}(s-)})dY_s,\\ 	
		\Gamma^n_t(4)=&u_{n,m}\int_0^t(N^{\beta_n}_{\eta_{nm}(s-)}-N^{\beta_n}_{\eta_{n}(s-)})dN^{\beta_n}_s,\\
			\Gamma^n_t(5)=& u_{n,m}\int_0^t[(A^{\beta_n}_{\eta_{nm}(s-)}-A^{\beta_n}_{\eta_{n}(s-)})+(N^{\beta_n}_{\eta_{nm}(s-)}-N^{\beta_n}_{\eta_{n}(s-)})]dM^{\beta_n}_s,\\
		\overline{\Gamma}^n_t(1)=&\frac{u_{n,m}d_n}{nm}\sum_{i=1}^{[nt]}\sum_{k=2}^mk(X^n_{t_i^1},\Delta Y_{T^{\beta_n}_1(t_i^1,t_i^k)})(N^{\beta_n}_{t_i^k}-N^{\beta_n}_{t_i^1})^2,\\
		\overline{\Gamma}^n_t(2)=&u_{n,m}\int_0^t(N^{\beta_n}_{\eta_{nm}(s-)}-N^{\beta_n}_{\eta_{n}(s-)})^2dA^{\beta_n}_s,\\
		\overline{\Gamma}^n_t(3)=&u_{n,m}\int_0^t(Y_{\eta_{nm}(s-)}-Y_{\eta_{n}(s-)})^2dY_s-\overline{\Gamma}^n_t(2),\\
	\overline{\Gamma}^n_t(4)=&\int_0^{\eta_n(t)}(k(X^n_{\eta_n(s-)},Y_{\eta_{nm}(s-)}-Y_{\eta_n(s-)})-k(X^n_{\eta_n(s-)},N^{\beta_{n}}_{\eta_{nm}(s-)}-N^{\beta_{n}}_{\eta_n(s-)}))d\overline{\Gamma}^n_s(2),\\
\overline{\Gamma}^n_t(5)=&\frac{u_{n,m}d_n}{nm}\sum_{i=1}^{[nt]}\sum_{k=2}^m(k(X^n_{t_i^1},N^{\beta_{n}}_{t_i^k}-N^{\beta_{n}}_{t_i^1})-k(X^n_{t_i^1},\Delta Y_{T^{\beta_n}_1(t_i^1,t_i^k)}))(N^{\beta_n}_{t_i^k}-N^{\beta_n}_{t_i^1})^2.\end{array}\right.
	\end{align*}
 Now, we rewrite $
\overline{\Gamma}^n_t(1)=\overline{\Gamma}^n_t(1,1)+\overline{\Gamma}^n_t(1,2),
$
where 
\begin{align*}
	\overline{\Gamma}^n_t(1,1)&=\frac{u_{n,m}d_n}{nm}\sum_{i=1}^{[nt]}\sum_{k=2}^mk(X^n_{t_i^1},\Delta Y_{T^{\beta_n}_1(t_i^1,t_i^k)})(\Delta Y_{T^{\beta_n}_1(t_i^1,t_i^k)})^2\1_{\{K({\s t_i^1,t_i^k})\geq1\}},\\
 \overline{\Gamma}^n_t(1,2)&=\frac{u_{n,m}d_n}{nm}\sum_{i=1}^{[nt]}\sum_{k=2}^mk(X^n_{t_i^1},\Delta Y_{T^{\beta_n}_1(t_i^1,t_i^k)})(\sum_{h=2}^{K({\s t_i^1,t_i^k})}(\Delta Y_{T^{\beta_n}_h(t_i^1,t_i^k)})^2+\sum_{\substack{h,h'=2\\h\neq h'}}^{K({\s t_i^1,t_i^k})}\Delta Y_{T^{\beta_n}_h(t_i^1,t_i^k)}\Delta Y_{T^{\beta_n}_{h'}(t_i^1,t_i^k)}),\\
 \end{align*}
with the convention $\sum\limits_{i}^{j}$ for $j< i$ equals to zero. For the term driven by $\Gamma^n(2)$, we rewrite
 \begin{align*}
	\int_0^{\eta_n(t)}ff'(X_{\eta_n(s-)}^n)d\Gamma^n_s(2)&=\Gamma^n_t(2,1)+\Gamma^n_t(2,2)+\Gamma^n_t(2,3),
\end{align*}
where 
\begin{align*}
	\Gamma^n_t(2,1)&=\frac{u_{n,m}d_n}{nm}\sum_{i=1}^{[nt]}\sum_{k=2}^mff'(X_{t_i^1}^n)\left[(k-1)\Delta Y_{T^{\beta_n}_1(t_i^k)}\1_{\{K({\s t_i^k})\geq1\}}\1_{\{K({\s t_i^1,t_i^k})=0\}}+\Delta Y_{T^{\beta_n}_1(t_i^1,t_i^k)}
\1_{\{K({\s t_i^1,t_i^k})\geq1\}}\right],\\
	\Gamma^n_t(2,2)&=\frac{u_{n,m}d_n}{nm}\sum_{i=1}^{[nt]}\sum_{k=2}^mff'(X_{t_i^1}^n)\Big[(k-1)\sum_{h=2}^{K({\s t_i^k})}\Delta Y_{T^{\beta_n}_h(t_i^k)}+\sum_{h=2}^{K({\s t_i^1,t_i^k})}\Delta Y_{T^{\beta_n}_h(t_i^1,t_i^k)}\Big],\\
	\Gamma^n_t(2,3)&=\frac{u_{n,m}d_n}{nm}\sum_{i=1}^{[nt]}\sum_{k=2}^mff'(X_{t_i^1}^n)(k-1)\Delta Y_{T^{\beta_n}_1(t_i^k)}\1_{\{K({\s t_i^k})\geq1\}}\1_{\{K({\s t_i^1,t_i^k})\geq1\}}.
\end{align*}
Then we rewrite $\Gamma^n_{t}(2,1)+\overline{\Gamma}^n_{t}(1,1)
:=\sum_{i=1}^{[nt]}(\tilde{z}^n_i(1)+\tilde{z}^n_i(2))$ where
\begin{align*}
	\tilde{z}^n_i(1)=&\frac{u_{n,m}d_n}{nm}\sum_{k=2}^mff'(X^n_{t^{1}_i})(k-1)\Delta Y_{T^{\beta_n}_1(t_i^k)}\1_{\{K(t_i^k)\geq1\}}\1_{\{K({\s t_i^1,t_i^k})=0\}},\\
	\tilde{z}^n_i(2)=&\frac{u_{n,m}d_n}{nm}\sum_{k=2}^mG(X^n_{t^{1}_i},\Delta Y_{T^{\beta_n}_1(t^{1}_i,t_i^k)})\1_{\{K(t^{1}_i,t_i^k)\geq1\}}.
\end{align*}	
Now, concerning $\tilde{z}^n_i(2)$, we rewrite the sum $\sum_{k=2}^mG(X^n_{t^{1}_i},\Delta Y_{T^{\beta_n}_1(t^{1}_i,t_i^k)})\1_{\{K(t^{1}_i,t_i^k)\geq1\}}$ as follows
\begin{align*}
	&\sum_{k=2}^mG(X^n_{t^{1}_i},\Delta Y_{T^{\beta_n}_1(t^{1}_i,t_i^k)})\1_{\{K(t^{1}_i,t_i^k)\geq1\}}\1_{\{K(t^{1}_i)\geq1\}}+\sum_{k=2}^mG(X^n_{t^{1}_i},\Delta Y_{T^{\beta_n}_1(t^{1}_i,t_i^k)})\1_{\{K(t^{1}_i,t_i^k)\geq1\}}1_{K(t^{1}_i)=0}\\=&(m-1)G(X^n_{t^{1}_i},\Delta Y_{T^{\beta_n}_1(t^{1}_i)})\1_{\{K(t^{1}_i)\geq1\}}+\sum_{k=2}^mG(X^n_{t^{1}_i},\Delta Y_{T^{\beta_n}_1(t^{1}_i,t_i^k)})\1_{\{K(t^{1}_i,t_i^k)\geq1\}}\1_{\{K(t^{1}_i)=0\}}\1_{\{K(t^{2}_i)\geq1\}}\\
	&+\sum_{k=2}^mG(X^n_{t^{1}_i},\Delta Y_{T^{\beta_n}_1(t^{1}_i,t_i^k)})\1_{\{K(t^{1}_i,t_i^k)\geq1\}}\1_{\{K(t^{1}_i)=0\}}\1_{\{K(t^{2}_i)=0\}}\end{align*}
As $K(t^{k}_i)=K(t^{k}_i,t_i^{k+1})$ for $k\in\{1,\hdots,m\}$, the first term in the first for $k=2$ vanishes  and the 2 first terms in the second sum vanishes also. Then, we have
\begin{multline*}
	\sum_{k=2}^mG(X^n_{t^{1}_i},\Delta Y_{T^{\beta_n}_1(t^{1}_i,t_i^k)})\1_{\{K(t^{1}_i,t_i^k)\geq1\}}=(m-1)G(X^n_{t^{1}_i},\Delta Y_{T^{\beta_n}_1(t^{1}_i)})\1_{\{K(t^{1}_i)\geq1\}}\\+(m-2)G(X^n_{t^{1}_i},\Delta Y_{T^{\beta_n}_1(t^{2}_i)})\1_{\{K(t^{1}_i,t^{2}_i)=0\}}\1_{\{K(t^{2}_i)\geq1\}}
	+\sum_{k=4}^mG(X^n_{t^{1}_i},\Delta Y_{T^{\beta_n}_1(t^{1}_i,t_i^k)})\1_{\{K(t^{1}_i,t_i^k)\geq1\}}\1_{\{K(t^{1}_i,t^{2}_i)=0\}}.
\end{multline*}	
Then, by induction, we deduce $$\tilde{z}^n_i(2)=\frac{u_{n,m}d_n}{nm}	\sum_{k=1}^{m-1} (m-k)G(X^n_{t^{1}_i},\Delta Y_{T^{\beta_n}_1(t^{k}_i)})\1_{\{K(t^{1}_i,t^{k}_i)=0\}}\1_{\{K(t^{k}_i)\geq1\}}.$$
Therefore, we can rewrite  $\Gamma^n_{t}(2,1)+\overline{\Gamma}^n_{t}(1,1)=\frac{u_{n,m}d_n}{nm}\sum_{k=1}^m\sum_{i=1}^{[nt]}{\tilde{z}}_{i,k}^n,$ where \begin{multline}\label{eq:z}{\tilde{z}}_{i,k}^n=[ff'(X^n_{t^{1}_i})(k-1)\Delta Y_{T^{\beta_n}_1(t^{k}_i)}+(m-k)G(X^n_{t^{1}_i},\Delta Y_{T^{\beta_n}_1(t_i^k)})]\1_{\{K(t^{1}_i,t_i^k)=0\}}\1_{\{K(t_i^k)\geq1\}}.\end{multline}
 In this case, we have  $u_{n,m}Z^{n,m}_t=\mathcal{M}^{n,m}_t+\mathcal{R}^{n,m}_t,$ with
 \begin{align}\label{R1}
 	\mathcal{M}^{n,m}_t&=\int_0^{\eta_n(t)}ff'(X_{\eta_n(s-)}^n)d\Gamma_s^n(1)+{\Gamma}^n_t(2,1)+	\overline{\Gamma}^n_t(1,1)\mbox{ and}\nonumber\\
	\mathcal{R}^{n,m}_t	
	&={\Gamma}^n_t(2,2)+{\Gamma}^n_t(2,3)+\sum_{i=3}^5\int_0^{\eta_n(t)}ff'(X_{\eta_n(s-)}^n)d\Gamma_s^n(i)\\&+	\overline{\Gamma}^n_t(1,2)+\int_0^{\eta_n(t)}k(X^n_{\eta_n(s-)},Y_{\eta_{nm}(s-)}-Y_{\eta_n(s-)})d\overline{\Gamma}^n_s(3)+\overline{\Gamma}^n_t(4)+\overline{\Gamma}^n_t(5).\nonumber
\end{align}
The proof of the following lemma is postponed in Appendix \ref{app:A} below.
\begin{lemma}\label{lem:rest1}
For case \ref{1}, we have as $n\rightarrow\infty$ 
	the sequences of processes $(\overline{Y}^n(1))_{n\ge0}$ and  $(\mathcal{R}^{n,m})_{n\ge0}$ converge uniformly in probability to $0$.
\end{lemma}
 \begin{lemma}\label{lem:tight1}
For case \ref{1}, we have the sequences  $(\overline{Y}^n(2))_{n\ge0}$ and $({\mathcal{M}}^{n,m})_{n\ge0}$ are tight.  
\end{lemma}
\begin{proof}
	First, we consider  $\overline{Y}^n_t(2)$ given by \eqref{Y}.
	By the Property \ref{I},  the relation \eqref{3.1} in particular $d_n=b'-d'_n$, the fact that $1-\lambda_{n,m}-e^{-\lambda_{n,m}}\le \frac{\lambda_{n,m}^2}{2}$, we have
		\begin{align}\label{eq:y2a}
			|\mathbb{E}(y^n_i(2)|\mathcal{F}_{t^{1}_i})|&= \left|\frac{d_{n}}{n}+(1-e^{-\lambda_{n,m}})\frac{{d'}_n}{n\lambda_{n,m}}\right|
			\le C\frac{1+\lambda_{n,m}|{d'}_n|}{n}.
			\end{align}
		Further, by the Jensen's inequality,  $\int_{\R}|x|^2F(dx)<+\infty$ (see Remark \ref{rmk:1}) and  the inequality $1-e^{-\lambda_{n,m}}\le \lambda_{n,m}$, we have
		\begin{align}\label{eq:y2b}
			\mathbb{E}(|y^n_i(2)|^2|\mathcal{F}_{t^{1}_i})&\le C(\frac{d_{n}^2}{n^2}+(1-e^{-\lambda_{n,m}})\frac{1}{n\lambda_{n,m}})\le C\frac{1}{n}(\frac{d_n^2}{n}+1).
		\end{align}
		Then, in case \ref{1}, using the boundedness of $|d'_n|$ and $|d_n|$ (see \eqref{eq:3.3}), $\lambda_{n,m}\rightarrow0$ as $n\rightarrow\infty$,  $y^n_i(2)$ satisfies \eqref{eq:2.8} and \eqref{eq:2.11} ensuring the tightness from the second part of Lemma \ref{lem:tightrest}.\\
		 Now, we consider the tightness of the sequence $(\mathcal{M}^{n,m})_{n\ge0}$. In this case, we recall that
		$
			\mathcal{M}^{n,m}_t=\int_0^{\eta_n(t)}ff'(X_{\eta_n(s-)}^n)d\Gamma_s^n(1)+\Gamma^n_{t}(2,1)+\overline{\Gamma}^n_{t}(1,1).
		$
		We rewrite 
	$
				\int_0^{\eta_n(t)}ff'(X_{\eta_n(s-)}^n)d\Gamma_s^n(1)=\sum_{i=1}^{[nt]}\zeta_i^n(1),$ $\zeta_i^n(1)=ff'(X_{t_i^k}^n)\frac{m(m-1)u_{n,m}d_n^2}{2n^2m^2},
		$
	\begin{align*}\begin{array}{ll}
				\overline{\Gamma}^n_{t}(1,1)=\sum_{i=1}^{[nt]}z_i^n(1),&z_i^n(1)=\frac{u_{n,m}d_n}{nm}\sum_{k=2}^mk(X^n_{t^{1}_i},\Delta Y_{T^{\beta_n}_1(t^{1}_i,t_i^k)})(\Delta Y_{T^{\beta_n}_1(t^{1}_i,t_i^k)})^2\1_{\{K(t^{1}_i,t_i^k)\geq1\}},\\
				\Gamma^n_{t}(2,1)=\sum_{i=1}^{[nt]}z_i^n(2),&z_i^n(2)=\frac{u_{n,m}d_n}{nm}\sum_{k=2}^mff'(X_{t_i^1}^n)\left[(k-1)\Delta Y_{T^{\beta_n}_1(t_i^k)}\1_{\{K(t_i^k)\geq1\}}\1_{\{K(t^{1}_i,t_i^k)=0\}}\right.\\&\left.\hspace{7cm}+\Delta Y_{T^{\beta_n}_1(t^{1}_i,t_i^k)}\1_{\{K(t^{1}_i,t_i^k)\geq1\}}\right].\end{array}
		\end{align*}
	By similar arguments and the fact that the functions $ff'$ and $k$  are bounded, we have $\mathbb{E}(|\zeta_i^n(1)||\mathcal{F}_{t^{1}_i})\le C\frac{u_{n,m}d_n^2}{n^2}$, $\mathbb{E}(|z_i^n(1)||\mathcal{F}_{t^{1}_i})\le C\frac{u_{n,m}|d_n|}{n^2}$ and $\mathbb{E}(|z_i^n(2)||\mathcal{F}_{t^{1}_i})\le C\frac{u_{n,m}|d_n|\delta_n}{n^2}$.
		Then, we are in case \ref{1}, using the boundedness $|d_n|$ and  $\delta_n$ (see \eqref{eq:3.3}) and $u_{n,m}=\frac{nm}{m-1}$, we obtain that $\zeta_i^n(1)$, $z_i^n(1)$ and $z_i^n(2)$ satisfy \eqref{eq:2.7} ensuring the tightness from Lemma \ref{lem:tightrest}.
\end{proof}	
\begin{theorem}\label{thm:noise1}
	For case \ref{1},  we have  
	$$(\overline{Y}^{n},{\mathcal{M}}^{n,m})\stackrel{\mathcal{L}}{\longrightarrow}(Y,Z),$$
	where $Z$ is the limit process given in $\eqref{eq:01}$.
\end{theorem}
\begin{proof}Since $ff'$ is Lipschitz-continuous, by virtue of Lemma \ref{lem:cv1tot}, in order to prove the convergence in law of $(\overline{Y}^{n},{\mathcal{M}}^{n,m})$ it suffices to prove the convergence of $(\overline{Y}^{n}_1(2),\Gamma^n_1(1),\Gamma^n_{1}(2,1)+\overline{\Gamma}^n_{1}(1,1))$. \\
 First as $u_{n,m}=\frac{nm}{m-1}$, we use \eqref{eq:3.3} to get $d_n\searrow d=b-\int_{|x|<1}xF(dx)$ and then  we have
		\begin{align*}
		\Gamma^n_1(1)=u_{n,m}\sum_{i=1}^{n}\sum_{k=2}^m  (A^{\beta_n}_{t_i^k}-A^{\beta_n}_{t^1_i})(A^{\beta_n}_{t^{k+1}_i}-A^{\beta_n}_{t_i^k})
		=\frac{(m-1)u_{n,m} d_n^2}{2mn}\underset{n\rightarrow\infty}{\longrightarrow}\frac{d^2}{2}.
		\end{align*}
		Thus, it is enough to prove the convergence in law of the pair $(\overline{Y}^{n}_1(2),\Gamma^n_{1}(2,1)+\overline{\Gamma}^n_{1}(1,1))$. To do so, we recall that $\Gamma^n_{1}(2,1)+\overline{\Gamma}^n_{1}(1,1)=\frac{u_{n,m}d_n}{nm}\sum_{k=1}^m\sum_{i=1}^{n}{\tilde{z}}_{i,k}^n,$ where ${\tilde{z}}_{i,k}^n$ is given by \eqref{eq:z}. 
		 Now, following Jacod \cite{a}, we first study the above triangular array with freezing the component  $X_{t^{1}_i}^n$. Thus, we treat the triangular array with  \begin{align*}{\tilde{z}}^n_{i,k}(z)=[ff'(z)(k-1)\Delta Y_{T^{\beta_n}_1(t_i^k)}+(m-k)G(z,\Delta Y_{T^{\beta_n}_1(t_i^k)})]\1_{\{K(t^{1}_i,t_i^k)=0\}}\1_{\{K(t_i^k)\geq1\}},\end{align*} $i\in\{1,\hdots,n \}$, $k\in\{1,\hdots,m\}$ with an arbitrary value $z\in\R$.  Next, based on the denotations in \eqref{Y}, we introduce $\overline{Y'}^{n}_t(2)=\sum_{i=1}^{[nt]}\sum_{k=1}^m{y'}_{i,k}^n(2)$, where ${y'}_{i,k}^n(2)=\Delta Y_{T^{\beta_n}_1(t_i^k)}\1_{\{K(t_i^k)\geq 1\}}$. In this case, we observe that $\overline{Y}^{n}_1(2)-\overline{Y'}^{n}_1(2)=\sum_{i=1}^{n}\sum_{k=1}^m\frac{d_n}{nm}\underset{n\rightarrow\infty}{\longrightarrow}d$ and $\frac{u_{n,m}d_n}{nm}\underset{n\rightarrow\infty}{\longrightarrow}\frac{d}{(m-1)}.$ 
		 Then instead of working with $(\overline{Y}^{n}_1(2),\Gamma^n_{1}(2,1)+\overline{\Gamma}^n_{1}(1,1))$, it is enough to prove the convergence in law of $((\sum_{i=1}^{n}{y'}_{i,k}^n(2),\sum_{i=1}^{n}{\tilde{z}}_{i,k}^n(z)), k\in\{1,\hdots,m\})\in(\R^2)^m$  to $m$ independent $\R^2$-L\'evy processes. 
		 First, since this $(\R^2)^m$-vector  is tight, it is enough to prove that every weakly convergent subsequence has the same limit. In what follows, we omit the notation for the subsequence for more readability. For the independence of the components of the limit vector, by using Ex.12.8-12.10 in \cite{d}, we only need to prove the independence between the limit marginals  of $(\sum_{i=1}^{n}{y'}_{i,k}^n(2),\sum_{i=1}^{n}{y'}_{i,k'}^n(2))$, $(\sum_{i=1}^n{y'}_{i,k}^n(2),\sum_{i=1}^n\tilde{z}^n_{i,k'}(z))$ and $(\sum_{i=1}^n\tilde{z}_{i,k}^n(z),\sum_{i=1}^n\tilde{z}^n_{i,k'}(z))$, for any $k,k'\in\{1,\dots,m\}, k\not =k'$.
		 		 
	$\bullet$ First, by the independent structure of  the subsequence marginals $\sum_{i=1}^{n}({y'}_{i,k}^n(2),{y'}_{i,k'}^n(2))$ for any $k,k'\in\{1,\dots,m\}, k\not =k'$, it is obvious that the limit marginals are independent. 
	
		$\bullet$ Second, for fixed $k,k'\in\{1,\dots,m\},k\not = k'$, we consider  the sequence $\sum_{i=1}^n({y'}_{i,k}^n(2),\tilde{z}^n_{i,k'}(z))$ whose variables $({y'}_{i,k}^n(2),\tilde{z}^n_{i,k'}(z))$, ${1\leq i\leq n}$ are i.i.d. For $k>k'$,  ${y'}_{i,k}^n(2)$ and  $\tilde{z}^n_{i,k'}(z)$ are obviously independent and the independence of the limit marginals is straightforward. For $k<k'$, we use Lemma \ref{lem:cv} to identify the limit characteristics.  Let us denote the law of $({y'}_{i,k}^n(2),\tilde{z}^n_{i,k'}(z))$ by $K^{k,k'}_{n,m}$. 
	We will study the convergence of $nK^{k,k'}_{n,m}(h)$    to $K^{k,k'}(h)$ with some function $h$  to be precised later where 
	\begin{align*}
	K^{k,k'}(h):=\frac{1}{m}\int h(x,0) F(dx)+\frac{1}{m}\int h(0,ff'(z)(k'-1)y+(m-k')G(z,y)) F(dy).
\end{align*}
	Therefore, we have $nK^{k,k'}_{n,m}(h)=n\mathbb{E}(h({y'}_{i,k}^n(2),\tilde{z}^n_{i,k'}(z))|\mathcal{F}_{t^{1}_i})$ 
	\begin{align*}
	=& n\mathbb{E}(h({y'}_{i,k}^n(2),\tilde{z}^n_{i,k'}(z)); K(t_i^k)=0,K(t^{k'}_i)=0|\mathcal{F}_{t^{1}_i})\\&+ n\mathbb{E}(h({y'}_{i,k}^n(2),\tilde{z}^n_{i,k'}(z));K(t^{k'}_i)\geq1,K(t^{1}_i,t^{k'}_i)=0|\mathcal{F}_{t^{1}_i})\\
	&+ n\mathbb{E}(h({y'}_{i,k}^n(2),\tilde{z}^n_{i,k'}(z));K(t_i^k)=0,K(t^{k'}_i)\geq1,K(t^{1}_i,t^{k'}_i)\geq1|\mathcal{F}_{t^{1}_i})\\
	&+n\mathbb{E}(h({y'}_{i,k}^n(2),\tilde{z}^n_{i,k'}(z));K(t_i^k)\geq1,K(t^{k'}_i)=0|\mathcal{F}_{t^{1}_i})\\
	&+n\mathbb{E}(h({y'}_{i,k}^n(2),\tilde{z}^n_{i,k'}(z));K(t_i^k)\geq1,K(t^{k'}_i)\geq1|\mathcal{F}_{t^{1}_i})\\
	=
	&ne^{-2\lambda_{n,m}}h(0,0)+e^{-\lambda_{n,m}(k'-1)}\frac{1-e^{-\lambda_{n,m}}}{m\lambda_{n,m}}\int_{|y|>\beta_n}h(0,ff'(z)(k'-1)y+(m-k')G(z,y))F(dy)\\
	&+n\mathbb{P}(K(t_i^k)=0,K(t^{1}_i,t^{k'}_i)\geq1|\mathcal{F}_{t^{1}_i})(1-e^{-\lambda_{n,m}})h(0,0)\\
	&+e^{-\lambda_{n,m}}\frac{(1-e^{-\lambda_{n,m}})}{m\lambda_{n,m}}\int_{|x|>\beta_n}h(x,0)F(dx)
	+\frac{(1-e^{-\lambda_{n,m}})^2}{m\lambda_{n,m}}\int_{|x|>\beta_n}h(x,0)F(dx).	\end{align*}
	Three following headings demonstrate three elements in  Lemma \ref{lem:cv}, each corresponding to some specific choices of function $h$.
	
		\paragraph{\textbf{Concerning assertion (i)}:}
		Since we work with bounded jumps, we observe that $|ff'(z)(k-1)x+(m-k)G(z,x)|\le C|x|$ and $|ff'(z)(k-1)x|\le C|x|$ on $x\in[-p,p]$.  We shall choose $h=h_{u,v}$ where $h_{u,v}(x,y)=\1_{\{|x|\ge u,|y|\ge v\}}$  bounded and vanishing on a neighborhood of $0$, with $(u,v)\not =(0,0)$. For any function $P$ satisfying $|P(y)|\le C|y|$, if  we have $C|y|\le v$ then $|P(y)|\le v$, which yields  $h_{u,v}(x,P(y))=h_{u,v}(x,P(y)){\1}_{\{C|y|>v\}}$ for any $x,y\in[-p,p]$.
		Then, we have $nK^{k,k'}_{n,m}(h_{u,v})\underset{n\rightarrow\infty}{\longrightarrow}K^{k,k'}({h_{u,v}})$.
		\paragraph{\textbf{Concerning assertion (ii)}:} For this case, we choose  $h={h'},{h''}$ where  ${h'}(x,y)=x\1_{\{x^2+y^2\le1\}}$ and ${h''}(x,y)=y\1_{\{x^2+y^2\le1\}}$. 	As $|ff'(z)(k-1)x+(m-k)G(z,x)|\le C|x|$ and $|ff'(z)(k-1)x|\le C|x|$ on $x\in[-p,p]$, using the dominated convergence theorem, $\int |x|F(dx)<C$ and $\frac{\theta(\beta_n)}{n}=m\lambda_{n,m}$ converges to $0$ as $n$ tends to infinity, we obtain  that
		 $nK^{k,k'}_{n,m}({h'})$ and $nK^{k,k'}_{n,m}({h''})$ converge respectively to $K^{k,k'}({h'})$ and $K^{k,k'}({h''})$ when $n$ goes to infinity.	
		\paragraph{\textbf{Concerning assertion (iii)}:} Here we take $h=h_1,h_2,h_3$ where ${h}_{1}(x,y)=x^2\1_{\{x^2+y^2\le1\}}$, ${h}_{2}(x,y)=xy\1_{\{x^2+y^2\le1\}}$ and ${h}_{3}(x,y)=y^2\1_{\{x^2+y^2\le1\}}$ and apply similar arguments as in (ii), we get $nK^{k,k'}_{n,m}({h_1})$, $nK^{k,k'}_{n,m}({h_2})$ and $nK^{k,k'}_{n,m}({h_3})$ converge respectively to $K^{k,k'}({h_1})$, $K^{k,k'}({h_2})$ and $K^{k,k'}({h_3})$ when $n$ goes to infinity.
In conclusion, for any fixed $k<k'$ the obtained limit pair has independent marginals, since it has no Gaussian part and its L\'evy measure $K^{k,k'}$ is supported on the union of the coordinate axes (see e.g. Ex.12.8  in \cite{d}).
		
			$\bullet$
		Third, we consider the pair  $\sum_{i=1}^n(\tilde{z}^n_{i,k}(z),\tilde{z}_{i,k'}^n(z))$ for fixed $k,k'\in\{1,\dots,m\}$, $k\not =k'$ and  by symmetry it is enough to study only the case $k<k'$. Let us denote the law of $(\tilde{z}^n_{i,k}(z),\tilde{z}_{i,k'}^n(z))$ by $L^{k,k'}_{n,m}$.  We will prove that $nL^{k,k'}_{n,m}(h)$ converges to  $L^{k,k'}(h)$ with some function $h$ to be precised later and where 
		\begin{multline*}
			L^{k,k'}(h):=\frac{1}{m}\int h(ff'(z)(k-1)x+(m-k)G(z,x),0) F(dx)\\+\frac{1}{m}\int h(0,ff'(z)(k'-1)y+(m-k')G(z,y)) F(dy).
		\end{multline*}
	Therefore, we have $nL^{k,k'}_{n,m}(h)=n\mathbb{E}(h({\zeta}_{i,k}^n(z),{\zeta}_{i,k'}^n(z))|\mathcal{F}_{t^{1}_i})$
		\begin{align*}
			=& n\mathbb{E}(h(\tilde{z}^n_{i,k}(z),\tilde{z}_{i,k'}^n(z)); K(t_i^k)=0,K(t^{k'}_i)=0|\mathcal{F}_{t^{1}_i})\\&+ n\mathbb{E}(h(\tilde{z}^n_{i,k}(z),\tilde{z}_{i,k'}^n(z));
			K(t^{k'}_i)\geq1,K(t^{1}_i,t^{k'}_i)=0|\mathcal{F}_{t^{1}_i})
			\\&+ n\mathbb{E}(h(\tilde{z}^n_{i,k}(z),\tilde{z}_{i,k'}^n(z));
			K(t_i^k)=0,
			K(t^{k'}_i)\geq1,K(t^{1}_i,t^{k'}_i)\ge1|\mathcal{F}_{t^{1}_i})\\
			&+n\mathbb{E}(h(\tilde{z}^n_{i,k}(z),\tilde{z}_{i,k'}^n(z));K(t_i^k)\geq1,K(t^{k'}_i)=0,K(t^{1}_i,t^{k}_i)=0|\mathcal{F}_{t^{1}_i})\\
			&+n\mathbb{E}(h(\tilde{z}^n_{i,k}(z),\tilde{z}_{i,k'}^n(z));K(t_i^k)\geq1,K(t^{k'}_i)=0,K(t^{1}_i,t_i^k)\geq1|\mathcal{F}_{t^{1}_i})\\
			&+n\mathbb{E}(h(\tilde{z}^n_{i,k}(z),\tilde{z}_{i,k'}^n(z));K(t_i^k)\geq1,K(t^{k'}_i)\geq1,K(t^{1}_i,t_i^k)=0|\mathcal{F}_{t^{1}_i})\\
			&+n\mathbb{E}(h(\tilde{z}^n_{i,k}(z),\tilde{z}_{i,k'}^n(z));K(t_i^k)\geq1,K(t^{k'}_i)\geq1,K(t^{1}_i,t_i^k)\geq1|\mathcal{F}_{t^{1}_i})\\
			=&ne^{-2\lambda_{n,m}}h(0,0)\\
			&+e^{-\lambda_{n,m}(k'-1)}\frac{1-e^{-\lambda_{n,m}}}{m\lambda_{n,m}}\int_{|y|>\beta_n}h(0,ff'(z)(k'-1)y+(m-k')G(z,y))F(dy)\\
			&+n\mathbb{P}(K(t_i^k)=0,K(t^{1}_i,t^{k'}_i)\ge1|\mathcal{F}_{t^{1}_i})(1-e^{-\lambda_{n,m}})h(0,0)\\
				&+e^{-\lambda_{n,m}-\lambda_{n,m}(k-1)}\frac{1-e^{-\lambda_{n,m}}}{m\lambda_{n,m}}\int_{|x|>\beta_n}h(ff'(z)(k-1)x+(m-k)G(z,x),0)F(dx)\\
			&+ne^{-\lambda_{n,m}}(1-e^{-\lambda_{n,m}})(1-e^{-\lambda_{n,m}(k-1)})h(0,0)\\
			&+\frac{(1-e^{-\lambda_{n,m}})^2e^{-\lambda_{n,m}(k-1)}}{m\lambda_{n,m}}\int_{|x|>\beta_n}h(ff'(z)(k-1)x+(m-k)G(z,x),0)F(dx)\\
			&+n(1-e^{-\lambda_{n,m}})^2(1-e^{-\lambda_{n,m}(k-1)})h(0,0).
		\end{align*}
Now, to check the three conditions of  Lemma \ref{lem:cv}, we use similar arguments as in the second point above .

\paragraph{\textbf{Concerning assertion (i)}:}
Since  $|ff'(z)(k-1)x+(m-k)G(z,x)|\le C|x|$ and $|ff'(z)(k-1)x|\le C|x|$ on $x\in[-p,p]$, we choose $h=h_{u,v}$ where $h_{u,v}(x,y)=\1_{\{|x|\ge u,|y|\ge v\}}$  bounded and vanishing on a neighborhood of $0$, with $(u,v)\not =(0,0)$. For any functions $P_1$ and $P_2$ satisfying $|P_1(x)|\le C|x|$ and $|P_2(y)|\le C|y|$, we have  $h_{u,v}(P_1(x),P_2(y))=h_{u,v}(P_1(x),P_2(y)){\1}_{\{C|x|\ge u,C|y|\ge v\}}$ for any $x,y\in[-p,p]$.
By similar arguments, we have $nL^{k,k'}_{n,m}(h_{u,v})\underset{n\rightarrow\infty}{\longrightarrow}L^{k,k'}({h_{u,v}})$.
\paragraph{\textbf{Concerning assertion (ii)}:} For this case, we choose  $h={h'},{h''}$ where  ${h'}(x,y)=x\1_{\{x^2+y^2\le1\}}$ and ${h''}(x,y)=y\1_{\{x^2+y^2\le1\}}$. 	As $|ff'(z)(k-1)x+(m-k)G(z,x)|\le C|x|$ and $|ff'(z)(k-1)x|\le C|x|$ on $x\in[-p,p]$, and applying similar arguments used in (ii) of the second point, we obtain  that
$nL^{k,k'}_{n,m}({h'})$ and $nL^{k,k'}_{n,m}({h''})$ converge respectively to $L^{k,k'}({h'})$ and $L^{k,k'}({h''})$ when $n$ goes to infinity.
\paragraph{\textbf{Concerning assertion (iii)}:} Here we take ${h}_{1}(x,y)=x^2\1_{\{x^2+y^2\le1\}}$, ${h}_{2}(x,y)=xy\1_{\{x^2+y^2\le1\}}$ and ${h}_{3}(x,y)=y^2\1_{\{x^2+y^2\le1\}}$ and apply similar arguments as in (iii) of the second point, we get $nL^{k,k'}_{n,m}({h_1})$, $nL^{k,k'}_{n,m}({h_2})$ and $nL^{k,k'}_{n,m}({h_3})$ converge respectively to $L^{k,k'}({h_1})$, $L^{k,k'}({h_2})$ and $L^{k,k'}({h_3})$ when $n$ goes to infinity.	In conclusion, for any fixed $k<k'$,  the obtained limit pair has independent marginals, no Gaussian part and its L\'evy measure $L^{k,k'}$ is supported on the union of the coordinate axes (see e.g. Ex.12.8 in \cite{d}).\\
Hence, combining the three above points, the independence of the $m$ $\R^2$-Lévy limits of $\{(\sum_{i=1}^{n}{y'}_{i,k}^n(2),\sum_{i=1}^{n}{\tilde{z}}_{i,k}^n(z)), k\in\{1,\hdots,m\}\}\in(\R^2)^m$ is shown (see e.g. Ex.12.9-Ex.12.10 in \cite{d}).
Now,  we turn to  identify the limit marginals of our vector. 
 Let us denote the law of $({y'}_{i,k}^n(2),\tilde{z}^n_{i,k}(z))$ by $K^{k}_{n,m}$. 	 We will prove that $nK^{k}_{n,m}(h)$ converges to  $K^{k}(h)$  where
\begin{align*}
	K^{k}(h)=\frac{1}{m}\int h(x,ff'(z)(k-1)x+(m-k)G(z,x)) F(dx).
\end{align*}
Therefore, we have $nK^{k}_{n,m}(h)=n\mathbb{E}(h({y'}_{i,k}^n(2),\tilde{z}^n_{i,k}(z))|\mathcal{F}_{t^{1}_i})$
	\begin{align*}
		=&n\mathbb{E}(h({y'}_{i,k}^n(2),\tilde{z}^n_{i,k}(z)); K(t_i^k)=0|\mathcal{F}_{t^{1}_i}) +n\mathbb{E}(h({y'}_{i,k}^n(2),\tilde{z}^n_{i,k}(z));K(t_i^k)\geq1,K(t^{1}_i,t_i^k)\geq1|\mathcal{F}_{t^{1}_i})\\
		&+n\mathbb{E}(h({y'}_{i,k}^n(2),\tilde{z}^n_{i,k}(z));K(t_i^k)\geq1,K(t^{1}_i,t_i^k)=0|\mathcal{F}_{t^{1}_i})\\
		=&ne^{-\lambda_{n,m}}h(0,0)+\frac{(1-e^{-\lambda_{n,m}})(1-e^{-\lambda_{n,m}(k-1)})}{m\lambda_{n,m}}\int_{|x|>\beta_n}h(x,0)F(dx)+\\
		&e^{-\lambda_{n,m}(k-1)}\frac{1-e^{-\lambda_{n,m}}}{m\lambda_{n,m}}\int_{|x|>\beta_n}h(x,ff'(z)(k-1)x+(m-k)G(z,x))F(dx).
	\end{align*}
	By the same arguments and specific choices of function $h$ as above, we easily verify the three elements in  Lemma \ref{lem:cv}. In this case,  for any fixed $k\in\{1,\dots,m\}$, the obtained limit pair does not have independent marginals since its Lévy measure $K^k$ is not supported on the union of the coordinate axes.	
	Finally,  the vector $((\sum_{i=1}^n{y'}_{i,k}^n(2),\sum_{i=1}^n\tilde{z}_{i,k}^n(z)),k\in\{1,\dots,m\})$ is convergent in law to $(({Y'}_1^{k},V_1^{k}(z)),k\in\{1,\dots,m\})$ and the sequence $(\overline{Y'}^{n}_1(2),\sum_{i=1}^n\sum_{k=1}^m\tilde{z}_{i,k}^n(z))$ weakly converges to $(Y_1-d,V_1(z))$ where $d=b-\int_{|x|\leq1}xF(dx)$ and by independence $	V_1(z)=\sum_{k=1}^mV^{k}_1(z)$ with Lévy measure $	K(h)= \sum_{k=1}^m\frac{1}{m}\int h(ff'(z)(k-1)x+(m-k)G(z,x)) F(dx)$, the drift part equal to $K(x\1_{|x|\le1})$ and no Gaussian part (see (ii) and (iii) right above). Since its Lévy measure can also be rewritten as    $K(h)= \int_{\R}\int_{0}^{1} h(ff'(z)\lfloor mu\rfloor x+(m-1-\lfloor mu\rfloor)G(z,x)) F(dx)du$, 
similarly to Jacod \cite[(5.23)]{a}, a possible representation of the limit process is given by
$$V_1(z)=\sum\limits_{\substack{n:R_n\le1}}\left(ff'(z)\lfloor m\Upsilon_n\rfloor \Delta Y_{R_n}+(m-1-\lfloor m\Upsilon_n\rfloor)(f(z+\Delta Y_{R_n}f(z))-f(z))\right)
$$
	where $(R_n)_{n\geq1}$ denotes an enumeration of the jump times of $Y$ (or of $X$) and $(\Upsilon_n)_{n\ge1}$ is a sequence of i.i.d. variables, uniform on $[0,1]$ and independent of $Y$. It is worth to note that this sum is of finite variation.
Now, as said at the beginning, we go back to consider the convergence related to our original term $\tilde{z}_{i,k}^n$ defined in \eqref{eq:z} where $z=X_{t_i^1}^n$ is no longer fixed. As $z\mapsto\tilde{z}_{i,k'}^n(z)$ is continous, by following step by step the proof's arguments of  \cite[Theorem 1.2(d)]{a}, we obtain $(\overline{Y'}^{n}(2),\sum_{i=1}^{[n.]}\sum_{k=1}^m\tilde{z}_{i,k'}^n)$ converges in law to $(Y_1-d,V_1)$ where 
\begin{multline*}V_1=\hspace{-0.3cm}\sum\limits_{\substack{n:R_n\le1}}\hspace{-0.2cm}[ff'(X_{R_n-})\lfloor m\Upsilon_n\rfloor \Delta Y_{R_n}+(m-1-\lfloor m\Upsilon_n\rfloor)(f(X_{R_n-}+\Delta Y_{R_n}f(X_{R_n-}))-f(X_{R_n-}))].
\end{multline*}
This completes the proof.  
\end{proof}

\subsection{Asymptotic behavior of the couple $(\overline{Y}^{n},u_{n,m}Z^{n,m})$ for case \ref{2} and \ref{4}.} For the first component $\overline{Y}^n$, we use the same decomposition given by the relation \eqref{Y}. For the second one, we consider the formula of $u_{n,m}Z^{n,m}$ given in \eqref{Znm}. In these cases, the second term from this formula does not contribute to the limit. Therefore, we need only to give the analysis for its first term and by using the classical decomposition \eqref{eq:3.7} of $Y$, we have 
	\begin{align}\label{eq:decomp}
		&u_{n,m}\int_0^{\eta_n(t)}ff'(X_{\eta_n(s-)}^n)(Y_{\eta_{nm}(s-)}-Y_{\eta_n(s-)})dY_s=\sum_{i=1}^4\Gamma^n_t(i),
	\end{align}
	where
	\begin{align*}\left\{\begin{array}{rl}
		\Gamma_t^n(1)=&u_{n,m}\int_0^{\eta_n(t)}ff'(X_{\eta_n(s-)}^n)(N^{\beta_{n}}_{\eta_{nm}(s-)}-N^{\beta_{n}}_{\eta_n(s-)})dN^{\beta_{n}}_s,\\
		\Gamma_t^n(2)=&u_{n,m}\int_0^{\eta_n(t)}ff'(X_{\eta_n(s-)}^n)(A^{\beta_{n}}_{\eta_{nm}(s-)}-A^{\beta_{n}}_{\eta_n(s-)})dA^{\beta_{n}}_s\\&+u_{n,m}\int_0^{\eta_n(t)}ff'(X_{\eta_n(s-)}^n)(N^{\beta_{n}}_{\eta_{nm}(s-)}-N^{\beta_{n}}_{\eta_n(s-)})dA^{\beta_{n}}_s\\
		&+u_{n,m}\int_0^{\eta_n(t)}ff'(X_{\eta_n(s-)}^n)(A^{\beta_{n}}_{\eta_{nm}(s-)}-A^{\beta_{n}}_{\eta_n(s-)})dN^{\beta_{n}}_s,\\
		\Gamma_t^n(3)=&u_{n,m}\int_0^{\eta_n(t)}ff'(X_{\eta_n(s-)}^n)(M^{\beta_{n}}_{\eta_{nm}(s-)}-M^{\beta_{n}}_{\eta_n(s-)})dY_s,\\
		\Gamma_t^n(4)=&u_{n,m}\int_0^{\eta_n(t)}ff'(X_{\eta_n(s-)}^n)(A^{\beta_{n}}_{\eta_{nm}(s-)}-A^{\beta_{n}}_{\eta_n(s-)})dM^{\beta_{n}}_s\\&+u_{n,m}\int_0^{\eta_n(t)}ff'(X_{\eta_n(s-)}^n)(N^{\beta_{n}}_{\eta_{nm}(s-)}-N^{\beta_{n}}_{\eta_n(s-)})dM^{\beta_{n}}_s.		\end{array}\right.
	\end{align*}
The three last terms in the above decomposition do not contribute on the limit. Then we only have to study  $\Gamma_t^n(1)$ and we seperate it into 2 terms: the first term that will be the essential term of the limit contains only the first jumps and the drifts and the second term that will be sort in the rest terms contains all the other jumps. More precisely, we rewrite
	\begin{align*}
	\Gamma_t^n(1)&=u_{n,m}\sum_{i=1}^{[nt]}\sum_{k=2}^m\int_{I(nm,i,k)}ff'(X^n_{t_i^1})(N^{\beta_{n}}_{t_i^k}-N^{\beta_{n}}_{t^{1}_i})dN^{\beta_{n}}_s\\&=u_{n,m}\sum_{i=1}^{[nt]}\sum_{k=2}^mff'(X^n_{t_i^1})\sum_{j=1}^{k-1}(N^{\beta_{n}}_{t^{j+1}_i}-N^{\beta_{n}}_{t^{j}_i})(N^{\beta_{n}}_{t^{k+1}_i}-N^{\beta_{n}}_{t_i^k})={\Gamma}_t^{n}(1,1)+{\Gamma}_t^{n}(1,2),
	\end{align*}
	where 
	\begin{align}\label{eq:gamma1}
	{\Gamma}_t^{n}(1,1)=&u_{n,m}\sum_{i=1}^{[nt]}\sum_{k=2}^mff'(X^n_{t_i^1})\Delta Y_{T^{\beta_n}_1(t_i^k)}\1_{\{K(t_i^k)\geq1\}}\sum_{j=1}^{k-1}\Delta Y_{T^{\beta_n}_1(t^{j}_i)}\1_{\{K(t^{j}_i)\geq1\}},\nonumber\\
	{\Gamma}_t^{n}(1,2)=&u_{n,m}\sum_{i=1}^{[nt]}\sum_{k=2}^mff'(X^n_{t_i^1})\Delta Y_{T^{\beta_n}_1(t_i^k)}\1_{\{K(t_i^k)\geq1\}}\sum_{j=1}^{k-1}\sum_{h=2}^{K(t^{j}_i)}\Delta Y_{T^{\beta_n}_h(t^{j}_i)}\\
		&+u_{n,m}\sum_{i=1}^{[nt]}\sum_{k=2}^mff'(X^n_{t_i^1})\sum_{h=2}^{K(t_i^k)}\Delta Y_{T^{\beta_n}_h(t_i^k)}\sum_{j=1}^{k-1}\Delta Y_{T^{\beta_n}_1(t^{j}_i)}\1_{\{K(t^{j}_i)\geq1\}}\nonumber\\&+u_{n,m}\sum_{i=1}^{[nt]}\sum_{k=2}^mff'(X^n_{t_i^1})\sum_{h=2}^{K(t_i^k)}\Delta Y_{T^{\beta_n}_h(t_i^k)}\sum_{j=1}^{k-1}\sum_{h=2}^{K(t^{j}_i)}\Delta Y_{T^{\beta_n}_h(t^{j}_i)}.\nonumber
	\end{align}
In this case, $u_{n,m}Z^{n,m}_t=\mathcal{M}^{n,m}_t+\mathcal{R}^{n,m}_t,$ with 
\begin{align}\label{eq:MR2}
{\mathcal{M}}^{n,m}_{t}={\Gamma}_t^{n}(1,1)\quad \textrm{and}\quad\mathcal{R}^{n,m}_t=&{\Gamma}_t^{n}(1,2)+\sum_{i=2}^5\Gamma^n_t(i),
\end{align}
where $\Gamma^n_t(5)=u_{n,m}\int_0^{\eta_n(t)}k(X_{\eta_n(s-)}^n,Y_{\eta_{nm}(s-)}-Y_{\eta_n(s-)})(Y_{\eta_{nm}(s-)}-Y_{\eta_n(s-)})^2dY_s$.
The proof of the following lemma is postponed to Appendix \ref{app:A}.
\begin{lemma}\label{lem:rest2}
	For the cases \ref{2} and \ref{4},  we have as $n\rightarrow\infty$ the sequences $(\overline{Y}^n(1))_{n\ge0}$ and $(\mathcal{R}^{n,m})_{n\ge0}$  converge uniformly in probability to $0$.
\end{lemma} 
\begin{lemma}\label{lem:tight2}
	For the cases \ref{2} and \ref{4},   the sequences  $(\overline{Y}^n(2))_{n\ge0}$ and $({\mathcal{M}}^{n,m})_{n\ge0}$ are tight.  
\end{lemma}
\begin{proof}
	First,  we consider  $\overline{Y}^n_t(2)$ given by \eqref{Y}. In these cases \ref{2} and \ref{4}, thanks to assumption \eqref{h3}, we have $d'_n=0$ and $|d_n|\le C$. Then from \eqref{eq:y2a} and \eqref{eq:y2b}, $y^n_i(2)$ satisfies \eqref{eq:2.8} ensuring the tightness of $(\overline{Y}^n(2))_{n\ge0}$ from the second part of Lemma \ref{lem:tightrest}.
		 Now, we rewrite that ${\Gamma}_t^{n}(1,1)=\sum_{k=1}^m\sum_{j=1}^{k-1}\sum_{i=1}^{[nt]}\zeta_{i,k,j}^n$ with
	\begin{displaymath}
		\zeta_{i,k,j}^n=u_{n,m}ff'(X_{t_i^1}^n)\Delta Y_{T^{\beta_n}_1(t_i^k)}\1_{\{K(t_i^k)\geq1\}}\Delta Y_{T^{\beta_n}_1(t^{j}_i)}\1_{\{K(t^{j}_i)\geq1\}}.
	\end{displaymath}
	For $k\in\{1,\dots,m\}$ and $j\in\{1,\dots,k-1\}$ fixed, by using \eqref{h3}, Property \ref{I}, $ff'$ is bounded,  the inequality $1-e^{-\lambda_{n,m}}\le \lambda_{n,m}$, $c(\beta)=\int_{|x|\le\beta}x^2F(dx)$ where $c(\beta)\le C\beta^{2-\alpha}$ (see \eqref{eq:3.3}) and $\theta(\beta)\le C\beta^{-\alpha}$ (see \eqref{h1}), we get  
	\begin{align*}
	\left\{\begin{array}{l}
		\mathbb{E}(\zeta_{i,k,j}^n\1_{|\zeta_{i,k,j}^n|\le1}|\mathcal{F}_{t^{1}_i})=ff'(X_{t_i^1}^n) u_{n,m}\frac{(1-e^{-\lambda_{n,m}})^2}{n^2m^2\lambda_{n,m}^2}\int_{|x|>\beta_n}\int_{\frac{1}{|ff'(X_{t_i^1}^n)| u_{n,m}|x|}\ge|y|>\beta_n}xyF(dx)F(dy)=0,\\
		\mathbb{E}(|\zeta_{i,k,j}^n|^2\1_{|\zeta_{i,k,j}^n|\le1}|\mathcal{F}_{t^{1}_i})\le Cu_{n,m}^2\frac{(1-e^{-\lambda_{n,m}})^2}{n^2\lambda_{n,m}^2}\int_{|x|>\beta_n}x^2c(\frac{1}{|ff'(X_{t_i^1}^n)|u_{n,m}|x|})F(dx)\le C\frac{u_{n,m}^{\alpha}\rho_n}{n^2},\\
	\mathbb{P}(|\zeta_{i,k,j}^n|>y|\mathcal{F}_{t^{1}_i})\le C\frac{(1-e^{-\lambda_{n,m}})^2}{n^2\lambda_{n,m}^2}\int_{|x|>\beta_n}\theta(\frac{y}{|ff'(X_{t_i^1}^n)|u_{n,m}|x|})F(dx)\le C\frac{u_{n,m}^{\alpha}\rho_n}{n^2y^{\alpha}} ,\quad \forall y>1\end{array}\right..
	\end{align*}
	Then, we conclude the tightness of $({\mathcal{M}}^{n,m})_{n\ge0}$ by  $\rho_n\le C\log{1/\beta_{n}}$ (see \eqref{eq:3.13}), the choice $u_{n,m}=\left[\frac{nm}{(m-1)\log{n}}\right]^{1/\alpha}$ with $\alpha\le1$, criteria \eqref{eq:2.9} and Lemma \ref{lem:tightrest}.
	\end{proof}
	
	\begin{theorem}\label{thm:noise24}
		For cases \ref{2} and \ref{4},  we have
		$$(\overline{Y}^{n},{\mathcal{M}}^{n,m})\stackrel{\mathcal{L}}{\longrightarrow} (Y,Z),$$ where $Z$ is the limit process given in $\eqref{eq:1.10}$.
	\end{theorem}
	\begin{proof}
		Let us first introduce ${\Gamma'}_t^{n}(1,1)=\sum_{k=1}^m\sum_{j=1}^{k-1}\sum_{i=1}^{[nt]}\zeta_{i,j,k}^n(1)$, where for any $1\leq k\leq m$ and $1\leq j\leq k-1$, 
		\begin{displaymath}
		\zeta_{i,j,k}^n(1)=u_{n,m}\Delta Y_{T^{\beta_n}_1(t_i^k)}\1_{\{K(t_i^k)\geq1\}}\Delta Y_{T^{\beta_n}_1(t^{j}_i)}\1_{\{K(t^{j}_i)\geq1\}}.
		\end{displaymath} 
	Since $ff'$ is  Lipschitz-continuous, by virtue of Lemma \ref{lem:cv1tot}, in order to prove the convergence in law of $(\overline{Y}^{n},{\mathcal{M}}^{n,m})$ it suffices to prove that $(\overline{Y}^{n}_1(2),{\Gamma'}_1^{n}(1,1))$ converges in law to $(Y_1,V_1)$ where $V$ is a L\'evy process independent of $Y$ and characterized by $\eqref{eq:1.10}$. Now, let us denote  $$\overline{Y'}^{n}_t(2)=\sum_{i=1}^{[nt]}\sum_{k=1}^m{y'}_{i,k}^n(2), \quad\textrm{where}\quad {y'}_{i,k}^n(2)=\Delta Y_{T^{\beta_n}_1(t_i^k)}\1_{\{K(t_i^k)\geq 1\}}.$$ From the hypothesis \eqref{h3},  $d_n=b$ in this case, then $\overline{Y}^{n}_1(2)-\overline{Y'}^{n}_1(2)=\sum_{i=1}^{n}\sum_{k=1}^m\frac{d_n}{nm}=b$ which allows us to prove instead the couple $(\overline{Y'}^{n}_1(2),{\Gamma'}_1^{n}(1,1))$ converges to the limit process $(Y_1-b,V_1)$ with no drift and no continuous martingale part.		
			To do so, we choose the strategy of  proving the convergence of the $\mathbb{R}^{\frac{m(m-1)}{2}+m}$-vector $$\left(\left(\sum_{i=1}^n{y'}_{i,k}^n(2)\right)_{k\in\{1,\hdots,m\}},\left(\sum_{i=1}^n\zeta_{i,j',k'}^n(1)\right)_{\substack{k',j'\in\{1,\hdots,m\}\\j'<k'}}\right)$$ to the limit vector whose components are pairwise independent L\'evy processes. 
		First, since this $\mathbb{R}^{\frac{m(m-1)}{2}+m}$-vector  is tight thanks to Lemma \ref{lem:tight2}, it is enough to prove that every weakly convergent subsequence has the same limit. In what follows, we omit the notation for the subsequence for more readability. For the independence of the components of the limit vector, by using Ex.12.8-12.10 in \cite{d}, we only need to prove the independence between the limit marginals  of the pairs $((\sum_{i=1}^{n}{y'}_{i,k}^n(2),\sum_{i=1}^{n}{y'}_{i,k'}^n(2));{k\neq k'})$, $(\sum_{i=1}^n{y'}_{i,k}^n(2),\sum_{i=1}^n\zeta_{i,j',k'}^n(1))$ and $((\sum_{i=1}^n\zeta_{i,j,k}^n(1),\sum_{i=1}^n\zeta_{i,j',k'}^n(1));(k,j)\neq (k',j'))$, for any $k,k',j,j'\in\{1,\dots,m\}$, and $j<k$, $j'<k'$, then we obtain the Fourrier transform of the limit vector.
		 	
		$\bullet$ First, by the independent structure of  the subsequence marginals $\sum_{i=1}^{n}({y'}_{i,k}^n(2),{y'}_{i,k'}^n(2))$ for any $k,k'\in\{1,\dots,m\}, k\not =k'$, it is obvious that the limit marginals are i.i.d. 
		 	
		$\bullet$ Second, for fixed $k,k',j'\in\{1,\dots,m\}$ and $1\le j'\le k'-1$ we consider the convergence of the triangular array  $\sum_{i=1}^n({y'}_{i,k}^n(2),\zeta_{i,j',k'}^n(1))$ whose generic terms $(({y'}_{i,k}^n(2),\zeta_{i,j',k'}^n(1)))_{1\leq i\leq n}$ are i.i.d. When $k,k'$ and  $j'$ are different, the independence between the marginals is obvious.  By symmetry of the roles played by $j'$ and $k'$ in $\zeta_{i,j',k'}^n(1)$, it is sufficient to consider the case when $k=k'$ and $j'\neq k'$.  Note that the law of $({y'}_{i,k}^n(2),\zeta_{i,j',k}^n(1))$ does not depend on parameters $k$ and $j'$, then we denote it by $K^1_{n,m}$. We will prove that $nK^1_{n,m}(h)$ converges to $K^1(h)$ as $n$ tends to infinity for some suitable function $h$  where $$K^1(dx,dy)=\frac{1}{m}\delta_0(dy)F(dx)+\frac{\theta^2\alpha}{2m(m-1)}\delta_0(dx)\frac{1}{|y|^{1+\alpha}}dy$$ and $\delta_0$ is Dirac measure sitting at point $0$. Therefore, we have
		\begin{align*}
		nK^1_{n,m}(h)&=n\mathbb{E}(h(\Delta Y_{T^{\beta_n}_1(t_i^k)}\1_{\{K(t_i^k)\geq 1\}},u_{n,m}\Delta Y_{T^{\beta_n}_1(t_i^k)}\1_{\{K(t_i^k)\geq 1\}}\Delta Y_{T^{\beta_n}_1(t^{j'}_i)}\1_{\{K(t^{j'}_i)\geq 1\}})|\mathcal{F}_{t_i^1})	\\
		&=n\mathbb{E}(h(0,0);K(t^{k}_i)=0|\mathcal{F}_{t_i^1})+n\mathbb{E}(h(\Delta Y_{T^{\beta_n}_1(t_i^k)},0);K(t_i^k)\geq 1,K(t^{j'}_i)=0|\mathcal{F}_{t_i^1})\\
		&+n\mathbb{E}(h(\Delta Y_{T^{\beta_n}_1(t_i^k)},u_{n,m}\Delta Y_{T^{\beta_n}_1(t_i^k)}\Delta Y_{T^{\beta_n}_1(t^{j'}_i)});K(t_i^k)\geq 1,K(t^{j'}_i)\geq1|\mathcal{F}_{t_i^1})\\
		&=ne^{-\lambda_{n,m}}h(0,0)+\frac{e^{-\lambda_{n,m}}(1-e^{-\lambda_{n,m}})}{m\lambda_{n,m}}\int_{|x|>\beta_{n}}\hspace{-0.6cm}h(x,0)F(dx)\\&+\frac{(1-e^{-\lambda_{n,m}})^2}{n(m\lambda_{n,m})^2}\int_{|x|>\beta_{n}}\hspace{-0.6cm}F(dx)\int_{|y|>\beta_{n}}\hspace{-0.6cm}h(x,u_{n,m}xy)F(dy).
			\end{align*}
		Three following headings demonstrate three elements in  Lemma \ref{lem:cv}, each corresponding to some specific choices of function $h$.
	
	\paragraph{\textbf{Concerning assertion (i)}:} 
	We choose $h=h_{u,v}$ where $h_{u,v}(x,y)=\1_{\{|x|\ge u,|y|\ge v\}}$ for all $u,v\in\mathbb{R}_+$ such that $(u,v)\neq(0,0)$.
				\begin{enumerate}[a)]
	\item For $u>0$ and $v=0$, 
	\begin{align*}
	nK_{n,m}^1(h_{u,0})=
	\frac{e^{-\lambda_{n,m}}(1-e^{-\lambda_{n,m}})}{m\lambda_{n,m}}\int_{|x|>\beta_{n}}\hspace{-0.5cm}\1_{\{|x|\ge u\}}F(dx)+\frac{(1-e^{-\lambda_{n,m}})^2}{m\lambda_{n,m}}\int_{|x|>\beta_{n}}\hspace{-0.5cm}\1_{\{|x|\ge u\}}F(dx).
				\end{align*}
As soon as $u>\beta_n$, we have 			 
$
	\int_{|x|>\beta_{n}}\1_{\{|x|\ge u\}}F(dx)=\theta(u-),
$
where $\theta(u-)$ denotes the left limit at point $u$ of the decreasing and right-continuous function $\theta(.)$. 
Then we get $	nK_{n,m}^1(h_{u,0})\underset{n\rightarrow\infty}{\longrightarrow}K^1(h_{u,0})=\frac{\theta(u-)}{m}$.
\item 
For $u=0$ and $v>0$, we have
\begin{align*}
	nK_{n,m}^1(h_{0,v})=
	\frac{(1-e^{-\lambda_{n,m}})^2}{nm^2\lambda_{n,m}^2}\int_{|x|>\beta_{n}}\int_{|y|>\beta_{n}}\1_{\{|u_{n,m}xy|\ge v\}}F(dy)F(dx).
\end{align*}
Now, we denote the constant $v_m=v(\frac{m-1}{m})^{1/\alpha}$,  using $u_{n,m}=\left[\frac{mn}{(m-1)\log{n}}\right]^{1/{\alpha}}$, $\beta_{n}=\left(\frac{\log{n}}{n}\right)^{1/{\alpha}}$ and $u_{n,m}\beta_n=(\frac{m-1}{m})^{1/\alpha}$, then as soon as $\beta_n<v_m$, we have 
\begin{multline*}
	\frac{1}{nm^2}\int_{|x|>\beta_{n}}\int_{|y|>\beta_{n}}\1_{\{|u_{n,m}xy|\ge v\}}F(dx)F(dy)\\=\frac{1}{nm^2}\int_{|x|> \beta_n}\theta(\beta_{n}\vee \frac{v_m\beta_n}{|x|}) F(dx)
	=\frac{1}{nm^2}(\theta(\beta_n)\theta(v_m-)+\int_{\beta_{n}<|x|\le  v_m}\theta(\frac{v_m\beta_n}{|x|})F(dx)).
\end{multline*}
 By \eqref{h2}, the first term is equivalent to $\frac{\theta}{nm^2\beta_n^{\alpha}}\theta(v_m-)$ which converges to $0$. 
Considering the second term, let us denote $y_n=\frac{1}{nm^2}\int_{\beta_{n}<|x|\le  v_m}\theta(\frac{v_m\beta_n}{|x|})F(dx)$. Let $\varepsilon>0$, by \eqref{h2}, there exists $\varepsilon'\in(0,v_m)$ such that for $\beta\in(0,\varepsilon')$ we have $|\frac{\beta^{\alpha}\theta(\beta)}{\theta}-1|\le\varepsilon$. Then, we denote $y_n=y_n^{1,\varepsilon'}+y_n^{2,\varepsilon'},$ where $y_n^{1,\varepsilon'}=\frac{1}{nm^2}\int_{\frac{v_m\beta_{n}}{\varepsilon'}<|x|\le  v_m}\theta(\frac{v_m\beta_n}{|x|})F(dx)$ and $y_n^{2,\varepsilon'}=\frac{1}{nm^2}\int_{\beta_n<|x|\le  \frac{v_m\beta_{n}}{\varepsilon'}}\theta(\frac{v_m\beta_n}{|x|})F(dx)$. On the one hand, by the fact that $\theta(.)$ is decreasing and \eqref{h1}, we have 
$
	y_n^{2,\varepsilon'}=\frac{1}{nm^2}\int_{\beta_n<|x|\le  \frac{v_m\beta_{n}}{\varepsilon'}}\theta(\frac{v_m\beta_n}{|x|})F(dx)\le\frac{\theta(\beta_n)\theta(\varepsilon')}{nm^2}\underset{n\rightarrow\infty}{\longrightarrow}0.
$
On the other hand, if we denote ${y'}_n^{1,\varepsilon'}=\frac{\theta}{nm^2v_m^{\alpha}\beta_n^{\alpha}}\int_{\frac{v_m\beta_{n}}{\varepsilon'}<|x|\le  v_m}|x|^{\alpha}F(dx)$, thanks to $\rho_n\sim\alpha\theta\log{1/\beta_n}$ (see \eqref{eq:3.5}) we have ${y'}_n^{1,\varepsilon'}=\frac{\theta}{nm^2v_m^{\alpha}\beta_n^{\alpha}}(\rho(\frac{v_m\beta_{n}}{\varepsilon'})-\rho(v_m-))\sim\frac{\alpha\theta^2}{nm^2v_m^{\alpha}\beta_n^{\alpha}}\log{1/\beta_n}$.
From \eqref{h2}, we have ${y'}_n^{1,\varepsilon'}(1-\varepsilon)\le{y}_n^{1,\varepsilon'}\le{y'}_n^{1,\varepsilon'}(1+\varepsilon)$ and since $\varepsilon$ is arbitrarily small, ${y}_n^{1,\varepsilon'}\sim\frac{\alpha\theta^2}{nm^2v_m^{\alpha}\beta_n^{\alpha}}\log{1/\beta_n}$. Then we get $	nK_{n,m}^1(h_{0,v})\underset{n\rightarrow\infty}{\longrightarrow}K^1(h_{0,v})=\frac{\theta^2}{m(m-1)v^{\alpha}}$. In what follows, we will reused the obtained result
\begin{align}\label{aid1}
	\frac{(1-e^{-\lambda_{n,m}})^2}{nm^2\lambda_{n,m}^2}\int_{|x|>\beta_{n}}\int_{|y|>\beta_{n}}\1_{\{|u_{n,m}xy|\ge v\}}F(dy)F(dx)\underset{n\rightarrow\infty}{\longrightarrow}\frac{\theta^2}{m(m-1)v^{\alpha}}.
\end{align}

\item For $u>0$ and $v>0$, as soon as $u>\beta_n$, by the inequality $1-e^{-\lambda_{n,m}}\le\lambda_{n,m}$, we have 
\begin{multline*}
nK_{n,m}^1(h_{u,v})
=\frac{(1-e^{-\lambda_{n,m}})^2}{nm^2\lambda_{n,m}^2}\int_{|x|\ge u}\int_{|y|>\beta_{n}}\1_{\{|u_{n,m}xy|\ge v\}}F(dy)F(dx)
\le \frac{\theta(u-)\theta(\beta_n)}{nm^2}
\underset{n\rightarrow\infty}{\longrightarrow}0.
\end{multline*}
Then we get $	nK_{n,m}^1(h_{u,v})\underset{n\rightarrow\infty}{\longrightarrow}K^1(h_{u,v})=0$.
\end{enumerate}
\paragraph{\textbf{Concerning assertion (ii)}:} We choose $h={h'},{h''}$ where  ${h'}(x,y)=x\1_{\{x^2+y^2\le1\}}$ and ${h''}(x,y)=y\1_{\{x^2+y^2\le1\}}$.
 Since \eqref{h3} holds, the laws $K_{n,m}^1$ and $K^1$ are invariant under the map $(x,y)\mapsto(-x,y)$ and $(x,y)\mapsto(x,-y)$. Then in this case, we get immediately $nK_{n,m}^1(h)=K^1(h)=0.$ 
 
\paragraph{\textbf{Concerning assertion (iii)}:}
 Here we take $h={h}_{1},{h}_{2},{h}_{3}$ where ${h}_{1}(x,y)=x^2\1_{\{x^2+y^2\le1\}}$, ${h}_{2}(x,y)=xy\1_{\{x^2+y^2\le1\}}$ and ${h}_{3}(x,y)=y^2\1_{\{x^2+y^2\le1\}}$. First of all, as above in (ii), by hypothesis \eqref{h3}, we have $nK_{n,m}^1(h_2)=K^1(h_2)=0$. Now, we consider
\begin{multline*}
nK_{n,m}^1(h_1)
=\frac{e^{-\lambda_{n,m}}(1-e^{-\lambda_{n,m}})}{m\lambda_{n,m}}\int_{\beta_n<|x|\le1}x^2F(dx)\\+\frac{(1-e^{-\lambda_{n,m}})^2}{nm^2\lambda_{n,m}^2}\int_{1\ge|x|>\beta_{n}}x^2F(dx)\int_{|y|>\beta_{n}}\1_{\{|y|\leq \frac{\sqrt{x^{-2}-1}}{u_{n,m}}\}}F(dy)
\end{multline*}
and
\begin{align*}
nK_{n,m}^1(h_3)=\frac{(1-e^{-\lambda_{n,m}})^2}{nm^2\lambda_{n,m}^2}\int_{1\ge|x|>\beta_{n}}\int_{|y|>\beta_{n}}u_{n,m}^2x^2y^2\1_{\{|y|\leq \frac{\sqrt{x^{-2}-1}}{u_{n,m}}\}}F(dx)F(dy).
\end{align*}
Concerning the term $nK_{n,m}^1(h_1)$,  it is clear that the first term   converges to $\frac{1}{m}\int_{|x|\le1}x^2F(dx)$ as $n\rightarrow\infty$ and
as $\int_{\R}x^2F(dx)<\infty$, its second term is bounded by $\frac{C\theta(\beta_n)}{n}$ which goes to $0$ as $n\rightarrow\infty$. Therefore, we get $nK_{n,m}^1(h_1)\underset{n\rightarrow\infty}{\longrightarrow}nK^1(h_1)$. Concerning the term $nK_{n,m}^1(h_3)$, let $a_m=(\frac{m-1}{m})^{1/\alpha}$ and ${a'}_m=\frac{a_m}{\sqrt{a^2_m+1}}$, we have for $n$ large enough $\beta_n\le {a'}_m$ and
\begin{align*}
	nK_{n,m}^1(h_3)&=\frac{(1-e^{-\lambda_{n,m}})^2}{nm^2\lambda_{n,m}^2}\int_{1\ge|x|>{a'}_m}\int_{|y|>\beta_{n}}u_{n,m}^2x^2y^2\1_{\{|y|\leq \frac{\sqrt{1-x^2}a_m\beta_n}{|x|}\}}F(dx)F(dy)\\
	&+\frac{(1-e^{-\lambda_{n,m}})^2}{nm^2\lambda_{n,m}^2}\int_{{a'}_m\ge|x|>\beta_{n}}u_{n,m}^2x^2(c(\frac{\sqrt{1-x^2}a_m\beta_n}{|x|})-c_n)F(dx).
\end{align*}
Using $\lim_{n\rightarrow\infty}\frac{1-e^{\lambda_{n,m}}}{\lambda_{n,m}}=1$ and  $c_{n}\le C\beta_n^{2-\alpha}$, it is easy to check that the first term in the right-hand side is bounded by $C\frac{u_{n,m}^2c_n}{n}$ which converges to $0$ and the term $\frac{u_{n,m}^2}{nm^2}\int_{{a'}_m\ge|x|>\beta_{n}}x^2c_nF(dx)$ converges also to $0$. Hence, we have $
nK_{n,m}^1(h_3)\sim\frac{u_{n,m}^2}{nm^2}\int_{{a'}_m\ge|x|>\beta_{n}}x^2c(\frac{\sqrt{1-x^2}a_m\beta_n}{|x|})F(dx).
$
Now, using  $c(\beta)\sim\frac{\alpha\theta}{2-\alpha}\beta^{2-\alpha}$ for $\beta\rightarrow0$, then for $\varepsilon>0$, there exists $\varepsilon'\in(0,1)$ such that for $\beta\in(0,{a}_m\varepsilon')$ we have $|\frac{(2-\alpha)\beta^{\alpha-2}c(\beta)}{\alpha\theta}-1|\le\varepsilon$ and for $n$ large enough such that $\beta_n< {a'}_m\varepsilon'$,  $\frac{u_{n,m}^2}{nm^2}\int_{{a'}_m\ge|x|>\beta_{n}}x^2c(\frac{\sqrt{1-x^2}a_m\beta_n}{|x|})F(dx)=x_n+y_n$   where
\begin{align*}
x_n=&\frac{u_{n,m}^2}{nm^2}\int_{\beta_{n}<|x|\le\beta_{n}/\varepsilon'}x^2c(\frac{\sqrt{1-x^2}a_m\beta_n}{|x|})F(dx)\\
y_n=&\frac{u_{n,m}^2}{nm^2}\int_{{a'}_m\ge|x|>\beta_{n}/\varepsilon'}x^2c(\frac{\sqrt{1-x^2}a_m\beta_n}{|x|})F(dx).
\end{align*}
On the one hand, for $x_n$, as $c(.)$ is increasing, we use that $c(\frac{\sqrt{1-x^2}a_m\beta_n}{|x|})$ is bounded to deduce an upper bound equal to $\frac{Cu_{n,m}^2c(\beta_{n}/\varepsilon')}{n}$ which converges to $0$. On the other hand, for $y_n$, if we denote ${y'}_n=\frac{u_{n,m}^2a_m^{2-\alpha}\beta_n^{2-\alpha}}{nm^2}\int_{{a'}_m\ge|x|>\beta_{n}/\varepsilon'}|x|^{\alpha}(1-x^2)^{\frac{2-\alpha}{2}}F(dx)$, we have $(1-\varepsilon){y'}_n\le y_n\le(1+\varepsilon){y'}_n$ which gives ${y}_n\sim{y'}_n$ since $\varepsilon$ is arbitrarily small. In what follows, we rewrite ${y'}_n={y'}_n^1+{y'}_n^2$ where
\begin{align*}
	{y'}_n^1=&\frac{\alpha\theta u_{n,m}^2a_m^{2-\alpha}\beta_n^{2-\alpha}}{(2-\alpha)nm^2}\int_{{a'}_m\ge|x|>\beta_{n}/\varepsilon'}|x|^{\alpha}F(dx),\\{y'}_n^2=&\frac{\alpha\theta u_{n,m}^2a_m^{2-\alpha}\beta_n^{2-\alpha}}{(2-\alpha)nm^2}\int_{{a'}_m\ge|x|>\beta_{n}/\varepsilon'}|x|^{\alpha}[(1-x^2)^{\frac{2-\alpha}{2}}-1]F(dx).
\end{align*}	
Then by $\rho_n{\sim}\alpha\theta\log{1/\beta_n}$  (see \eqref{eq:3.5}), $u_{n,m}=\left[\frac{mn}{(m-1)\log{n}}\right]^{1/{\alpha}}$ and  $\beta_{n}=\left(\frac{\log{n}}{n}\right)^{1/{\alpha}}$, we get that $${y'}_n^1\sim\frac{\alpha\theta u_{n,m}^2a_m^{2-\alpha}\beta_n^{2-\alpha}}{(2-\alpha)nm^2}(\rho(\beta_{n}/\varepsilon')-\rho({a'}_m))\sim\frac{\alpha^2\theta^2 u_{n,m}^2a_m^{2-\alpha}\beta_n^{2-\alpha}}{(2-\alpha)nm^2}\log{(\varepsilon'/\beta_{n})}\underset{n\rightarrow\infty}{\longrightarrow}\frac{\alpha\theta^2}{m(m-1)(2-\alpha)}$$
and that ${y'}_n^2\le [(1-{\beta_n}^2)^{\frac{2-\alpha}{2}}-1]{y'}_n^1$ which converges to $0$. Therefore, clearly, $nK_{n,m}^1(h_3)\underset{n\rightarrow\infty}{\longrightarrow}K^1(h_3)=\frac{\alpha\theta^2}{m(m-1)(2-\alpha)}$. Thanks to this proof, in particular we have proved that
\begin{align}\label{aid2}
\frac{(1-e^{-\lambda_{n,m}})^2}{nm^2\lambda_{n,m}^2}\int_{1\ge|x|>\beta_{n}}\int_{|y|>\beta_{n}}u_{n,m}^2x^2y^2\1_{\{|y|\leq \frac{\sqrt{x^{-2}-1}}{u_{n,m}}\}}F(dx)F(dy)\underset{n\rightarrow\infty}{\longrightarrow}\frac{\alpha\theta^2}{m(m-1)(2-\alpha)}.
\end{align}	
  In conclusion, the obtained limit pair has independent  marginals since it has no Gaussian part and its L\'evy measure $K^{1}$ is supported on the union of the coordinate axes (see e.g. \cite[Ex.12.8 ]{d}). 		
  	
		$\bullet$
		Third, for fixed $k,k',j,j'\in\{1,\dots,m\}$, $(j,k)\neq(j',k')$ and $1\le j\le k-1$, $1\le j'\le k'-1$ we consider the convergence of  $\sum_{i=1}^n(\zeta_{i,j,k}^n(1),\zeta_{i,j',k'}^n(1))$ whose variables $(\zeta_{i,j,k}^n(1),\zeta_{i,j',k'}^n(1))$, ${1\leq i\leq n}$ are i.i.d. When $k$, $k'$, $j$ and $j'$ are different, we have straightforward the independence between the marginals of the limit pair. Otherwise, by symmetry of the roles played by $j$ and $k$ and the roles played by  $j'$ and $k'$, it is enough to consider the particular case where $k=k'$ and $j\neq j'$. Note that the law of $(\zeta_{i,j,k}^n(1),\zeta_{i,j',k}^n(1))$ does not depend on parameters $k$ and $j'$,  we denote its law  by $K^2_{n,m}$. We will prove that $nK_{n,m}^2\underset{n\rightarrow\infty}{\longrightarrow}K^2$, where$$K^2(dx,dy)=\frac{\theta^2\alpha}{2m(m-1)}\delta_0(dy)\frac{1}{|x|^{1+\alpha}}dx+\frac{\theta^2\alpha}{2m(m-1)}\delta_0(dx)\frac{1}{|y|^{1+\alpha}}dy$$ and $\delta_0$ is Dirac measure sitting at point $0$. 	Therefore, we have
	\begin{align*}
	&nK^2_{n,m}(h)=n\mathbb{E}((\zeta_{i,j,k}^n(1),\zeta_{i,j',k}^n(1))|\mathcal{F}_{t_i^1})	\\
&=n\mathbb{E}(h(0,0); K(t_i^k)=0|\mathcal{F}_{t_i^1})+n\mathbb{E}(h(0,0);K(t_i^k)\geq 1,K(t^{j}_i)=0,K(t^{j'}_i)=0|\mathcal{F}_{t_i^1})\\
&+n\mathbb{E}(h(0,u_{n,m}\Delta Y_{T^{\beta_n}_1(t_i^k)}\Delta Y_{T^{\beta_n}_1(t^{j'}_i)});K(t_i^k)\geq 1,K(t^{j}_i)=0,K(t^{j'}_i)\geq 1|\mathcal{F}_{t_i^1})\\
&+n\mathbb{E}(h(u_{n,m}\Delta Y_{T^{\beta_n}_1(t_i^k)}\Delta Y_{T^{\beta_n}_1(t^{j}_i)},0);K(t_i^k)\geq 1,K(t^{j}_i)\geq 1,K(t^{j'}_i)=0|\mathcal{F}_{t_i^1})\\
&+n\mathbb{E}(h(u_{n,m}\Delta Y_{T^{\beta_n}_1(t_i^k)}\Delta Y_{T^{\beta_n}_1(t^{j}_i)},u_{n,m}\Delta Y_{T^{\beta_n}_1(t_i^k)}\Delta Y_{T^{\beta_n}_1(t^{j'}_i)});K(t_i^k)\geq 1,K(t^{j}_i)\geq1,K(t^{j'}_i)\geq 1|\mathcal{F}_{t_i^1})\\
&=ne^{-\lambda_{n,m}}(1+e^{-\lambda_{n,m}}(1-e^{-\lambda_{n,m}}))h(0,0)\\&+\frac{e^{-\lambda_{n,m}}(1-e^{-\lambda_{n,m}})^2}{nm^2\lambda_{n,m}^2}\int_{|x|>\beta_{n}}\int_{|y|>\beta_{n}}h(0,u_{n,m}xy)F(dx)F(dy)\\
&+\frac{e^{-\lambda_{n,m}}(1-e^{-\lambda_{n,m}})^2}{nm^2\lambda_{n,m}^2}\int_{|x|>\beta_{n}}\int_{|y|>\beta_{n}}h(u_{n,m}xy,0)F(dx)F(dy)\\
&+\frac{(1-e^{-\lambda_{n,m}})^3}{n^2m^3\lambda_{n,m}^3}\int_{|x|>\beta_{n}}\int_{|y|>\beta_{n}}\int_{|z|>\beta_{n}}h(u_{n,m}xy,u_{n,m}xz)F(dx)F(dy)F(dz).
\end{align*}
Now, we verify the three elements in  Lemma \ref{lem:cv} with suitable choices of function $h$.

\paragraph{\textbf{Concerning assertion (i)}:}
We choose $h=h_{u,v}$ where $h_{u,v}(x,y)=\1_{\{|x|\geq u,|y|\geq v\}}$ for all $u,v\in\mathbb{R}_+$ such that $(u,v)\neq(0,0)$.
\begin{enumerate}[a)]
	\item For $u>0$ and $v=0$, 
		\begin{multline*}
	nK_{n,m}^2(h_{u,0})
	=\frac{e^{-\lambda_{n,m}}(1-e^{-\lambda_{n,m}})^2}{nm^2\lambda_{n,m}^2}\int_{|x|>\beta_{n}}\int_{|y|>\beta_{n}}\1_{\{|u_{n,m}xy|\geq u\}}F(dx)F(dy)\\
+\frac{(1-e^{-\lambda_{n,m}})^3}{nm^2\lambda_{n,m}^2}\int_{|x|>\beta_{n}}\int_{|y|>\beta_{n}}\1_{\{|u_{n,m}xy|\geq u\}}F(dx)F(dy).
	\end{multline*}
By \eqref{aid1}, the first term contributes at the limit and the second term vanishes when $n\rightarrow\infty$. Then we get $	nK_{n,m}^2(h_{u,0})\underset{n\rightarrow\infty}{\longrightarrow}K^2(h_{u,0})=\frac{\theta^2}{m(m-1)u^{\alpha}}$.
\item For $u=0$ and $v>0$, we have $nK_{n,m}^2(h_{0,v})=nK_{n,m}^2(h_{v,0})$. Then, by a) we get that $	nK_{n,m}^2(h_{0,v})\underset{n\rightarrow\infty}{\longrightarrow}K^2(h_{0,v})=\frac{\theta^2}{m(m-1)v^{\alpha}}$.
\item For $u>0$ and $v>0$, 
	\begin{align*}
nK_{n,m}^2(h_{u,v})=\frac{(1-e^{-\lambda_{n,m}})^3}{nm^2\lambda_{n,m}^2}\int_{|x|>\beta_{n}}\int_{|y|>\beta_{n}}\int_{|z|>\beta_{n}}\hspace{-0.5cm}\1_{\{|u_{n,m}xy|\geq u,|u_{n,m}xz|\geq v\}}F(dx)F(dy)F(dz)\\
\le\frac{(1-e^{-\lambda_{n,m}})^3}{nm^2\lambda_{n,m}^2}\int_{|x|>\beta_{n}}\int_{|y|>\beta_{n}}\1_{\{|u_{n,m}xy|\geq u\}}F(dx)F(dy).
\end{align*}
Again, by \eqref{aid1}, we have $nK_{n,m}^2(h_{u,v})
\underset{n\rightarrow\infty}{\longrightarrow}K^2(h_{u,v})
=0$.
	\end{enumerate}
\paragraph{\textbf{Concerning assertion (ii)}:} We choose $h={h'},{h''}$ where  ${h'}(x,y)=x\1_{\{x^2+y^2\le1\}}$ and ${h''}(x,y)=y\1_{\{x^2+y^2\le1\}}$. Similarly as above, by \eqref{h3}, $nK_{n,m}^2(h)=K^2(h)=0$.	

\paragraph{\textbf{Concerning assertion (iii)}:} We choose $h={h}_{1},{h}_{2},{h}_{3}$ where ${h}_{1}(x,y)=x^2\1_{\{x^2+y^2\le1\}}$, ${h}_{2}(x,y)=xy\1_{\{x^2+y^2\le1\}}$ and ${h}_{3}(x,y)=y^2\1_{\{x^2+y^2\le1\}}$. First of all, as above in (ii), by hypothesis \eqref{h3}, we have $nK_{n,m}^2(h_2)=K^2(h_2)=0$. Now, by symmetry, we have  
\begin{multline*}
	nK_{n,m}^1(h_1)=nK_{n,m}^1(h_3)
	=\frac{e^{-\lambda_{n,m}}(1-e^{-\lambda_{n,m}})^2}{nm^2\lambda_{n,m}^2}\int_{|x|>\beta_n}\int_{|y|>\beta_n}\hspace{-0.6cm}u_{n,m}^2x^2y^2\1_{\{|u_{n,m}xy|\le1\}}F(dx)F(dy)\\
	\\+\frac{(1-e^{-\lambda_{n,m}})^3}{n^2m^3\lambda_{n,m}^3}\int_{|x|>\beta_{n}}\int_{|y|>\beta_{n}}\int_{|z|>\beta_{n}}u_{n,m}^2x^2y^2\1_{\{|xy|^2\leq 1/u_{n,m}^2-|xz|^2\}}F(dx)F(dy)F(dz).
\end{multline*}
By similar estimations as in  \eqref{aid2}, we can easily deduce
\begin{align*}
	\frac{1}{nm^2}\int_{|x|>\beta_{n}}\int_{|y|>\beta_{n}}u_{n,m}^2x^2y^2\1_{\{u_{n,m}|xy|\leq 1\}}F(dx)F(dy)\underset{n\rightarrow\infty}{\longrightarrow}\frac{\alpha\theta^2}{m(m-1)(2-\alpha)}
\end{align*}
which gives the limit of the first term  and that the second term vanishes as it is bounded by
$
\frac{(1-e^{-\lambda_{n,m}})^3}{nm^2\lambda_{n,m}^2}\int_{|x|>\beta_{n}}\int_{|y|>\beta_{n}}u_{n,m}^2x^2y^2\1_{\{u_{n,m}|xy|\leq 1\}}F(dx)F(dy)
$ converging to $0$ as $n\rightarrow\infty$. Therefore, $nK_{n,m}^2(h_1)\underset{n\rightarrow\infty}{\longrightarrow}K^2(h_1)$ and $nK_{n,m}^2(h_3)\underset{n\rightarrow\infty}{\longrightarrow}K^2(h_3)$. In conclusion, the obtained limit pair has i.i.d.  marginals since it has no Gaussian part and its L\'evy measure $K^{2}$ is supported on the union of the coordinate axes (see e.g. Ex.12.8 in \cite{d}). 

Overall, by the pairwise independence proven above, we can realize the limit of the vector $(\overline{Y'}^{n}_1(2),{\Gamma'}_1^{n}(1,1))$  as some vector $(\sum_{k=1}^m\overline{Y'}^k_1,\sum_{k=2}^m\sum_{j=1}^{k-1}V^{j,k}_1)$ Lévy process with no drift, no Gaussian part and  where
\begin{align*}
\mathbb{E}(e^{i(u\sum_{k=1}^m\overline{Y'}^k_1+v\sum_{k=2}^m\sum_{j=1}^{k-1}V^{j,k}_1)})=[\mathbb{E}(e^{iu\overline{Y'}^1_1})]^m\times[\mathbb{E}(e^{ivV^{1,2}_1})]^{\frac{m(m-1)}{2}}
\end{align*}  
which is equal to 
$$\exp{\left(\int F(dx)(e^{iux}-1-iux\1_{\{|x|\le1\}})+\frac{\alpha\theta^2}{2m(m-1)}\int\frac{1}{|x|^{1+\alpha}}(e^{ivx}-1-ivx\1_{\{|x|\le1\}})dx\right)}.$$
This completes the proof.
\end{proof}

\subsection{Asymptotic behavior of the couple $(\overline{Y}^{n},u_{n,m}Z^{n,m})$ for case \ref{3}.}
 For the first component $\overline{Y}^n$, we use the same decomposition given by the relation \eqref{Y}. For the second one, we consider the formula of $u_{n,m}Z^{n,m}$ given in \eqref{Znm}. In this case, the second term from this formula does not contribute to the limit. Therefore, we need only to give the analysis for its first term. To do so, we consider the same decomposition given in \eqref{eq:decomp}.
The two last terms in this decomposition do not contribute on the limit. Then we only have to study  $\Gamma_t^n(1)$ and $\Gamma_t^n(2)$. For the first one, we use the same decomposition as cases \ref{2} and \ref{4}, namely we have $\Gamma_t^n(1)=\Gamma_t^n(1,1)+\Gamma_t^n(1,2)$  as given in \eqref{eq:gamma1}.
Now, similarly, we separate $\Gamma^n(2)$ into two terms: the first term that will be the essential term of the limit contains only the first jumps and the drifts and the second term that will be sort in the rest terms contains all the other jumps. Then, we have
	\begin{align*}
	\Gamma_t^n(2)=&u_{n,m}\sum_{i=1}^{[nt]}\sum_{k=2}^m\left[\int_{I(nm,i,k)}\hspace{-0.5cm}ff'(X_{t_{i}^1}^n)(A^{\beta_{n}}_{t_i^k}-A^{\beta_{n}}_{t^{1}_i})dA^{\beta_{n}}_s+\int_{I(nm,i,k)}ff'(X_{t_{i}^1}^n)(N^{\beta_{n}}_{t_i^k}-N^{\beta_{n}}_{t^{1}_i})dA^{\beta_{n}}_s\right.\\
	&\left.+\int_{I(nm,i,k)}ff'(X_{t_{i}^1}^n)(A^{\beta_{n}}_{t_i^k}-A^{\beta_{n}}_{t^{1}_i})dN^{\beta_{n}}_s\right]\\
	=&u_{n,m}\sum_{i=1}^{[nt]}\sum_{k=2}^mff'(X_{t_{i}^1}^n)\left[\frac{d_{n}^2(k-1)}{n^2m^2}+\frac{d_{n}}{nm}\sum_{j=1}^{k-1}(N^{\beta_{n}}_{t^{j+1}_i}-N^{\beta_{n}}_{t^{j}_i})+\frac{d_{n}(k-1)}{nm}(N^{\beta_{n}}_{t^{k+1}_i}-N^{\beta_{n}}_{t_i^k})\right]\\=&\Gamma_t^n(2,1)+\Gamma_t^n(2,2),
\end{align*}
where 
\begin{align*}
	\Gamma_t^n(2,1)=&u_{n,m}\sum_{i=1}^{[nt]}\sum_{k=2}^mff'(X_{t_{i}^1}^n)\left[\frac{d_{n}^2(k-1)}{n^2m^2}+\frac{d_{n}}{nm}\sum_{j=1}^{k-1}\Delta Y_{T^{\beta_n}_1(t^{j}_i)}\1_{\{K(t^{j}_i)\geq1\}}\right.\\&\left.\hspace{8cm}+\frac{d_{n}(k-1)}{nm}\Delta Y_{T^{\beta_n}_1(t_i^k)}\1_{\{K(t_i^k)\geq1\}}\right],\\
	\Gamma_t^n(2,2)=&u_{n,m}\sum_{i=1}^{[nt]}\sum_{k=2}^mff'(X_{t_{i}^1}^n)\left[\frac{d_{n}}{nm}\sum_{j=1}^{k-1}\sum_{h=2}^{K(t_{j}^i)}\Delta Y_{T^{\beta_n}_h(t^{j}_i)}+\frac{d_{n}(k-1)}{nm}\sum_{h=2}^{K(t^{j}_i)}\Delta Y_{T^{\beta_n}_h(t^{j}_i)}\right].
\end{align*}
In this case, $u_{n,m}Z^{n,m}_t=\mathcal{M}^{n,m}_t+\mathcal{R}^{n,m}_t,$ where  
\begin{align}\label{eq:MR3}
\mathcal{M}^{n,m}_t=\Gamma_t^n(1,1)+\Gamma_t^n(2,1)\quad \textrm{and}\quad	\mathcal{R}^{n,m}_t={\Gamma}^n_t(1,2)+{\Gamma}^n_t(2,2)+\sum_{i=3}^5\Gamma^n_t(i)
\end{align}
where $\Gamma^n_t(5)=u_{n,m}\int_0^{\eta_n(t)}k(X_{\eta_n(s-)}^n,Y_{\eta_{nm}(s-)}-Y_{\eta_n(s-)})(Y_{\eta_{nm}(s-)}-Y_{\eta_n(s-)})^2dY_s$.
The proof of the following lemma is postponed to Appendix \ref{app:A}. 
\begin{lemma}\label{lem:rest3}
	For the case \ref{3}, we have  as $n\rightarrow\infty$, the sequence   $(\mathcal{R}^{n,m})_{n\ge0}$  converges uniformly to $0$ in probability.	
\end{lemma}
\begin{lemma}\label{lem:tight3}
		For the case \ref{3}, the sequences $(\overline{Y}^n(1))_{n\ge0 }$, $(\overline{Y}^n(2))_{n\ge0}$ and $({\mathcal{M}}^{n,m})_{n\ge0}$ are tight.
\end{lemma}
\begin{proof}
First, instead of working with  $\overline{Y}^{n}_t(1)=\sum_{i=1}^{[nt]}\sum_{k=1}^m(M^{\beta_n}_{t_i^k,t_i^{k+1}}+\sum_{j\geq2}\Delta Y_{T^{\beta_n}_{ \s j}( t_i^k)}\1_{\{K({\s t_i^k})\geq j\}})$, it is enough to prove that for each $k\in\{1,\dots,m\}$ the triangular arrays with generic terms $y^{n,m}_{i,k}(1,1)=M^{\beta_n}_{t_i^k,t_i^{k+1}}$ and $y^{n,m}_{i,k}(1,2)=\sum_{j\geq2}\Delta Y_{T^{\beta_n}_{ \s j}( t_i^k)}\1_{\{K({\s t_i^k})\geq j\}}$ are tight. 
By property \ref{I}, \eqref{eq:3.13} and Lemma \ref{lem:smallj}, for the first one,  we have
\begin{align*}
	\mathbb{E}(y^{n,m}_{i,k}(1,1)|\mathcal{F}_{t^{1}_i})
	=0,\quad
	\mathbb{E}((y^{n,m}_{i,k}(1,1))^2|\mathcal{F}_{t^{1}_i})=\frac{c_{n}}{nm}.
\end{align*}
and we conclude by using  \eqref{eq:3.13}, $c_n\le C\beta_n^{2-\alpha}$, the criteria \eqref{eq:2.8} and Lemma \ref{lem:tightrest}. For the second one,  we have
\begin{multline}\label{eq:y2case3}
	\mathbb{E}(|y^{n,m}_{i,k}(1,2)||\mathcal{F}_{t^{1}_i})
	\le\frac{1}{\theta(\beta_{n})}\int_{|x|>\beta_n}|x|F(dx)\sum_{j\ge2}\mathbb{P}(K({\s t_i^k})\geq j|\mathcal{F}_{t^{1}_i})\\=\frac{\delta_n}{\theta(\beta_n)}(\mathbb{E}(K({\s t_i^k})|\mathcal{F}_{t^{1}_i})-\mathbb{P}(K({\s t_i^k})\ge1|\mathcal{F}_{t^{1}_i}))=\frac{\delta_n}{\theta(\beta_n)}(\lambda_{n,m}+e^{-\lambda_{n,m}}-1)\le \frac{\delta_n\lambda_{n,m}^2}{\theta(\beta_{n})}= \frac{\delta_n\lambda_{n,m}}{nm}
\end{multline}
and we conclude by using \eqref{eq:3.13}, $\delta_n\le C\log{1/\beta_n}$, $\lambda_{n,m}\le\frac{C}{n\beta_n}$ (see \eqref{h1}), $\beta_n=\frac{\log{n}}{n}$, the criteria \eqref{eq:2.7} and from the second part of Lemma \ref{lem:tightrest}.
Therefore, it is clear that for case \ref{3},  $(\overline{Y}^n(1))_{n\ge0}$ is tight.
Next, considering $\overline{Y}^n(2)$, from \eqref{eq:y2a} and \eqref{eq:y2b}, as $d_n$ and ${d'}_n$ are bounded by $C\log{1/\beta_n}$ from \eqref{eq:3.13}, $\lambda_{n,m}\le\frac{C}{n\beta_n}$ from \eqref{h1}, $\beta_n=\frac{\log{n}}{n}$, $y^n_i(2)$ satisfies \eqref{eq:2.8} ensuring the tightness of $(\overline{Y}^n(2))_{n\ge0}$ from Lemma \ref{lem:tightrest}.
		Finally, we consider $(\mathcal{M}^{n,m})_{n\ge0}$, equivalently, we prove that $(\Gamma^n(1,1))_{n\ge0}$ and $(\Gamma^n(2,1))_{n\ge0}$ are tight. Because $ff'$ is bounded,  for $k\in\{2,\dots,m\}$ and $j<k$ fixed, it is enough to prove that the triangular arrays corresponding to generic terms  $\zeta_{i,k,j}^n(1)$ and $\zeta_{i,k}^n(2)$ are tight where
		\begin{align}\label{eq:mt}
			&\zeta_{i,j,k}^n(1)=u_{n,m}\Delta Y_{T^{\beta_n}_1(t_i^k)}\1_{K(t_i^k)\geq1}\Delta Y_{T^{\beta_n}_1(t^{j}_i)}\1_{\{K(t^{j}_i)\geq1\}},\\
			&\zeta_{i,k}^n(2)=u_{n,m}\left[\frac{d_{n}^2(k-1)}{n^2m^2}+\frac{d_{n}}{nm}\sum_{j=1}^{k-1}\Delta Y_{T^{\beta_n}_1(t^{j}_i)}\1_{\{K(t^{j}_i)\geq1\}}+\frac{d_{n}(k-1)}{nm}\Delta Y_{T^{\beta_n}_1(t_i^k)}\1_{\{K(t_i^k)\geq1\}}\right].\nonumber
		\end{align} 
For the first term, by similar arguments, we have
	\begin{align*}
		\mathbb{E}(|\zeta_{i,k,j}^n(1)||\mathcal{F}_{t^{1}_i})=&u_{n,m}\frac{(1-e^{-\lambda_{n,m}})^2}{(\theta(\beta_n))^2}\delta_n^2\le C\frac{u_{n,m}\delta_n^2}{n^2}.	
	\end{align*}
	Then, the tightness is obtained by \eqref{eq:3.13} that ${\delta}_n$ are bounded by $C\log{1/\beta_n}$,  $\beta_n=\frac{\log{n}}{n}$, $u_{n,m}=\frac{mn}{(m-1)(\log{n})^2}$, \eqref{eq:2.7} and Lemma \ref{lem:tightrest}.
	For the second term, by using property \ref{I} and $1-e^{\lambda_{n,m}}\le \lambda_{n,m}$, we have 
	$$	\mathbb{E}(|\zeta_{i,k}^n(2)||\mathcal{F}_{t^{1}_i})\le \frac{Cu_{n,m}d_n}{n^2}(d_n+2\delta_n),$$
	and we conclude by  $d_n$ and ${\delta}_n$ are bounded by $C\log{1/\beta_n}$ (see \eqref{eq:3.13}),  $\beta_n=\frac{\log{n}}{n}$, $u_{n,m}=\frac{mn}{(m-1)(\log{n})^2}$, criteria \eqref{eq:2.7} and Lemma \ref{lem:tightrest}. Therefore, we get $(\mathcal{M}^{n,m})_{n\ge0}$ is tight.
\end{proof}
\begin{theorem}\label{thm:noise3}
		For the case \ref{3}, we have 
	$$(\overline{Y}^{n},{\mathcal{M}}^{n,m})\stackrel{\mathbb{P}}{\longrightarrow}(Y,Z),$$
	where $Z$ is defined as $\eqref{eq:1.9}$.
\end{theorem}
\begin{proof}
		Since $\overline{Y}^{n}$ converges pointwise to $Y$ when $n\rightarrow\infty$ for the Skorokhod topology, then we only need to prove ${\mathcal{M}}^{n,m}\stackrel{\mathbb{P}}{\longrightarrow}Z$. Since $ff'$ is Lipschitz-continuous, by virtue of Lemma \ref{lem:cv1tot} or Lemma \ref{lem:t1}, it is enough to prove ${\Gamma'}^n_1(1,1)+{\Gamma'}^n_1(2,1)\stackrel{\mathbb{P}}{\longrightarrow}-\frac{{\theta'}^2}{4}$ where ${\Gamma'}^n_1(1,1)=\sum_{i=1}^{n}\sum_{k=1}^{m}\sum_{j=1}^{k-1}\zeta_{i,j,k}^n(1)$ and ${\Gamma'}^n_1(2,1)=\sum_{k=2}^m\sum_{i=1}^{n}\zeta_{i,k}^n(2)$ with $\zeta_{i,j,k}^n(1)$ and $\zeta_{i,k}^n(2)$ given by \eqref{eq:mt}.	First, concerning ${\Gamma'}^n_1(2,1)$, for $k\in\{2,\dots,m\}$ fixed, on the one hand, by using property \ref{I}, we have 
		\begin{align*}
			\mathbb{E}(\zeta_{i,k}^n(2)|\mathcal{F}_{t^{1}_i})=&u_{n,m}\left[\frac{d_{n}^2(k-1)}{n^2m^2}+\frac{2(k-1)d_{n}}{nm}\frac{1-e^{-\lambda_{n,m}}}{\theta(\beta_{n})}\int_{|x|>\beta_{n}}xF(dx)\right]\\
			\sim&\frac{(k-1)u_{n,m}}{n^2m^2}(d_{n}^2+2d_{n}{d'}_{n}).
				\end{align*}
		From \eqref{3.1} and \eqref{eq:3.5}, we
		have 
		$
	 {d'}_n\sim \theta'\log{\frac{1}{\beta_{n}}}$ and  $d_{n}=b'-{d'}_{n}$. 
		 Therefore, using $u_{n,m}=\frac{nm}{(m-1)(\log{n})^2}$, $\beta_{n}=\frac{\log{n}}{n}$ and    $\mathbb{E}(\zeta_{i,k}^n(2)|\mathcal{F}_{t^{1}_i})$ is non random and independent of $i$, we get
		$
			n\mathbb{E}(\zeta_{1,k}^n(2))				\underset{n\rightarrow\infty}{\longrightarrow}-\frac{(k-1){\theta'}^2}{m(m-1)}.
	$
		On the other hand, using property \ref{I}, the inequality $(a+b)^2\le 2(a^2+b^2)$ and  $\int_{\R}x^2F(dx)<\infty$ (see Remark \ref{rmk:1}), we have 
	$
			\mathbb{E}(|\zeta_{i,k}^n(2)|^2|\mathcal{F}_{t^{1}_i})	
			\le\frac{Cu_{n,m}^2}{n^3}(\frac{d_{n}^4}{n}+d_{n}^2).	
	$
	Since $d_n\le C\log{1/\beta_n}$ (see \eqref{eq:3.13}), we proceed similarly as above to get  $	n\mathbb{E}(|\zeta_{1,k}^n(2)|^2)
	\underset{n\rightarrow\infty}{\longrightarrow}0$.
Then, since
		$\mathbb{V}\left(	\sum_{i=1}^n\zeta_{i,k}^n(2)\right)\le n\mathbb{E}((	\zeta_{1,k}^n(2))^2)$, 
	we get $
		\sum_{i=1}^n\zeta_{i,k}^n(2)\stackrel{\mathbb{P}}{\rightarrow}-\frac{(k-1){\theta'}^2}{m(m-1)}
		$ and  we deduce that
		$
		{\Gamma'}_1^n(2,1)\stackrel{\mathbb{P}}{\rightarrow}-\frac{{\theta'}^2}{2}.
		$
	Secondly, concerning ${\Gamma'}^n_1(1,1)$, we prove its uniform convergence in probability by considering for  $k\in\{2,\dots,m\}$ and $j<k$ fixed, the generic term  $\zeta_{i,j,k}^n(1)$. To do so, we apply Lemma \ref{lem:cv} to this sequence in which $(\zeta_{i,j,k}^n(1))_{1\le i\le n}$ are  i.i.d. Note that the law of $\zeta_{i,j,k}^n(1)$ does not depend on parameters $j$ and $k$, then we denote it by $K_{n,m}$. We will prove that $nK_{n,m}(h)$ converges as $n\rightarrow\infty$ for some suitable function $h$ corresponding to each assertions of Lemma \ref{lem:cv}.
	Therefore, we have $$nK_{n,m}(h)=\frac{(1-e^{-\lambda_{n,m}})^2}{nm^2\lambda_{n,m}^2}\int_{|x|>\beta_{n}}\int_{\{|y|>\beta_n\}}h(u_{n,m}xy)F(dy)F(dx).$$ 
	\paragraph{\textbf{Concerning assertion (i)}:} Here, we choose $h=h_{\omega}$ where $h_{\omega}(x)=\1_{\{|x|>\omega\}}$ with some $\omega>0$. Using $1-e^{-\lambda_{n,m}}\le\lambda_{n,m}$, it is easy to check that 
		$
				nK_{n,m}(h_{\omega})
								\le\frac{1}{nm^2}\int_{|x|>\beta_{n}}\theta\left(\frac{\omega}{u_{n,m}|x|}\right)F(dx).
		$
	By hypothesis \eqref{h1} for $\alpha=1$, $\theta(\beta)\le\frac{C}{\beta}$ and $\delta_n\equiv\rho_n$, we have
		$
				nK_{n,m}(h_{\omega})\le\frac{Cu_{n,m}\rho_n}{n\omega}.
	$
		 Then, by our choices of $u_{n,m}$, $\beta_n$ and $\rho(\beta)\le C\log{1/\beta}$ (see \eqref{eq:3.3}), we get $nK_{n,m}(h_{\omega})\underset{n\rightarrow\infty}{\longrightarrow}0$.
		\paragraph{\textbf{Concerning assertion (ii)}:} We choose $h=h'$ where $h'(x)=x\1_{\{|x|\le1\}}$. Using $1-e^{-\lambda_{n,m}}\sim\lambda_{n,m}$, $u_{n,m}\beta_n\underset{n\rightarrow\infty}{\longrightarrow}0$ and assumption \ref{A}, namely $F$ vanishes outside $[-p,p]$,  we have for $n$ large enough $\beta_{n}\le\frac{1}{u_{n,m}p}\le\frac{1}{u_{n,m}\mid x\mid}$ and 
			\begin{align*}
				nK_{n,m}(h')\sim&\frac{u_{n,m}}{nm^2}\int_{\mid x\mid>\beta_{n}}\hspace{-0.3cm}x\int_{\beta_{n}<\mid y\mid\le\frac{1}{u_{n,m}\mid x\mid}}\hspace{-1cm}yF(dy)F(dx)
				=\frac{u_{n,m}}{nm^2}\left({d'}_{n}^2-\int_{\mid x\mid>\beta_{n}}\hspace{-0.3cm}x\int_{\mid y\mid>\frac{1}{u_{n,m}\mid x\mid}}\hspace{-0.5cm}yF(dy)F(dx)\right).
			\end{align*}
	Since $u_{n,m}=\frac{nm}{(m-1)(\log{n})^2}$ and using ${d'}_n\sim\theta'\log{1/\beta_n}$ (see \eqref{eq:3.5}), the first term in the r.h.s. is equivalent to $\frac{\theta'^2}{m(m-1)}$.  Now, taking $\varepsilon>0$,  there exists $\varepsilon'\in(0,1)$ such that for $\beta\in(0,\varepsilon')$ we have $|\frac{{d'}(\beta)}{\log{(1/\beta)}\theta'}-1|\le\varepsilon$. Considering the second term in the r.h.s., as for $n$ large enough $\frac{1}{u_{n,m}\varepsilon'}\ge\beta_n$, we rewrite it as the sum of $x_n$ and $y_n$ with
	\begin{align*}
	\left\{\begin{array}{l}
		x_n=\frac{u_{n,m}}{nm^2}\int_{1/(u_{n,m}\varepsilon')\ge\mid x\mid>\beta_{n}}x\int_{\mid y\mid>\frac{1}{u_{n,m}\mid x\mid}}yF(dy)F(dx)\\
		y_n=\frac{u_{n,m}}{nm^2}\int_{\mid x\mid>1/(u_{n,m}\varepsilon')}x\int_{\mid y\mid>\frac{1}{u_{n,m}\mid x\mid}}yF(dy)F(dx).
	\end{array}\right.	
	\end{align*} 
 First, using $\frac{1}{u_{n,m}|x|}\ge\varepsilon'$,  $\delta(.)$ is decreasing and $\delta_n\le C\log{n}$  from \eqref{eq:3.13} we derive that
		\begin{align*}
		|x_n|\le& 	\frac{u_{n,m}}{nm^2}\int_{1/(u_{n,m}\varepsilon')\ge\mid x\mid>\beta_{n}}|x|F(dx)\delta(\varepsilon')\le \frac{Cu_{n,m}\delta_n}{n}\le  \frac{C}{\log{n}},  
		\end{align*} 
	 which converges to $0$. 
	Second, on the one hand, if we denote $${y'}_n=\frac{u_{n,m}{\theta'}}{nm^2}\int_{\mid x\mid>1/(u_{n,m}\varepsilon')}x\log{(u_{n,m}\mid x\mid)}F(dx),$$ then  we have $(1-\varepsilon){y'}_n\le y_n\le(1+\varepsilon){y'}_n$ which gives ${y}_n\sim{y'}_n$ since $\varepsilon$ is arbitrarily small. On the other hand, using ${d'}(\beta)\sim\theta'\log{1/\beta}$ (see \eqref{eq:3.5}) we have that $\frac{u_{n,m}{\theta'}}{nm^2}\log{(u_{n,m})}{d'}(\frac{1}{u_{n,m}\varepsilon'})$ converges to $\frac{{\theta'}^2}{m(m-1)}$ and  thanks to \eqref{eq:3.6} we have that $\frac{u_{n,m}{\theta'}}{nm^2}\int_{\mid x\mid>1/(u_{n,m}\varepsilon')}x\log{(\mid x\mid)}F(dx)$ converges to $-\frac{{\theta'}^2}{2m(m-1)}$.
			Then, it is clear that $nK_{n,m}(h')\underset{n\rightarrow\infty}{\longrightarrow} \frac{{\theta'}^2}{2m(m-1)}.$
	\paragraph{\textbf{Concerning assertion (iii)}:} Choosing $h=h_1$ where $h_1(x)=x^2\1_{\{|x|\le1\}}$, using the inequality $1-e^{-\lambda_{n,m}}\le\lambda_{n,m}$ and $c(\beta)\le C\beta$ (see \eqref{eq:3.3}),  we have 
		\begin{align*}
			nK_{n,m}(h_1)\le&\frac{u_{n,m}^2}{nm^2} \int_{|x|>\beta_{n}}\int_{|y|>\beta_n} x^2y^2\1_{\{u_{n,m}|xy|\leq1\}}F(dx)F(dy)\\
			\le&\frac{u_{n,m}^2}{nm^2}\int_{|x|>\beta_{n}}x^2c\left( \frac{1}{u_{n,m}|x|}\right)F(dx)\le \frac{Cu_{n,m}\rho_n}{n}.
		\end{align*}
		Therefore,  thanks to our choices of $u_{n,m}$, $\beta_n$ and using $\rho(\beta)\le C\log{1/\beta}$ (see \eqref{eq:3.3}), we get that $nK_{n,m}(h_1)$ converges to $0$ as $n\rightarrow\infty$.
		In conclusion, the limit processes have no Gaussian part, a Lévy measure equal to $0$ and a drift part equal to $\frac{{\theta'}^2}{2m(m-1)}$. Finally, we get
		$
		\sum_{i=1}^{n}\zeta_{i,j,k}^{n}(1)\stackrel{\mathbb{P}}{\rightarrow}\frac{{\theta'}^2}{2m(m-1)}
		$ and
		$
		{\Gamma'}_1^{n}(1,1)\stackrel{\mathbb{P}}{\rightarrow}\sum_{k=2}^m\sum_{j=1}^{k-1}\frac{{\theta'}^2}{2m(m-1)}=\frac{{\theta'}^2}{4}.$ This completes the proof.
\end{proof}
\subsection{Asymptotic behavior of the couple $(\overline{Y}^{n},u_{n,m}Z^{n,m})$ for case \ref{5}.}
For the first component $\overline{Y}^n$, we use the same decomposition given by the relation \eqref{Y}. For the second one, we consider the formula of $u_{n,m}Z^{n,m}$ given in \eqref{Znm}. In this case, the second term from this formula does not contribute to the limit. Therefore, we need only to give the analysis for its first term. We have
	\begin{align*}
		u_{n,m}\int_0^{\eta_n(t)}ff'(X_{\eta_n(s-)}^n)(Y_{\eta_{nm}(s-)}-Y_{\eta_n(s-)})dY_s
		=\sum_{i=1}^4\Gamma^n_t(i),
	\end{align*}
	where
	\begin{align*}
		\left\{\begin{array}{l}
			\Gamma^n_t(1)=u_{n,m}(\int_0^{\eta_n(t)}ff'(X_{\eta_n(s-)}^n)(M_{\eta_{nm}(s-)}^{\beta_{n}}-M_{\eta_n(s-)}^{\beta_{n}})dN_s^{\beta_{n}}\\\hspace{5cm}+\int_0^{\eta_n(t)}ff'(X_{\eta_n(s-)}^n)(N_{\eta_{nm}(s-)}^{\beta_{n}}-N_{\eta_n(s-)}^{\beta_{n}})dM^{\beta_{n}}_s),\\
			\Gamma^n_t(2)=u_{n,m}(\int_0^{\eta_n(t)}ff'(X_{\eta_n(s-)}^n)(Y_{\eta_{nm}(s-)}^{\beta_{n}}-Y_{\eta_n(s-)}^{\beta_{n}})dM_s^{\beta_{n}}\\\hspace{5cm}+\int_0^{\eta_n(t)}ff'(X_{\eta_n(s-)}^n)(Y_{\eta_{nm}(s-)}^{\beta_{n}}-Y_{\eta_n(s-)}^{\beta_{n}})dA_s^{\beta_{n}}),\\
			\Gamma^n_t(3)=u_{n,m}(\int_0^{\eta_n(t)}ff'(X_{\eta_n(s-)}^n)(A_{\eta_{nm}(s-)}^{\beta_{n}}-A_{\eta_n(s-)}^{\beta_{n}})dN_s^{\beta_{n}}\\\hspace{5cm}+\int_0^{\eta_n(t)}ff'(X_{\eta_n(s-)}^n)(N_{\eta_{nm}(s-)}^{\beta_{n}}-N_{\eta_n(s-)}^{\beta_{n}})dA^{\beta_{n}}_s),\\
			\Gamma^n_t(4)=u_{n,m}\int_0^{\eta_n(t)}ff'(X_{\eta_n(s-)}^n)(N_{\eta_{nm}(s-)}^{\beta_{n}}-N_{\eta_n(s-)}^{\beta_{n}})dN^{\beta_{n}}_s.\end{array}\right.
	\end{align*}
 In this case, the three last terms do not contribute to the limit and we only have to study the first term $\Gamma^n(1)$. Let us first rewrite 
	$
		\Gamma^n_{t}(1)=\sum_{i=1}^{[nt]}\zeta^n_i(1),
	$
	with row-wise i.i.d. random variables $\zeta_i^n,i=1,2,\hdots$ defined by
		\begin{align*}
	\zeta^n_{i}(1)=u_{n,m}ff'(X_{t_i^1}^n)\sum_{k=2}^m\left[(M_{t_i^k}^{\beta_{n}}-M_{t^{1}_i}^{\beta_{n}})(N_{t_i^{k+1}}^{\beta_{n}}-N_{t^{k}_i}^{\beta_{n}})+(N_{t_i^k}^{\beta_{n}}-N_{t^{1}_i}^{\beta_{n}})(M_{t_i^{k+1}}^{\beta_{n}}-M_{t^{k}_i}^{\beta_{n}})\right].
	\end{align*}
	Now, using Fubini for the second term, we have that
	\begin{align*}
		\sum_{k=2}^m[\sum_{j=1}^{k-1}(N_{t_i^{j+1}}^{\beta_{n}}-N_{t^{j}_i}^{\beta_{n}})](M^{\beta_{n}}_{t^{k+1}_i}-M^{\beta_{n}}_{t_i^k})
		=\sum_{k=1}^{m-1}(N^{\beta_{n}}_{t^{k+1}_i}-N^{\beta_{n}}_{t_i^k})(M_{t^{m+1}_i}^{\beta_{n}}-M_{t^{k+1}_i}^{\beta_{n}}).
	\end{align*}
	Then we can rewrite our triangular array as follows
	\begin{align*}
	\zeta_i^n(1)=&u_{n,m}ff'(X_{t_i^1}^n)\left[\sum_{k=2}^m(M_{t_i^k}^{\beta_{n}}-M_{t^{1}_i}^{\beta_{n}})(N_{t^{k+1}_i}^{\beta_{n}}-N_{t_i^k}^{\beta_{n}})+\sum_{k=1}^{m-1}(N^{\beta_{n}}_{t^{k+1}_i}-N^{\beta_{n}}_{t_i^k})(M_{t^{m+1}_i}^{\beta_{n}}-M_{t^{k+1}_i}^{\beta_{n}})\right]\\
			=&u_{n,m}ff'(X_{t_i^1}^n)\sum_{k=1}^{m}(N_{t^{k+1}_i}^{\beta_{n}}-N_{t_i^k}^{\beta_{n}})[(M_{t^{m+1}_i}^{\beta_{n}}-M_{t^{1}_i}^{\beta_{n}})-(M_{t^{k+1}_i}^{\beta_{n}}-M_{t_i^k}^{\beta_{n}})]\\
		=&u_{n,m}ff'(X_{t_i^1}^n)\sum_{k=1}^{m}\left(\Delta Y_{T^{\beta_n}_1(t_i^k)}\1_{\{K(t_i^k)\geq1\}}+\sum_{j=2}^{K(t_i^k)}\Delta Y_{T^{\beta_n}_j(t_i^k)}\right)\tilde{M}^{n,m}_{i,k},
	\end{align*}
	where $\tilde{M}^{n,m}_{i,k}=(M_{t^{m+1}_i}^{\beta_{n}}-M_{t^{1}_i}^{\beta_{n}})-(M_{t^{k+1}_i}^{\beta_{n}}-M_{t^{k}_i}^{\beta_{n}})=\sum\limits_{j=1, j\neq k}^m(M_{t^{j+1}_i}^{\beta_{n}}-M_{t^{j}_i}^{\beta_{n}})$. Now,  we separate $\Gamma^n(1)$ into two terms: the first term which is the essential term of the limit corresponds to the part with the first jumps and the second term which will be sort in the rest terms corresponds to the part of all the other jumps. In particular, we have $\Gamma^n(1)={\Gamma}^n(1,1)+\Gamma^n(1,2)$ where  
	\begin{align}\label{eq:gamma}\left\{\begin{array}{l}
		{\Gamma}^n_t(1,1)=u_{n,m}\sum_{i=1}^{[nt]}\sum_{k=1}^{m}ff'(X_{t_i^1}^n)\Delta Y_{T^{\beta_n}_1(t_i^k)}\1_{\{K(t_i^k)\geq1\}}\tilde{M}^{n,m}_{i,k},\\
		{\Gamma}^n_t(1,2)=u_{n,m}\sum_{i=1}^{[nt]}\sum_{k=1}^{m}ff'(X_{t_i^1}^n)\sum_{j=2}^{K(t_i^k)}\Delta Y_{T^{\beta_n}_j(t_i^k)}\tilde{M}^{n,m}_{i,k}.\end{array}\right.
	\end{align}
 In this case, we have  \begin{align}\label{eq:MR5}
 \mathcal{M}^{n,m}_t={\Gamma}^n_t(1,1)\quad \textrm{and}\quad \mathcal{R}^{n,m}_t={\Gamma}^n_t(1,2)+\sum_{i=2}^5\Gamma^n_t(i),
 \end{align}
with $\Gamma^n_t(5)=u_{n,m}\int_0^{\eta_n(t)}k(X_{\eta_n(s-)}^n,Y_{\eta_{nm}(s-)}-Y_{\eta_n(s-)})(Y_{\eta_{nm}(s-)}-Y_{\eta_n(s-)})^2dY_s$.
The proof of the following lemma is postponed to Appendix \ref{app:A}.
\begin{lemma}\label{lem:rest5}
For case \ref{5}, we have as $n\rightarrow\infty$, the sequences $(\overline{Y}^n(1))_{n\ge0 }$ and    $(\mathcal{R}^{n,m})_{n\ge0}$ converge uniformly to $0$ in probability.	
\end{lemma}
\begin{lemma}\label{lem:tight5}
For case \ref{5}, the sequences  $(\overline{Y}^n(2))_{n\ge0}$ and $({\mathcal{M}}^{n,m})_{n\ge0}$ are tight. 
\end{lemma}
\begin{proof}
		First, we consider the sequence $(\overline{Y}^n(2))_{n\ge0}$ given by \eqref{Y}. From \eqref{eq:y2a} and \eqref{eq:y2b}, using hypothesis \eqref{h1},  $\lambda_{n,m}\le\frac{C}{n\beta_n}$, $d_n$ and ${d'}_n$ are bounded by $C\beta_n^{1-\alpha}$ from \eqref{eq:3.3} with the choice $\beta_n=\frac{\log{n}}{n^{1/(2\alpha)}}$, then $y^n_i(2)$ satisfies \eqref{eq:2.8} ensuring the tightness from the second part of Lemma \ref{lem:tightrest}.
		Now, we recall that $\mathcal{M}^{n,m}_t={\Gamma}^n_t(1,1)$ given by \eqref{eq:gamma} and as $ff'$ is bounded, for $k,j\in\{1,\dots,m\}$ fixed and $j\neq k$, it is enough to prove that the triangular array with the generic term $\zeta_{i,k,j}^n(1,1)=u_{n,m}\Delta Y_{T^{\beta_n}_1(t_i^k)}\1_{\{K(t_i^k)\geq1\}}(M_{t^{j+1}_i}^{\beta_{n}}-M_{t^{j}_i}^{\beta_{n}})$ is tight. 
		By using Property \ref{M} and the independence of the increments, for $|u|\le1$ we have
		\begin{align*}
			\mathbb{E}(e^{\texttt{i}u\zeta^n_{i,k,j}(1,1)}|\mathcal{F}_{t^{1}_i})&=e^{-\lambda_{n,m}}+\frac{1-e^{-\lambda_{n,m}}}{\theta(\beta_n)}\int_{|x|>\beta_n}	F(dx)\mathbb{E}\left(e^{\texttt{i}uu_{n,m}x(M^{\beta_n}_{t^{j+1}_i}-M^{\beta_n}_{t^{j}_i})}\right)\\
			&=e^{-\lambda_{n,m}}+\frac{1-e^{-\lambda_{n,m}}}{\theta(\beta_n)}\int_{|x|>\beta_n}	F(dx)e^{z_{n,m}(x,u)}\\
			&=1+\frac{1-e^{-\lambda_{n,m}}}{nm\lambda_{n,m}}\int_{|x|>\beta_n}	F(dx)(e^{z_{n,m}(x,u)}-1),
		\end{align*}
	where $z_{n,m}(x,u)=\frac{1}{nm}\int_{|y|\le\beta_n}(e^{\texttt{i}uu_{n,m}xy}-1-\texttt{i}uu_{n,m}xy)F(dy)$. 
		By applying \cite[Lemma 8.6]{d}  for first and second orders, we get $|e^{\texttt{i}uu_{n,m}xy}-1-\texttt{i}uu_{n,m}xy|\le C|uu_{n,m}xy|\wedge|uu_{n,m}xy|^2$. As $u_{n,m}\beta_n\underset{n\rightarrow\infty}{\longrightarrow}\infty$, for $n$ large enough, combining these results with $\delta(\beta)\le C\beta^{1-\alpha}$ and $c(\beta)\le C\beta^{2-\alpha}$ (see \eqref{eq:3.3}), we have
		\begin{multline}\label{eq:tool}
			|z_{n,m}(x,u)|\le\frac{C}{n}\int_{|y|\le\beta_n} (|uu_{n,m}xy|\wedge|uu_{n,m}xy|^2)F(dy)\\=	\frac{C}{n}|uu_{n,m}x|\int_{\beta_n\ge|y|>1/uu_{n,m}|x|)}\hspace{-1cm}|y|F(dy)+	\frac{C}{n}|uu_{n,m}x|^2	\int_{|y|\le1/(uu_{n,m}|x|)}\hspace{-1cm}y^2F(dy)\le 	\frac{C}{n}|uu_{n,m}x|^{\alpha}.
		\end{multline}
		Then, for $u_{n,m}=\left[\frac{mn}{(m-1)\log{n}}\right]^{1/\alpha}$, the suprema of $|z_{n,m}(x,u)|$ over all  $|x|\le p$ and $|u|\le 1$ goes to $0$ as $n$ tends to infinity. 
		Therefore, using $|e^{z_{n,m}(x,u)}-1|\le C |z_{n,m}(x,u)|$ for $n$ large enough  by \eqref{eval}, $\rho_n\le C\log{\frac{1}{\beta_n}}$ (see \eqref{eq:3.3}) and $1-e^{-\lambda_{n,m}}\le\lambda_{n,m}$,  we deduce that
		$
			|\mathbb{E}(e^{\texttt{i}u\zeta^n_{i,k,j}(1,1)}|\mathcal{F}_{t_i^1})-1|\le \frac{C|u|^{\alpha}u_{n,m}^{\alpha}\log{\frac{1}{\beta_n}}}{n^2}.$
		Then, $\zeta^n_{i,k,j}(1,1)$ satisfies \eqref{eq:2.12} with $\xi{'''}_{n,u}=\frac{C|u|^{\alpha}u_{n,m}^{\alpha}}{n}\log{\frac{1}{\beta_n}}$ which is bounded by $C$ for all $|u|\le1$. Thus, combining Lemma \ref{lem:add} and the second part of  Lemma \ref{lem:tightrest}, we get the tightness of $(\mathcal{M}^{n,m})_{n\ge0}$.
\end{proof}
	 \begin{theorem}\label{thm:noise5}For case \ref{5}, we have
	 	\begin{align}\label{cv5}(\overline{Y}^{n}(2),{\mathcal{M}}^{n,m})\stackrel{\mathcal{L}}{\longrightarrow}(Y,Z),
	 	\end{align}
	 	where $Z$ is defined as $\eqref{eq:1.8}$.
	 \end{theorem}
	 \begin{proof}
	 First, we denote ${\Gamma'}^n_t(1,1)=u_{n,m}\sum_{i=1}^{[nt]}\sum_{k=1}^{m}\Delta Y_{T^{\beta_n}_1(t_i^k)}\1_{\{K(t_i^k)\geq1\}}\tilde{M}^{n,m}_{i,k}$. Then, as $ff'$ is Lipschitz-continuous, by virtue of Lemma \ref{lem:cv1tot},  in order to prove the convergence in law of the pair $(\overline{Y}^n(2),\mathcal{M}^{n,m})$, it is enough to consider the convergence of the pair $(\overline{Y}^n_1(2),{\Gamma'}^n_1(1,1))$ with  $\overline{Y}^n_1(2)$ given in \eqref{Y}. 
	 By  the independence structure, for $u$ and $v$ in $\R$, we have
	\begin{align*}
		&\mathbb{E}(e^{\texttt{i}(u\overline{Y}^n_1(2)+v{\Gamma'}^n_1(1,1))})
		=e^{\texttt{i}ud_{n}}\mathbb{E}(e^{\texttt{i}\sum_{i=1}^{n}\sum_{k=1}^{m}\Delta Y_{T^{\beta_n}_1(t_i^k)}\1_{\{K(t_i^k)\geq1\}}(u+vu_{n,m}\tilde{M}^{n,m}_{i,k})})\\
		=&e^{\texttt{i}ud_{n}}\prod_{i=1}^n\mathbb{E}(e^{\texttt{i}\sum_{k=1}^m\Delta Y_{T^{\beta_n}_1(t_i^k)}\1_{\{K(t_i^k)\geq1\}}(u+vu_{n,m}\tilde{M}^{n,m}_{i,k})}).
	\end{align*}
For $i\in\{1,\dots,n\}$ fixed, again by tower property, $\mathbb{E}(e^{\texttt{i}\sum_{k=1}^{m}\Delta Y_{T^{\beta_n}_1(t_i^k)}\1_{\{K(t_i^k)\geq1\}}(u+vu_{n,m}\tilde{M}^{n,m}_{i,k})})$ equals to
	\begin{align*}
&\mathbb{E}(\mathbb{E}(e^{\texttt{i}\sum_{k=1}^{m}\Delta Y_{T^{\beta_n}_1(t_i^k)}\1_{\{K(t_i^k)\geq1\}}(u+vu_{n,m}\tilde{M}^{n,m}_{i,k})}|\sigma(M_{t^{j+1}_i}^{\beta_{n}}-M_{t^{j}_i}^{\beta_{n}},j\in\{1,\dots,m\})))\\
			=&\mathbb{E}(\prod_{k=1}^{m}\mathbb{E}(e^{\texttt{i}\Delta Y_{T^{\beta_n}_1(t_i^k)}\1_{\{K(t_i^k)\geq1\}}(u+vu_{n,m}\tilde{M}^{n,m}_{i,k})}|\sigma(M_{t^{j+1}_i}^{\beta_{n}}-M_{t^{j}_i}^{\beta_{n}},j\in\{1,\dots,m\}))).
	\end{align*}
	For  $k\in\{1,\dots,m\}$ fixed, 	note that $\tilde{M}^{n,m}_{i,k}$ is the martingale part of the small jumps which is $\sigma(M_{t^{j+1}_i}^{\beta_{n}}-M_{t^{j}_i}^{\beta_{n}},j\in\{1,\dots,m\})$-measurable, independent of $K(t_i^k)$ and $\Delta Y_{T^{\beta_n}_1(t^{k}_i)}$  then by  \ref{I}  we get that  equals to  
	\begin{multline*}	
	\mathbb{E}(e^{\texttt{i}\Delta Y_{T^{\beta_n}_1(t^{k}_i)}\1_{\{K(t_i^k)\geq1\}}(u+vu_{n,m}\tilde{M}^{n,m}_{i,k})}|\sigma(M_{t^{j+1}_i}^{\beta_{n}}-M_{t^{j}_i}^{\beta_{n}},j\in\{1,\dots,m\})\vee\mathcal{F}_{t_i^k})\\	=e^{-\lambda_{n,m}}+\frac{1-e^{-\lambda_{n,m}}}{\theta(\beta_{n})}\int_{|x|>\beta_{n}}e^{\texttt{i}x(u+vu_{n,m}\tilde{M}^{n,m}_{i,k})}F(dx).
	\end{multline*}
	Therefore, using the independence structure of $(\tilde{M}^{n,m}_{i,k})_{i\in\{1,\dots,n\}}$ we can easily see by \ref{M} that for all $i\in\{1,\dots,n\}$  $\tilde{M}^{n,m}_{i,k}$ has the same distribution as $\tilde{M}^{n,m}_{1,k}$. Thus, we get
	\begin{align*}
	&	\mathbb{E}(e^{\texttt{i}(u\overline{Y}^n_1(2)+v{\Gamma'}^n_1(1,1))})
				=e^{\texttt{i}ud_{n}}\left[\mathbb{E}(\prod_{k=1}^{m}(1+\frac{1-e^{-\lambda_{n,m}}}{\theta(\beta_{n})}\int_{|x|>\beta_{n}}(e^{\texttt{i}x(u+vu_{n,m}\tilde{M}^{n,m}_{1,k})}-1))F(dx))\right]^n\\
		=&e^{\texttt{i}ud_{n}}\left[\mathbb{E}\left(\prod_{k=1}^{m}\exp{\left(\log{(1+\frac{1-e^{-\lambda_{n,m}}}{\theta(\beta_{n})}\int_{|x|>\beta_{n}}(e^{\texttt{i}x(u+vu_{n,m}\tilde{M}^{n,m}_{1,k})}-1)F(dx))}\right)}\right)\right]^n\\
		=&e^{\texttt{i}ud_{n}}\left[\mathbb{E}\left(\exp{\left(\sum_{k=1}^{m}\log{(1+\frac{1-e^{-\lambda_{n,m}}}{\theta(\beta_{n})}\int_{|x|>\beta_{n}}(e^{\texttt{i}x(u+vu_{n,m}\tilde{M}^{n,m}_{1,k})}-1)F(dx))}\right)}\right)\right]\\
	\nonumber
		=&e^{\texttt{i}ud_{n}}\left[\mathbb{E}\left(\exp{\left(\sum_{k=1}^{m}\frac{1-e^{-\lambda_{n,m}}}{\theta(\beta_{n})}\int_{|x|>\beta_{n}}F(dx)(e^{\texttt{i}x(u+vu_{n,m}\tilde{M}^{n,m}_{1,k})}-1)+\sum_{k=1}^{m}\mathcal{R}_{1,k}^{n,m}\right)}\right)\right]^n\\\nonumber
		=&e^{\texttt{i}ud_{n}}\left[\mathbb{E}\left(1+\sum_{k=1}^{m}\frac{1-e^{-\lambda_{n,m}}}{\theta(\beta_{n})}\int_{|x|>\beta_{n}}F(dx)(e^{\texttt{i}x(u+vu_{n,m}\tilde{M}^{n,m}_{1,k})}-1)+\sum_{k=1}^{m}\mathcal{R}_{1,k}^{n,m}+\mathcal{R}_1^n\right)\right]^n,
	\end{align*}
where
\begin{align*}
	&\mathcal{R}_{1,k}^{n,m}=\log{\left(1+\frac{1-e^{-\lambda_{n,m}}}{\theta(\beta_{n})}\int_{|x|>\beta_{n}}F(dx)(e^{\texttt{i}x(u+vu_{n,m}\tilde{M}^{n,m}_{1,k})}-1)\right)}\\
	&\hspace{5cm}-\frac{1-e^{-\lambda_{n,m}}}{\theta(\beta_{n})}\int_{|x|>\beta_{n}}F(dx)(e^{\texttt{i}x(u+vu_{n,m}\tilde{M}^{n,m}_{1,k})}-1),\\
	&\mathcal{R}_1^n=\exp{\left(\sum_{k=1}^{m}\frac{1-e^{-\lambda_{n,m}}}{\theta(\beta_{n})}\int_{|x|>\beta_{n}}F(dx)(e^{\texttt{i}x(u+vu_{n,m}\tilde{M}^{n,m}_{1,k})}-1)+\sum_{k=1}^{m}\mathcal{R}_{1,k}^{n,m}\right)}-1\\&\hspace{5cm}-\sum_{k=1}^{m}\frac{1-e^{-\lambda_{n,m}}}{\theta(\beta_{n})}\int_{|x|>\beta_{n}}F(dx)(e^{\texttt{i}x(u+vu_{n,m}\tilde{M}^{n,m}_{1,k})}-1)-\sum_{k=1}^{m}\mathcal{R}_{1,k}^{n,m}.
\end{align*}
Now, thanks to property \ref{M} we have that $\mathbb{E}(e^{\texttt{i}(u\overline{Y}^n_1(2)+v{\Gamma'}^n_1(1,1))})$ equals to
	\begin{align}\label{eq:main}
	\nonumber
	&e^{\texttt{i}ud_{n}}\left(1+\sum_{k=1}^{m}\frac{1-e^{-\lambda_{n,m}}}{\theta(\beta_{n})}\int_{|x|>\beta_{n}}(\mathbb{E}(e^{\texttt{i}x(u+vu_{n,m}\tilde{M}^{n,m}_{1,k})})-1)F(dx)+\sum_{k=1}^{m}\mathbb{E}(\mathcal{R}_{1,k}^{n,m})+\mathbb{E}(\mathcal{R}_1^n)\right)^n\\\nonumber
	=&e^{\texttt{i}ud_{n}}\left(1+\sum_{k=1}^{m}\frac{1-e^{-\lambda_{n,m}}}{\theta(\beta_{n})}\int_{|x|>\beta_{n}}(e^{\texttt{i}ux}\prod_{j=1,j\neq k}^{m}\mathbb{E}(e^{\texttt{i}vu_{n,m}x(M^{\beta_{n}}_{t_{j+1}^1}-M^{\beta_{n}}_{t_{j}^1})})-1)F(dx)\right.\\\nonumber
	&\left.\hspace{11cm}+\sum_{k=1}^{m}\mathbb{E}(\mathcal{R}_{1,k}^{n,m})+\mathbb{E}(\mathcal{R}_1^n)\right)^n\\
	=&e^{\texttt{i}ud_{n}}\left(1+\frac{1-e^{-\lambda_{n,m}}}{n\lambda_{n,m}}\int_{|x|>\beta_{n}}(e^{\texttt{i}ux+(m-1)z_{n,m}(x,v)}-1)F(dx)+\sum_{k=1}^{m}\mathbb{E}(\mathcal{R}_{1,k}^{n,m})+\mathbb{E}(\mathcal{R}_1^n)\right)^n,
\end{align}
	where 
$z_{n,m}(x,v)=\frac{1}{nm}\int_{|y|\leq\beta_{n}}(e^{\texttt{i} vu_{n,m} xy}-1-\texttt{i}vu_{n,m}xy)F(dy).$ 	Now, using $|e^{\texttt{i}x(u+vu_{n,m}\tilde{M}^{n,m}_{1,k})}-1|\le2$ and $1-e^{-\lambda_{n,m}}\le\lambda_{n,m}$, then, it is easy to check that for $n$ large enough, we have
\begin{align*}
	\left|\frac{1-e^{-\lambda_{n,m}}}{\theta(\beta_{n})}\int_{|x|>\beta_{n}}(e^{\texttt{i}x(u+vu_{n,m}\tilde{M}^{n,m}_{1,k})}-1)F(dx)\right|\le C\lambda_{n,m}\leq \frac{1}{2}.
\end{align*}
	From this, for any $k\in\{1,\dots,m\}$, using  the first evaluation in \ref{eval} and \eqref{h1}, we have
	\begin{align*}
		n|\mathbb{E}(\mathcal{R}_{1,k}^{n,m})|&\leq n\mathbb{E} \left[\frac{1-e^{-\lambda_{n,m}}}{nm\lambda_{n,m}}\int_{|x|>\beta_{n}}|e^{\texttt{i}x(u+vu_{n,m}\tilde{M}^{n,m}_{1,k})}-1|F(dx)\right]^2\leq C\frac{(\theta(\beta_n))^2}{n} \le \frac{C}{n\beta_{n}^{2\alpha}}
	\end{align*}
which converges to $0$ as $n\rightarrow\infty$ by the choice $\beta_n=\frac{\log{n}}{n^{1/(2\alpha)}}$.
Similarly, by the second evaluation in \ref{eval}, we have
	\begin{align*}
	&	n|\mathbb{E}(\mathcal{R}_1^n)|\leq n\mathbb{E}\left(\sum_{k=1}^{m}\frac{1-e^{-\lambda_{n,m}}}{\theta(\beta_{n})}\int_{|x|>\beta_{n}}|e^{\texttt{i}x(u+vu_{n,m}\tilde{M}^{n,m}_{1,k})}-1|F(dx)+\sum_{k=1}^{m}|\mathcal{R}_{1,k}^{n,m}|\right)^2\\
		&\leq C \left(n\mathbb{E}(\frac{1-e^{-\lambda_{n,m}}}{nm\lambda_{n,m}}\int_{|x|>\beta_{n}}\hspace{-0.5cm}|e^{\texttt{i}x(u+vu_{n,m}\tilde{M}^{n,m}_{1,k})}-1|F(dx))^2+n\mathbb{E}(|\mathcal{R}_{1,k}^{n,m}|)^2\right)
		\le \frac{C}{n\beta_{n}^{2\alpha}}+\frac{C}{n^3\beta_{n}^{4\alpha}}
	\end{align*}
which converges to $0$ as $n\rightarrow\infty$ by the choice $\beta_n=\frac{\log{n}}{n^{1/(2\alpha)}}$.
	Now, concerning the main term inside the bracket of \eqref{eq:main}, we have $\int_{|x|>\beta_{n}}(e^{\texttt{i}ux+(m-1)z_{n,m}(x,v)}-1)F(dx)=A_{n,m}(u)+B_{n,m}(v)+C_{n,m}(u,v)$ where
	\begin{align*}\left\{\begin{array}{l}
		A_{n,m}(u)=\int_{|x|>\beta_{n}}(e^{\texttt{i}ux}-1)F(dx),\quad
		B_{n,m}(v)=\int_{|x|>\beta_{n}}(e^{(m-1)z_{n,m}(v,x)}-1)F(dx),\\
		C_{n,m}(u,v)=\int_{|x|>\beta_{n}}(e^{\texttt{i}ux}-1)(e^{(m-1)z_{n,m}(x,v)}-1)F(dx).
		\end{array}\right.
	\end{align*}
  Since $1-e^{-\lambda_{n,m}}\sim\lambda_{n,m}$, it is enough to prove that the three terms $A_{n,m}(u)+iud_n$, $C_{n,m}(u,v)$ and $B_{n,m}(v)$ converge.
  	\paragraph{\textbf{Concerning $A_{n,m}(u)+\texttt{i}ud_{n}$.}} As $d_{n}=b-\int_{\beta_n<|x|\le1}xF(dx)$, this term is equal to
  \begin{align*}
  	\texttt{i}ub+\int_{|x|>\beta_n} (e^{\texttt{i}ux}-1-\texttt{i}ux\1_{|x|\leq 1})F(dx)\underset{n\rightarrow\infty}{\longrightarrow} \texttt{i}ub+\int (e^{\texttt{i}ux}-1-\texttt{i}ux1_{|x|\leq 1})F(dx).
  \end{align*} 
	\paragraph{\textbf{Concerning $C_{n,m}(u,v)$.}}By \eqref{eq:tool},  $|z_{n,m}(x,v)|\le \frac{C}{n}|vu_{n,m}x|^{\alpha}$ and the suprema of $|z_{n,m}(x,v)|$ over all $|x|\le p$ and $|v|\le 1$ go to $0$ as $n$ tends to $0$. 
Now, using \ref{eval}, we have  $|e^{(m-1)z_{n,m}(x,v)}-1|\le C |z_{n,m}(x,v)|$ and $|e^{iux}-1|\le C|ux|$. Then, since $x\mapsto|x|^{\alpha+1}$ is $F$-integrable as $\alpha>1$ (see Remark \ref{rmk:1}), we get $|C_{n,m}(u,v)|\leq \frac{C}{nm}|u||v|^{\alpha}u_{n,m}^{\alpha}\underset{n\rightarrow\infty}{\longrightarrow}0.
$
	\paragraph{\textbf{Concerning $B_{n,m}(v)$}}
	We rewrite $B_{n,m}(v)={B'}_{n,m}(v)+{B''}_{n,m}(v)$, with
	\begin{align*}\left\{\begin{array}{l}
		{B'}_{n,m}(v)=\int_{|x|>\beta_{n}}(m-1)z_{n,m}(x,v)F(dx),\\
		{B''}_{n,m}(v)=\int_{|x|>\beta_{n}}(e^{(m-1)z_{n,m}(x,v)}-1-(m-1)z_{n,m}(x,v))F(dx).
		\end{array}\right.
	\end{align*}
First, by same arguments as above, we get $|{B''}_{n,m}(v)|\leq \frac{C}{n^2}|v|^{2\alpha}u_{n,m}^{2\alpha}$, hence, 
	$
	{B''}_{n,m}(v)\underset{n\rightarrow\infty}{\longrightarrow}0.
	$
	Second, note that ${B'}_{n,m}(v)=\int (e^{ivx}-1-ivx)K_{n,m}(dx)$, where 
	\begin{align*}
		K_{n,m}(h)&=\frac{m-1}{nm}\int_{|x|>\beta_{n}}\int_{|y|\leq\beta_{n}}h(u_{n,m}xy)F(dy)F(dx)
	\end{align*}
with some function $h$.
	We will prove that 
	$
	\int (e^{\texttt{i}vx}-1-\texttt{i}vx)K_{n,m}(dx)\underset{n\rightarrow\infty}{\longrightarrow}\int (e^{\texttt{i}vx}-1-\texttt{i}vx)K(dx),
	$
	with $$K(dx)=\frac{\alpha}{2}((\theta_+^2+\theta_-^2)\1_{\{x>0\}}+2\theta_+\theta_-\1_{\{x<0\}})\frac{1}{|x|^{1+\alpha}}dx.$$ To do so, we use Theorem \ref{thm:cv}, then it is reduced to prove that $K_{n,m}(h)\underset{n\rightarrow\infty}{\longrightarrow} K(h)$ for $h$ equal either to $h_{\omega}=\1_{(\omega,\infty)}$ for $\omega>0$, or ${h'}_{\omega}=\1_{(-\infty,-\omega)}$ for $\omega>0$, or $h'(x)=x^2\1_{\{|x|\leq1\}}$, or $h''(x)=x\1_{\{|x|>1\}}$.
	\paragraph*{$\bullet$ \textit{First case $h=h_{\omega}$}}
	Since $u_{n,m}\beta_{n}^2\underset{n\rightarrow\infty}{\longrightarrow}\infty$, then for $n$ large enough such that $\frac{\omega}{\beta_{n} u_{n,m}}<\beta_{n}$.
	 Then, we rewrite $K_{n,m}(h_{\omega})=y_{n,m}^1+y_{n,m}^2,$ where
	 \begin{align*}\left\{\begin{array}{l} y_{n,m}^1=\frac{m-1}{nm}\int_{x>\beta_{n}}(\theta_+(\frac{\omega}{u_{n,m}x})-\theta_+(\beta_{n}))F(dx)\\ y_{n,m}^2=\frac{m-1}{nm}\int_{x<-\beta_{n}}(\theta_-(\frac{-\omega}{u_{n,m}x})-\theta_-(\beta_{n}))F(dx).\end{array}\right.\end{align*} 
	 Note that by \eqref{h1} we have $\frac{(\theta_\pm(\beta_n))^2}{n}\le\frac{C}{n\beta_n^{2\alpha}}\underset{n\rightarrow\infty}{\longrightarrow}0$.
	 Further, as $\frac{\omega}{\beta_{n}u_{n,m}x}\underset{n\rightarrow\infty}{\longrightarrow}0$ uniformly on $\{x>\beta_{n}\}$, using \eqref{h2}, $\frac{\rho_+(\beta_{n})}{\log{(1/\beta_{n})}}\underset{n\rightarrow\infty}{\longrightarrow}\alpha\theta_+$ (see $\eqref{eq:3.5}$), we get 
	$
	y_{n,m}^1\sim\frac{m-1}{nm}\int_{x>\beta_{n}}\frac{\theta_+u_{n,m}^{\alpha}x^{\alpha}}{\omega^{\alpha}}F(dx)\sim\frac{\alpha\theta_+^2}{\omega^{\alpha}}\frac{(m-1)u_{n,m}^{\alpha}}{nm}\log{\frac{1}{\beta_{n}}}\underset{n\rightarrow\infty}{\longrightarrow} \frac{\theta_+^2}{2\omega^{\alpha}}$ by the choices $u_{n,m}=\left[\frac{mn}{(m-1)\log{n}}\right]^{1/\alpha}$ and $\beta_n=\frac{\log{n}}{n^{1/(2\alpha)}}$.
 Similarly, we have ${y}_{n,m}^2\underset{n\rightarrow\infty}{\longrightarrow}\frac{\theta_-^2}{2\omega^{\alpha}}$, so $K_{n,m}(h_{\omega})\underset{n\rightarrow\infty}{\longrightarrow} K(h_{\omega})$.
	\paragraph*{$\bullet$ \textit{Second case $h={h'}_{\omega}$}}
	At first, similarly, we rewrite $K_{n,m}(h_{\omega})={y'}_{n,m}^1+{y'}_{n,m}^2,$ where
	\begin{align*}\left\{\begin{array}{l} {y'}_{n,m}^1=\frac{m-1}{nm}\int_{x>\beta_{n}}(\theta_-(\frac{\omega}{u_{n,m}x})-\theta_-(\beta_{n}))F(dx)\\ {y'}_{n,m}^2=\frac{m-1}{nm}\int_{x<-\beta_{n}}(\theta_+(\frac{-\omega}{u_{n,m}x})-\theta_+(\beta_{n}))F(dx).\end{array}\right.\end{align*}
	Using the same arguments as above, we easily get  ${y'}_{n,m}^1\sim \frac{m-1}{nm}\int_{x>\beta_{n}}\frac{\theta_-u_{n,m}^{\alpha}x^{\alpha}}{(-\omega)^{\alpha}}F(dx)\sim\frac{\alpha\theta_-\theta_+u_{n,m}^{\alpha}(m-1)}{(-\omega)^{\alpha}nm}\log{1/\beta_n}\underset{n\rightarrow\infty}{\longrightarrow}\frac{\theta_+\theta_-}{2(-\omega)^{\alpha}}$ and also similarly ${y'}_{n,m}^2\underset{n\rightarrow\infty}{\longrightarrow}\frac{\theta_+\theta_-}{2(-\omega)^{\alpha}}$. Therefore, we have that $K_{n,m}({h'}_{\omega})\underset{n\rightarrow\infty}{\longrightarrow} K({h'}_{\omega})$.
	\paragraph*{$\bullet$ \textit{Third case $h=h'$}}For $n$ large enough such that $\frac{1}{\beta_{n} u_{n,m}}<\beta_{n}$, as $\frac{1}{u_{n,m}|x|}\underset{n\rightarrow\infty}{\longrightarrow}0$ uniformly on $\{|x|>\beta_{n}\}$, by $	c(\beta)\sim\frac{\alpha\theta}{2-\alpha}\beta^{2-\alpha}$ and $\rho(\beta)\sim\alpha\theta\log{(1/\beta)}$ for $\beta\rightarrow0$ (see $\eqref{eq:3.5}$), we have 
	\begin{multline*}
		K_{n,m}(h')=\frac{m-1}{nm}\int_{|x|>\beta_{n}}\int_{|y|\leq\frac{1}{u_{n,m}|x|}}u_{n,m}^2x^2y^2F(dy)F(dx)\\
		\sim\frac{m-1}{nm}\int_{|x|>\beta_{n}}\frac{\alpha\theta}{2-\alpha}u_{n,m}^{\alpha}|x|^{\alpha}F(dx)\sim\frac{\alpha^2\theta^2}{2-\alpha}\frac{(m-1)u_{n,m}^{\alpha}}{nm}\log{\frac{1}{\beta_{n}}}\underset{n\rightarrow\infty}{\longrightarrow}\frac{\alpha\theta^2}{2(2-\alpha)}=K(h').
	\end{multline*}
	\paragraph*{$\bullet$ \textit{Last case $h=h''$}}For $n$ large enough such that $\frac{1}{\beta_{n} u_{n,m}}<\beta_{n}$, we have
	\begin{align*}
		K_{n,m}(h'')&=\frac{m-1}{nm}\int_{|x|>\beta_{n}}\int_{\frac{1}{u_{n,m}|x|}<|y|\leq\beta_{n}}u_{n,m}xyF(dy)F(dx)\\
		&=\frac{(m-1)u_{n,m}}{nm}\int_{|x|>\beta_{n}}x\left(\int_{|y|>\frac{1}{u_{n,m}x}}yF(dy)-{d'}_n\right)F(dx).
	\end{align*}
At first, note that by $	{d'}_n\le C\beta_n^{1-\alpha}$ (see  \eqref{eq:3.3}), we have $\frac{(m-1)u_{n,m}{d'}_n^2}{mn}\underset{n\rightarrow\infty}{\longrightarrow}0$ and $K_{n,m}(h'')\sim\frac{(m-1)u_{n,m}}{nm}\int_{|x|>\beta_{n}}\int_{|y|>\frac{1}{u_{n,m}x}}xyF(dy)F(dx)$ for $n\rightarrow\infty$. Then, as $\frac{1}{u_{n,m}|x|}\underset{n\rightarrow\infty}{\longrightarrow}0$ uniformly on $\{|x|>\beta_{n}\}$, using $	{d'}(\beta)\sim\frac{\alpha}{\alpha-1}\theta'\beta^{1-\alpha}$, $\rho_+(\beta)\sim\alpha\theta_+\log{(1/\beta)}$ and $\rho_-(\beta)\sim\alpha\theta_-\log{(1/\beta)}$ for $\beta\rightarrow0$ (see $\eqref{eq:3.5}$),  we get for $n\to\infty$
	\begin{align*}
	K_{n,m}(h'')&
	\sim \frac{(m-1)\alpha\theta'u_{n,m}^{\alpha}}{nm(\alpha-1)}\big(\int_{x>\beta_{n}}|x|^{\alpha}F(dx)-\int_{x<-\beta_{n}}|x|^{\alpha}F(dx)\big)\\
	&\sim\frac{m-1}{nm}\frac{\alpha^2{\theta'}^2}{\alpha-1}u_{n,m}^{\alpha}\log{\frac{1}{\beta_{n}}}\underset{n\rightarrow\infty}{\longrightarrow}\frac{\alpha{\theta'}^2}{2(\alpha-1)}=K(h'').
\end{align*}
Therefore, we get
	$$
	{B'}_{n,m}(v)\underset{n\rightarrow\infty}{\longrightarrow}\int\frac{\alpha}{2}[(\theta_+^2+\theta_-^2)\1_{\{x>0\}}+2\theta_+\theta_-\1_{\{x<0\}}]\frac{1}{|x|^{1+\alpha}}(e^{\texttt{i}vx}-1-\texttt{i}vx)dx.
$$
	Finally, we have
	\begin{multline*}
		\mathbb{E}(\exp{(\texttt{i}(u\overline{Y}^n_1(2)+v{\Gamma'}^n_1(1,1)))})\underset{n\rightarrow\infty}{\longrightarrow}\exp\left\{\texttt{i}ub+\int F(dx)(e^{\texttt{i}ux}-1-\texttt{i}ux\1_{\{|x|\leq 1\}})+\right.\\\left.\frac{\alpha}{2}[(\theta_+^2+\theta_-^2)\1_{\{x>0\}}+2\theta_+\theta_-\1_{\{x<0\}}]\frac{1}{|x|^{1+\alpha}}(e^{\texttt{i}vx}-1-\texttt{i}vx)dx\right\},
	\end{multline*} 
which completes the proof.
	\end{proof}
\subsection{Conclusion} The challenge ahead is to apply these approximations to obtain a central limit theorem type for the multilevel Monte Carlo method  with the stochastic differential equation \eqref{eq:1.1} driven by a pure jump L\'evy process, in the spirit of the ones established by Ben Alaya and Kebaier  \cite{k} and Ben Alaya, Kebaier and Ngo  \cite{BAKN} for the case of a diffusion process, Dereich and Li  \cite{DereichLi} for the case of jump-diffusion process and Giorgi et al.\ \cite{Gior} for the case of nested Multilevel Monte Carlo. We keep this work for a future research.  

\appendix
\section{Proof of lemmas \ref{lem:rest1}, \ref{lem:rest2}, \ref{lem:rest3}, \ref{lem:rest5}  concerning the rest terms}\label{app:A}
Note that throughout this section, $C$ is a generic constant (may depending on $m$) which can be changed from line to line.
\subsection{Proof of Lemma \ref{lem:rest1}} 
Here, we prove that 
the sequences of processes $(\overline{Y}^n(1))_{n\ge0}$ and  $(\mathcal{R}^{n,m})_{n\ge0}$ converge uniformly in probability to $0$ as $n\rightarrow\infty$. First, instead of considering the form  $\overline{Y}^{n}_t(1)=\sum_{i=1}^{[nt]}\sum_{k=1}^m(M^{\beta_n}_{t_i^k,t_i^{k+1}}+\sum_{j\geq2}\Delta Y_{T^{\beta_n}_{ \s j}( t_i^k)}\1_{\{K({\s t_i^k})\geq j\}})$, it is enough to prove that for each $k\in\{1,\dots,m\}$ the triangular arrays with generic terms $y^{n,m}_{i,k}(1,1)=M^{\beta_n}_{t_i^k,t_i^{k+1}}$ and $y^{n,m}_{i,k}(1,2)=\sum_{j\geq2}\Delta Y_{T^{\beta_n}_{ \s j}( t_i^k)}\1_{\{K({\s t_i^k})\geq j\}}$ converge in probability to $0$ as $n\rightarrow\infty$. 
By Property \ref{I}, \eqref{eq:3.13} and Lemma \ref{lem:smallj}, for the first one,  we have
	\begin{align}\label{eq:tighty1}
			\mathbb{E}(y^{n,m}_{i,k}(1,1)|\mathcal{F}_{t^{1}_i})
			=0,\quad
			\mathbb{E}((y^{n,m}_{i,k}(1,1))^2|\mathcal{F}_{t^{1}_i})=\frac{c_{n}}{nm}
	\end{align}
 and therefore we conclude  using  \eqref{eq:3.13}, $c_n\le C\beta_n^{2-\alpha}$, the criteria \eqref{eq:2.8} and Lemma \ref{lem:tightrest}. For the second one,  we have
	\begin{multline}\label{eq:tighty2}
	\mathbb{E}(|y^{n,m}_{i,k}(1,2)||\mathcal{F}_{t^{1}_i})
		\le\frac{1}{\theta(\beta_{n})}\int_{|x|>\beta_n}|x|F(dx)\sum_{j\ge2}\mathbb{P}(K({\s t_i^k})\geq j|\mathcal{F}_{t^{1}_i})\\=\frac{\delta_n}{\theta(\beta_n)}(\mathbb{E}(K({\s t_i^k})|\mathcal{F}_{t^{1}_i})-\mathbb{P}(K({\s t_i^k})\ge1|\mathcal{F}_{t^{1}_i}))=\frac{\delta_n}{\theta(\beta_n)}(\lambda_{n,m}+e^{-\lambda_{n,m}}-1)\le \frac{\delta_n\lambda_{n,m}^2}{\theta(\beta_{n})}= \frac{\delta_n\lambda_{n,m}}{nm}
\end{multline}
and therefore we conclude by using \eqref{eq:3.13},  the boundedness of $\delta_n$ in case \ref{1}, the criteria \eqref{eq:2.7} and Lemma \ref{lem:tightrest}.
 Thus, it is clear that for case \ref{1},  $\overline{Y}^n(1)\stackrel{\mathbb{P}}{\rightarrow}0$.
Now, by using the formula of the rest term given by \eqref{R1}, we have 
\begin{multline*}
	\mathcal{R}^{n,m}_t	
	={\Gamma}^n_t(2,2)+{\Gamma}^n_t(2,3)+\sum_{i=3}^5\int_0^{\eta_n(t)}ff'(X_{\eta_n(s-)}^n)d\Gamma_s^n(i)\\+	\overline{\Gamma}^n_t(1,2)+\int_0^{\eta_n(t)}k(X^n_{\eta_n(s-)},Y_{\eta_{nm}(s-)}-Y_{\eta_n(s-)})d\overline{\Gamma}^n_s(3)+\overline{\Gamma}^n_t(4)+\overline{\Gamma}^n_t(5).
\end{multline*}
According to Theorem \ref{thm:ut} (iii), in order to prove the convergence of the third and the fifth terms in  the r.h.s. of the above relation, we only need to prove the convergence of each $\Gamma^n(i)$, $i\in\{3,4,5\}$ and $\overline{\Gamma}^n(3)$ to $0$ as $n\rightarrow\infty$. Now, we  prove that each term converges uniformly in probability to $0$ when $n\rightarrow\infty$.\\
\underline{\textbf{The term ${\Gamma}^n_t(2,2)$}}:
Let us rewrite ${\Gamma}^n_t(2,2)=\sum_{k=2}^m\sum_{i=1}^{[nt]}{\zeta}^n_{i,k}(2,2)$ with
\begin{align*}
	{\zeta}^n_{i,k}(2,2)&=\frac{u_{n,m}d_n}{nm}ff'(X_{t_i^1}^n)((k-1)\sum_{h=2}^{K(t_i^k)}\Delta Y_{T^{\beta_n}_h(t_i^k)}+\sum_{h=2}^{K(t^{1}_i,t_i^k)}\Delta Y_{T^{\beta_n}_h(t^{1}_i,t_i^k)}).
\end{align*}
For each $k\in\{2,\hdots,m\}$,   by property \ref{I}, \eqref{eq:3.13} and the boundedness of $ff'$, similarly to the calculations in \eqref{eq:tighty2}, we have
\begin{align*}
	\mathbb{E}(|{\zeta}^n_{i,k}(2,2)||\mathcal{F}_{t^{1}_i})
	\le&C\frac{u_{n,m}|d_n|\delta_n}{n\theta(\beta_n)}[(k-1)(\lambda_{n,m}+e^{-\lambda_{n,m}}-1)+(k-1)\lambda_{n,m}+e^{-(k-1)\lambda_{n,m}}-1]\\
		\le& C\frac{u_{n,m}|d_n|\delta_n\lambda_{n,m}}{n^2}.
\end{align*}	
 Then we conclude by using \eqref{eq:3.13}, the boundedness of $|d_n|$ and $\delta_n$ in case \ref{1}, $u_{n,m}=\frac{nm}{m-1}$, the criteria \eqref{eq:2.7} and Lemma \ref{lem:tightrest}.\\
\underline{\textbf{The term ${\Gamma}^n_t(2,3)$}}:
Let us rewrite ${\Gamma}^n_t(2,3)=\sum_{k=2}^m\sum_{i=1}^{[nt]}{\zeta}^n_{i,k}(2,3)$ with
$${\zeta}^n_{i,k}(2,3)=\frac{u_{n,m}d_n}{nm}ff'(X_{t_i^1}^n)(k-1)\Delta Y_{T^{\beta_n}_1(t_i^k)}\1_{\{K({\s t_i^k})\geq1\}}\1_{\{K({\s t_i^1,t_i^k})\geq1\}}.$$ 
For each $k\in\{2,\hdots,m\}$, by boundedness of $ff'$, similar as above, we have  $$\mathbb{E}(|{\zeta}^n_{i,k}(2,3)||\mathcal{F}_{t^{1}_i})\le C\frac{u_{n,m}|d_n|\delta_n}{n\theta(\beta_n)}(1-e^{-\lambda_{n,m}})(1-e^{-(k-1)\lambda_{n,m}})\le C\frac{u_{n,m}|d_n|\delta_n\lambda_{n,m}}{n^2}.$$
Then, we conclude similarly that ${\Gamma}^n(2,3)\stackrel{\mathbb{P}}{\rightarrow}0$ by  the criteria \eqref{eq:2.7} and Lemma \ref{lem:tightrest}.\\
\underline{\textbf{The term $\overline{\Gamma}^n_t(1,2)$}}:
Let us rewrite $\overline{\Gamma}^n_t(1,2)=\sum_{k=2}^m\sum_{i=1}^{[nt]}\overline{\zeta}^n_{i,k}(1,2)$ with
\begin{align*}
	\overline{\zeta}^n_{i,k}(1,2)&=\frac{u_{n,m}d_n}{nm}k(X^n_{t_i^1},\Delta Y_{T^{\beta_n}_1(t_i^1,t_i^k)})(\sum_{h=2}^{K(t^{1}_i,t_i^k)}(\Delta Y_{T^{\beta_n}_h(t^{1}_i,t_i^k)})^2+\sum_{\substack{h,h'=2\\h\neq h'}}^{K(t^{1}_i,t_i^k)}\Delta Y_{T^{\beta_n}_h(t^{1}_i,t_i^k)}\Delta Y_{T^{\beta_n}_{h'}(t^{1}_i,t_i^k)}).
\end{align*}
For any fixed $k\in\{2,\hdots,m\}$, by similar calculations as above, using the boundedness of the function $k$ and $\int_{\R}x^2F(dx)<\infty$ (see Remark \ref{rmk:1}), we have
\begin{align*}
	\mathbb{E}(|\overline{\zeta}^n_{i,k}(1,2)||\mathcal{F}_{t^{1}_i})\le& C\frac{u_{n,m}|d_n|}{n\theta(\beta_n)}\mathbb{E}((K(t^{1}_i,t_i^k)-1)\1_{\{K(t^{1}_i,t_i^k)\ge2\}}+\frac{\delta_n^2}{\theta(\beta_n)}K(t^{1}_i,t_i^k)(K(t^{1}_i,t_i^k)-1)|\mathcal{F}_{t^{1}_i})\\
\le& C\frac{u_{n,m}|d_n|}{n^2}(\lambda_{n,m}+\frac{\delta_n^2}{n}).
\end{align*}	
Then, we conclude similarly that $\overline{\Gamma}^n(1,2)\stackrel{\mathbb{P}}{\rightarrow}0$ by  the criteria \eqref{eq:2.7} and Lemma \ref{lem:tightrest}.\\
\underline{\textbf{The term $\int_0^{\eta_n(t)}ff'(X_{\eta_n(s-)}^n)d{\Gamma}^n_s(3)$}}:
This is equal to $\sum_{k=2}^m\sum_{i=1}^{[nt]}\zeta_{i,k}^n(3)$ with
$$
	\zeta_{i,k}^n(3)=u_{n,m}ff'(X_{t_i^1}^n)(M^{\beta_n}_{t_i^k}-M^{\beta_n}_{t^{1}_i})(Y_{t^{k+1}_i}-Y_{t_i^k}).	
$$
For a fixed $k\in\{2,\hdots,m\}$, by the boundedness of $ff'$, Property \ref{I},  \eqref{eq:3.13}, the inequality $(a^2+b^2)\le 2(a^2+b^2)$, $\int_{\R}x^2F(dx)<\infty$, lemmas \ref{lem:smallj} and \ref{lem:bigj},  we have $\mathbb{E}(\zeta_{i,k}^n(3)|\mathcal{F}_{t^{1}_i})=0,$
\begin{align*}
	\mathbb{E}(|\zeta_{i,k}^n(3)|^2|\mathcal{F}_{t^{1}_i})\leq &Cu_{n,m}^2\mathbb{E}((M^{\beta_n}_{t_i^k}-M^{\beta_n}_{t^{1}_i})^2|\mathcal{F}_{t^{1}_i})\mathbb{E}((Y^{\beta_n}_{t_i^{k+1}}-Y^{\beta_n}_{t^{k}_i})^2+(N^{\beta_n}_{t_i^{k+1}}-N^{\beta_n}_{t^{k}_i})^2|\mathcal{F}_{t^{1}_i})\\
	\leq& C \frac{u_{n,m}^2c_n}{n} (\frac{c_n}{nm}+\frac{d_n^2}{n^2m^2}+\frac{1}{nm})
	\leq C \frac{u_{n,m}^2c_n}{n^2} (c_n+\frac{d_n^2}{n}+1).
	\end{align*}
Then as $\alpha<1$, we conclude  using the boundedness of $|d_n|$, $c_n\le C \beta_n^{2-\alpha}$, $\beta_n=\frac{(\log{n})^{2}}{n}$, $u_{n,m}=\frac{nm}{m-1}$, the criteria \eqref{eq:2.8} and Lemma \ref{lem:tightrest}.
 Therefore, we have $\int_0^{\eta_n(.)}ff'(X_{\eta_n(s-)}^n)d{\Gamma}^n_s(3)\stackrel{\mathbb{P}}{\rightarrow}0$.\\
\underline{\textbf{The term $\int_0^{\eta_n(t)}ff'(X_{\eta_n(s-)}^n)d{\Gamma}^n_s(4)$}}:
This is bounded by $C\sum_{k=2}^m\sum_{i=1}^{[nt]}\zeta_{i,k}^n(4)$  with
$$
	\zeta_{i,k}^n(4)=u_{n,m}|N^{\beta_n}_{t_i^k}-N^{\beta_n}_{t^{1}_i}||N^{\beta_n}_{t^{k+1}_i}-N^{\beta_n}_{t_i^k}|.	
$$
 Let $k\in\{1,\dots,m\}$ be fixed. Using the independence of the increments $N^{\beta_n}_{t_i^k}-N^{\beta_n}_{t^{1}_i}$ and $N^{\beta_n}_{t^{k+1}_i}-N^{\beta_n}_{t_i^k}$ and  Lemma \ref{lem:bigj}, we have
\begin{align*}
\mathbb{E}(|\zeta_{i,k}^n(4)|\1_{\{|\zeta_{i,k}^n(4)|\leq1\}}|\mathcal{F}_{t^{1}_i})=&  \frac{u_{n,m}}{nm}\int_{|x|>\beta_n}|x|\mathbb{E}(|N^{\beta_n}_{t_i^k}-N^{\beta_n}_{t_i^1}|\1_{\left\{|N^{\beta_n}_{t_i^k}-N^{\beta_n}_{t_i^1}|\leq\frac{1}{u_{n,m|x|}}\right\}})F(dx)\\
	\le&\frac{Cu_{n,m}\delta_n}{n^2}\int_{\frac{1}{u_{n,m}\beta_n}>|y|>\beta_n}|y|F(dy).
	\end{align*}
Then   using the boundedness of $\delta_n$,  $\beta_n=\frac{(\log{n})^{2}}{n}$, $u_{n,m}=\frac{nm}{m-1}$ and Lebesgue's theorem, we easily check that the two first points of the criteria \eqref{eq:2.9} are satisfied.
In order to prove the third point of this criteria, noticing that $u_{n,m}\beta_n^2{\rightarrow}0$ when ${n\rightarrow\infty}$, then for $y>1$ we have $u_{n,m}\beta_n^2<y\Leftrightarrow\beta_n<\frac{y}{u_{n,m}\beta_n}$.	Similarly, we have
\begin{align*}
	\mathbb{P}(\zeta_{i,k}^n(4)>y|\mathcal{F}_{t^{1}_i})
	\le& \frac{C}{n^2}\int_{|x|>\beta_n}F(dx)\int_{|z|>\beta_n\vee\frac{y}{u_{n,m}|x|}}\hspace{-1cm}F(dz)
	=\frac{C}{n^2}\int_{|x|>\beta_n}F(dx)\theta\left(\frac{y}{u_{n,m}|x|}\vee \beta_n\right)\\
	\le&\frac{C}{n^2}\left(\int_{|x|>\frac{y}{u_{n,m}\beta_n}}F(dx)\theta\left( \beta_n\right)+\int_{\frac{y}{u_{n,m}\beta_n}>|x|>\beta_n}F(dx)\theta\left(\frac{y}{u_{n,m}|x|}\right)\right).
	\end{align*}
Now, using \eqref{h1} and as $\frac{y}{u_{n,m}\beta_n}<1$ for $n$ large enough, we get
\begin{align*}
	\mathbb{P}(\zeta_{i,k}^n(4)>y|\mathcal{F}_{t^{1}_i})\le&\frac{C}{n^2}\left(\theta\left(\frac{y}{u_{n,m}\beta_n}\right)\theta\left( \beta_n\right)+\frac{u_{n,m}^{\alpha}}{y^{\alpha}}\int_{\frac{y}{u_{n,m}\beta_n}>|x|>\beta_n}\hspace{-1cm}|x|^{\alpha}F(dx)\right)
	\le\frac{Cu_{n,m}^{\alpha}}{n^2y^{\alpha}}\left(1+\rho_n\right).
\end{align*}
Then as $\alpha<1$,  using  \eqref{eq:3.13}, $\rho_n\le C\log{\frac{1}{\beta_n}}$,  $\beta_n=\frac{(\log{n})^{2}}{n}$, $u_{n,m}=\frac{nm}{m-1}$, the third point of criteria \eqref{eq:2.9} is satisfied.   Therefore by Lemma \ref{lem:tightrest}, we have $\int_0^{\eta_n(.)}ff'(X_{\eta_n(s-)}^n)d{\Gamma}^n_s(4)\stackrel{\mathbb{P}}{\rightarrow}0$ .\\
\underline{\textbf{The term $\int_0^{\eta_n(t)}ff'(X_{\eta_n(s-)}^n)d{\Gamma}^n_s(5)$}}:
This is equal to $\sum_{k=2}^m\sum_{i=1}^{[nt]}\zeta_{i,k}^n(5)$ where
$$
\zeta_{i,k}^n(5)=u_{n,m}ff'(X_{t_{i}^1}^n)\left[(A^{\beta_n}_{t_i^k}-A^{\beta_n}_{t^{1}_i})+(N^{\beta_n}_{t_i^k}-N^{\beta_n}_{t^{1}_i})\right](M^{\beta_n}_{t^{k+1}_i}-M^{\beta_n}_{t_i^k}).
$$
For any fixed $k\in\{1,\dots,m\}$, by the boundedness of $ff'$, the independence structure, property \ref{I},  \eqref{eq:3.13}, the inequality $(a+b)^2\le 2(a^2+b^2)$, $\int_{\R}x^2F(dx)<\infty$ (see Remark \ref{rmk:1}), lemmas \ref{lem:smallj} and \ref{lem:bigj}, we have   $\mathbb{E}(\zeta_{i,k}^n(5)|\mathcal{F}_{t^{1}_i})=0,$
\begin{align*}
	\mathbb{E}(|\zeta_{i,k}^n(5)|^2|\mathcal{F}_{t^{1}_i})\leq &Cu_{n,m}^2\mathbb{E}((M^{\beta_n}_{t^{k+1}_i}-M^{\beta_n}_{t_i^k})^2)\left[\frac{d_n^2(k-1)^2}{n^2m^2}+\mathbb{E}((N^{\beta_n}_{t_i^k}-N^{\beta_n}_{t^{1}_i})^2)\right]\\
	\leq&\frac{C u_{n,m}^2c_n}{n} (\frac{d_n^2}{n^2}+\frac{1}{n}).
\end{align*}
Then as $\alpha<1$, we conclude  using the boundedness of $|d_n|$, $c_n\le C \beta_n^{2-\alpha}$,  $u_{n,m}=\frac{nm}{m-1}$, the criteria \eqref{eq:2.8} and Lemma \ref{lem:tightrest}.
Therefore, we have $\int_0^{\eta_n(.)}ff'(X_{\eta_n(s-)}^n)d{\Gamma}^n_s(5)\stackrel{\mathbb{P}}{\rightarrow}0$.\\
\underline{\textbf{The term $\overline{\Gamma}^n_t(4)$}}:
We recall that
\begin{multline*}
\overline{\Gamma}^n_t(4)=u_{n,m}\int_0^{\eta_n(t)}(k(X^n_{\eta_n(s-)},Y_{\eta_{nm}(s-)}-Y_{\eta_n(s-)})-k(X^n_{\eta_n(s-)},N^{\beta_{n}}_{\eta_{nm}(s-)}-N^{\beta_{n}}_{\eta_n(s-)}))\\(N^{\beta_n}_{\eta_{nm}(s-)}-N^{\beta_n}_{\eta_{n}(s-)})^2dA^{\beta_n}_s.
\end{multline*}
 Since $\frac{\partial k}{\partial y}(x,y)$ is bounded on  $\mathbb{R}^2$, it is enough to prove that for each $k\in\{1,\dots,m\}$, the triangular array with generic term $\overline{\zeta}^n_{i,k}(4)$ converges to $0$ when $n\rightarrow\infty$ where $$\overline{\zeta}^n_{i,k}(4)=\frac{u_{n,m}|d_n|}{nm}|Y^{\beta_n}_{t^{k+1}_i}-Y^{\beta_n}_{t^{1}_i}|(N^{\beta_n}_{t^{k+1}_i}-N^{\beta_n}_{t^{1}_i})^2.$$
By property \ref{I}, the independence between $Y^{\beta_n}_{t^{k+1}_i}-Y^{\beta_n}_{t^{1}_i}$ and $N^{\beta_n}_{t^{k+1}_i}-N^{\beta_n}_{t^{1}_i}$, Cauchy-Schwarz's inequality, $\int_{\R} x^2F(dx)<\infty$, Lemma \ref{lem:bigj} and Lemma \ref{lem:smallj}, we have
 \begin{align*}
 	\mathbb{E}(|\overline{\zeta}^n_{i,k}(4)||\mathcal{F}_{t^{1}_i})\le& C\frac{u_{n,m}|d_n|}{nm}(\frac{|d_n|}{nm}+\frac{\sqrt{c_n}}{\sqrt{nm}})\frac{1}{nm}\int_{|x|>\beta_n}x^2F(dx)\le C\frac{u_{n,m}|d_n|}{n^2}(\frac{|d_n|}{n}+\frac{\sqrt{c_n}}{\sqrt{n}}).
 	\end{align*}
Then we conclude  using the boundedness of $|d_n|$,  $c_n$,  $u_{n,m}=\frac{nm}{m-1}$, the criteria \eqref{eq:2.7} and Lemma \ref{lem:tightrest}. Therefore, we have $\overline{\Gamma}^n(4)\stackrel{\mathbb{P}}{\rightarrow}0$.  \\
\underline{\textbf{The term $\overline{\Gamma}^n_t(5)$}}:
We recall that
$$
		\overline{\Gamma}^n_t(5)=\frac{u_{n,m}d_n}{nm}\sum_{i=1}^{[nt]}\sum_{k=2}^m(k(X^n_{t^{1}_i},N^{\beta_{n}}_{t_i^k}-N^{\beta_{n}}_{t^{1}_i})-k(X^n_{t^{1}_i},\Delta Y_{T_1^{\beta_n}(t^{1}_i,t_i^k)}))(N^{\beta_n}_{t_i^k}-N^{\beta_n}_{t^{1}_i})^2.
$$
 Similarly as for the term $\overline{\Gamma}^n(4)$, it is enough to prove that for each $k\in\{1,\dots,m\}$, the triangular array with generic term $\overline{\zeta}^n_{i,k}(5)$ converges to $0$ when $n\rightarrow\infty$ with $$\overline{\zeta}^n_{i,k}(5)=\frac{u_{n,m}|d_n|}{nm}\left|\sum_{j=2}^{K(t^{1}_i,t_i^k)}\Delta Y_{T^{\beta_n}_j(t^{1}_i,t_i^k)}\right|(N^{\beta_n}_{t_i^k}-N^{\beta_n}_{t^{1}_i})^2.$$
By using Cauchy-Schwarz's inequality,  property \ref{I}, Lemmas \ref{lem:bigj} and \ref{lem:smallj}, $\int_{\R} x^4F(dx)<\infty$ (see Remark \ref{rmk:1}) and the calculations developped for the term $\overline{\Gamma}^n(1,2)$, we have
\begin{align*}
	\mathbb{E}(|\overline{\zeta}^n_{i,k}(5)||\mathcal{F}_{t^{1}_i})&\le C\frac{u_{n,m}|d_n|}{n}	\left[\mathbb{E}((\sum_{j=2}^{K(t^{1}_i,t_i^k)}\Delta Y_{T^{\beta_n}_j(t^{1}_i,t_i^k)})^2|\mathcal{F}_{t^{1}_i})\times\mathbb{E}((N^{\beta_n}_{t_i^k}-N^{\beta_n}_{t^{1}_i})^4)\right]^{1/2}\\
	&\le C\frac{u_{n,m}|d_n|}{n^{3/2}}	\left[\mathbb{E}(\sum_{j=2}^{K(t^{1}_i,t_i^k)}(\Delta Y_{T^{\beta_n}_j(t^{1}_i,t_i^k)})^2+\sum_{\substack{j,j'=2\\j\neq j'}}^{K(t^{1}_i,t_i^k)}\Delta Y_{T^{\beta_n}_j(t^{1}_i,t_i^k)}\Delta Y_{T^{\beta_n}_{j'}(t^{1}_i,t_i^k)}|\mathcal{F}_{t^{1}_i})\right]^{1/2}\\&\le C\frac{u_{n,m}|d_n|}{n^{3/2}}\left[\frac{1}{n}(\lambda_{n,m}+\frac{\delta_n^2}{n})\right]^{1/2}.
\end{align*}
Then we conclude  using the boundedness of $|d_n|$,  $\delta_n$,  $u_{n,m}=\frac{nm}{m-1}$, $\lambda_{n,m}\rightarrow0$ when $n\rightarrow\infty$, the criteria \eqref{eq:2.7} and Lemma \ref{lem:tightrest}.  Therefore we get $\overline{\Gamma}^n(5)\stackrel{\mathbb{P}}{\rightarrow}0$.\\
\underline{\textbf{The term $\int_0^{\eta_n(t)}k(X^n_{\eta_n(s-)},Y_{\eta_{nm}(s-)}-Y_{\eta_n(s-)})d\overline{\Gamma}^n_s(3)$}}:
Since $k(x,y)$ is bounded on $\R^2$, using the inequality $(a+b+c)^2\le 4(a^2+b^2+c^2)$, it is enough to prove that for $k\in\{1,\dots,m\}$, the following eight  triangular arrays with generic terms $\{\overline{\zeta}_{i,k}^n(3,j),j\in\{1,\dots,8\}$  converge to $0$ as $n\rightarrow\infty$ with
\begin{align}\label{eq:RC1}
	\left\{\begin{array}{ll}
	\overline{\zeta}_{i,k}^n(3,1)=u_{n,m}(A^{\beta_n}_{t_i^k}-A^{\beta_n}_{t^{1}_i})^2|A^{\beta_n}_{t^{k+1}_i}-A^{\beta_n}_{t_i^k}|,&
	\overline{\zeta}_{i,k}^n(3,2)=u_{n,m}(A^{\beta_n}_{t_i^k}-A^{\beta_n}_{t^{1}_i})^2	|M^{\beta_n}_{t^{k+1}_i}-M^{\beta_n}_{t_i^k}|\\
	\overline{\zeta}_{i,k}^n(3,3)=u_{n,m}(A^{\beta_n}_{t_i^k}-A^{\beta_n}_{t^{1}_i})^2	|N^{\beta_n}_{t^{k+1}_i}-N^{\beta_n}_{t_i^k}|,&
	\overline{\zeta}_{i,k}^n(3,4)=u_{n,m}(M^{\beta_n}_{t_i^k}-M^{\beta_n}_{t^{1}_i})^2	|A^{\beta_n}_{t^{k+1}_i}-A^{\beta_n}_{t_i^k}|\\
	\overline{\zeta}_{i,k}^n(3,5)=u_{n,m}(N^{\beta_n}_{t_i^k}-N^{\beta_n}_{t^{1}_i})^2	|M^{\beta_n}_{t^{k+1}_i}-M^{\beta_n}_{t_i^k}|,&
	\overline{\zeta}_{i,k}^n(3,6)=u_{n,m}(M^{\beta_n}_{t_i^k}-M^{\beta_n}_{t^{1}_i})^2	|N^{\beta_n}_{t^{k+1}_i}-N^{\beta_n}_{t_i^k}|\\
	\overline{\zeta}_{i,k}^n(3,7)=u_{n,m}(M^{\beta_n}_{t_i^k}-M^{\beta_n}_{t^{1}_i})^2	|M^{\beta_n}_{t^{k+1}_i}-M^{\beta_n}_{t_i^k}|,&
	\overline{\zeta}_{i,k}^n(3,8)=u_{n,m}(N^{\beta_n}_{t_i^k}-N^{\beta_n}_{t^{1}_i})^2	|N^{\beta_n}_{t^{k+1}_i}-N^{\beta_n}_{t_i^k}|.
\end{array}\right.
\end{align}
For the first four triangular arrays, as $A^{\beta_n}_t=d_nt$ is  deterministic, applying Cauchy-Schwarz's inequality  and Lemma \ref{lem:smallj} for the increment $|M^{\beta_n}_{t^{k+1}_i}-M^{\beta_n}_{t_i^k}|$ and 
Lemma \ref{lem:bigj} for the increment $N^{\beta_n}_{t^{k+1}_i}-N^{\beta_n}_{t_i^k}$, it is easy to check 
\begin{align}\label{eq:RC1.}
	\left\{\begin{array}{ll}
		\mathbb{E}(|	\overline{\zeta}_{i,k}^n(3,1)||\mathcal{F}_{t^{1}_i})
		\le C\frac{u_{n,m}|d_n|^3}{n^3},&
		\mathbb{E}(|	\overline{\zeta}_{i,k}^n(3,2)||\mathcal{F}_{t^{1}_i})
		{\le} C\frac{u_{n,m}|d_n|^2\sqrt{c_n}}{n^2\sqrt{n}}\\
		\mathbb{E}(|	\overline{\zeta}_{i,k}^n(3,3)||\mathcal{F}_{t^{1}_i})
		{\le} C\frac{u_{n,m}|d_n|^2\delta_n}{n^3},&
		\mathbb{E}(|	\overline{\zeta}_{i,k}^n(3,4)||\mathcal{F}_{t^{1}_i})
		\le C\frac{u_{n,m}|d_n|c_n}{n^2}.
	\end{array}\right.
\end{align}
Then, for our choices of $u_{n,m}$ and $\beta_n$, by the boundedness of $d_n$ and $\delta_n$, $c_n\le C\beta_n^{2-\alpha}$ with $\beta_n=\frac{(\log{n})^2}{n}$, the application of the criteria \eqref{eq:2.7} and Lemma \ref{lem:tightrest} is straightforward for these first four triangular arrays. 
Now, by using the independence between the increments of $M^{\beta_n}$ and $N^{\beta_n}$, applying Cauchy-Schwarz's inequality  and Lemma \ref{lem:smallj} for the estimation of the increment $|M^{\beta_n}_{t^{k+1}_i}-M^{\beta_n}_{t_i^k}|$, 
Lemma \ref{lem:bigj} and $\int_{\R}x^2F(dx)<\infty$, we have
\begin{align}\label{eq:RC1.1}
		\begin{array}{ll}
		\mathbb{E}(|	\overline{\zeta}_{i,k}^n(3,5)||\mathcal{F}_{t^{1}_i})
\le C \frac{u_{n,m}\sqrt{c_n}}{n\sqrt{n}},&
		\mathbb{E}(|	\overline{\zeta}_{i,k}^n(3,6)||\mathcal{F}_{t^{1}_i})
	\le C \frac{u_{n,m}c_n\delta_n}{n^2}.
\end{array}
\end{align}
By the same arguments as above, we get $nc_n$ converges to $0$ as $n\rightarrow\infty$ and then we apply the criteria \eqref{eq:2.7} and Lemma \ref{lem:tightrest} to get the convergence of these two triangular arrays to $0$. For the seventh triangular array, we use the independence between the two increments $M^{\beta_n}_{t^{k+1}_i}-M^{\beta_n}_{t_i^k}$ and $M^{\beta_n}_{t^{k}_i}-M^{\beta_n}_{t_i^1}$ and Cauchy-Schwarz's inequality to obtain
\begin{align}\label{eq:RC1.13}
		\mathbb{E}(|	\overline{\zeta}_{i,k}^n(3,7)||\mathcal{F}_{t^{1}_i})
		\le C \frac{u_{n,m}c_n^{3/2}}{n\sqrt{n}}.
\end{align}
Then we conclude similarly the convergence of this triangular arrays to $0$.
Finally, concerning the last triangular array, we use  similar calculations and arguments as for the term $\int_0^{\eta_n(t)}ff'(X_{\eta_{n}(s-)}^n)d\Gamma^n_s(4)$ and we get
\begin{align*}
	\mathbb{E}(|\overline{\zeta}_{i,k}^n(3,8)|\1_{\{|\overline{\zeta}_{i,k}^n(3,8)|\leq1\}}|\mathcal{F}_{t^{1}_i})
	=&  \frac{u_{n,m}}{nm}\int_{|x|>\beta_n}|x|\mathbb{E}(|N^{\beta_n}_{t_i^k}-N^{\beta_n}_{t_i^1}|^2\1_{\left\{|N^{\beta_n}_{t_i^k}-N^{\beta_n}_{t_i^1}|\leq\frac{1}{\sqrt{u_{n,m|x|}}}\right\}})F(dx)\\
	\le&\frac{Cu_{n,m}\delta_n}{n^2}\int_{\frac{1}{\sqrt{u_{n,m}\beta_n}}>|y|>\beta_n}|y|^2F(dy),
\end{align*}
and using  \eqref{h1} and \eqref{eq:3.3}, for $y>1$ and $n$ large enough, we have
\begin{align*}
	\mathbb{P}(\overline{\zeta}_{i,k}^n(3,8)>y|\mathcal{F}_{t^{1}_i})
	\le&\frac{C}{n^2}\int_{|x|>\beta_n}F(dx)\theta\left(\sqrt{\frac{y}{u_{n,m}|x|}}\vee \beta_n\right)\\
	\le&\frac{C}{n^2}\left(\int_{|x|>\sqrt{\frac{y}{u_{n,m}\beta_n}}}F(dx)\theta\left( \beta_n\right)+\int_{\sqrt{\frac{y}{u_{n,m}\beta_n}}>|x|>\beta_n}F(dx)\theta\left(\sqrt{\frac{y}{u_{n,m}|x|}}\right)\right)\\
	\le&\frac{C}{n^2}\left(\theta\left(\sqrt{\frac{y}{u_{n,m}\beta_n}}\right)\theta\left( \beta_n\right)+\frac{u_{n,m}^{\alpha/2}}{y^{\alpha/2}}\int_{\sqrt{\frac{y}{u_{n,m}\beta_n}}>|x|>\beta_n}|x|^{\alpha/2}F(dx)\right)
	\\\le&\frac{Cu_{n,m}^{\alpha/2}}{n^2y^{\alpha/2}}\left(1+\int_{|x|>\beta_n}|x|^{\alpha/2}F(dx)\right)\le\frac{Cu_{n,m}^{\alpha/2}}{n^2y^{\alpha/2}}\left(1+\beta_n^{-\alpha/2}\right).
\end{align*}
Note that, as $\alpha<1$, $\beta_n=\frac{(\log{n})^{2}}{n}$and $u_{n,m}=\frac{nm}{m-1}$, we have that $\frac{u_{n,m}^{\alpha/2}\beta_n^{-\alpha/2}}{n}$ converges to $0$ as $n\rightarrow\infty$. Therefore, we conclude the convergence of our last triangular array by criteria \eqref{eq:2.9} and Lemma \ref{lem:tightrest}.
This completes the proof of Lemma \ref{lem:rest1}.
\subsection{Proof of Lemma \ref{lem:rest2}}
Here, we prove that 
the sequences of processes $(\overline{Y}^n(1))_{n\ge0}$ and  $(\mathcal{R}^{n,m})_{n\ge0}$ converge uniformly in probability to $0$ as $n\rightarrow\infty$. First, instead of considering the form $\overline{Y}^{n}_t(1)$ given in \eqref{Y}, it is enough to prove that for each $k\in\{1,\dots,m\}$ the triangular arrays with generic terms $y^{n,m}_{i,k}(1,1)=M^{\beta_n}_{t_i^k,t_i^{k+1}}$ and $y^{n,m}_{i,k}(1,2)=\sum_{j\geq2}\Delta Y_{T^{\beta_n}_{ \s j}( t_i^k)}\1_{\{K({\s t_i^k})\geq j\}}$ converge uniformly in probability to $0$ as $n\rightarrow\infty$. On the one hand, from the above estimates \eqref{eq:tighty1}, we have
$
	\mathbb{E}(y^{n,m}_{i,k}(1,1)|\mathcal{F}_{t^{1}_i})
	=0,\quad
	\mathbb{E}(|y^{n,m}_{i,k}(1,1)|^2|\mathcal{F}_{t^{1}_i})=\frac{c_{n}}{nm}.
$ Then, $y^{n,m}_{i,k}(1,1)$ satisfies $\eqref{eq:2.8}$ and we conclude using $c_{n}\le C\beta_n^{2-\alpha}$ from \eqref{eq:3.13} and Lemma \ref{lem:tightrest}. On the other hand,  using similar calculations as in \eqref{eq:tighty2}, property \ref{I} and  $\int_{\R}x^2F(dx)<\infty$,
 we have 
	\begin{align}\label{eq:y12}
		\left\{\begin{array}{l}
	\mathbb{E}(y^{n,m}_{i,k}(1,2)|\mathcal{F}_{t^{1}_i})=\frac{{d'}_n}{nm\lambda_{n,m}}(\lambda_{n,m}+e^{-\lambda_{n,m}}-1),\\
		\mathbb{E}(|y^{n,m}_{i,k}(1,2)|^2|\mathcal{F}_{t^{1}_i})=\mathbb{E}(\sum_{j=2}^{K({\s t_i^k})}(\Delta Y_{T^{\beta_n}_{ \s j}( t_i^k)})^2+\sum_{\substack{j,j'=2\\j\neq j'}}^{K({\s t_i^k})}\Delta Y_{T^{\beta_n}_{ \s j}( t_i^k)}\Delta Y_{T^{\beta_n}_{ \s j'}( t_i^k)}|\mathcal{F}_{t^{1}_i})\\\le C \mathbb{E}(\frac{1}{\theta(\beta_n)}K(t_i^k)\1_{\{K(t_i^k)\ge2\}}+\frac{\delta_n^2}{(\theta(\beta_n))^2}K(t_i^k)(K(t_i^k)-1)\1_{\{K(t_i^k)\ge2\}}|\mathcal{F}_{t^{1}_i})
		\le \frac{C}{n}(\lambda_{n,m}+\frac{\delta_n^2}{n}).
	\end{array}\right.
	\end{align}
  Then we conclude using ${d'}_n=0$ by the hypothesis \eqref{h3}, $\delta_n\le C\log{1/\beta_n}$ (see \eqref{eq:3.13}), $\beta_n=\frac{\log{n}}{n}$,  $\lambda_{n,m}\rightarrow0$ as $n\rightarrow\infty$, criteria \eqref{eq:2.8} and Lemma \ref{lem:tightrest}. Therefore, we have $\overline{Y}^n(1)\stackrel{\mathbb{P}}{\rightarrow}0$.
Now, from the formula of the rest term given by \eqref{eq:MR2}, we have 
\begin{align*}
\mathcal{R}^{n,m}_t={\Gamma}_t^{n}(1,2)+\sum_{i=2}^5\Gamma^n_t(i).
\end{align*}
 In what follows, we prove that each term converges uniformly in probability to $0$.\\
\underline{\textbf{The term ${\Gamma}_t^{n}(1,2)$}}: We recall that $${\Gamma}_t^{n}(1,2)=\sum_{i=1}^{[nt]}\sum_{k=2}^m\sum_{j=1}^{k-1}({\zeta}_{i,k,j}^n(1,2)+{\zeta'}_{i,k,j}^n(1,2)+{\zeta''}_{i,k,j}^n(1,2)),$$
with
\begin{align*}
	\left\{\begin{array}{l}
		{\zeta}_{i,k,j}^n(1,2)=u_{n,m}ff'(X_{t_i^1}^n)\Delta Y_{T^{\beta_n}_1(t_i^k)}\1_{\{K(t_i^k)\geq1\}}\sum_{h=2}^{K(t^{j}_i)}\Delta Y_{T^{\beta_n}_h(t^{j}_i)},\\{\zeta'}_{i,k,j}^n(1,2)=u_{n,m}ff'(X_{t_i^1}^n)\sum_{h=2}^{K(t_i^k)}\Delta Y_{T^{\beta_n}_h(t_i^k)}\Delta Y_{T^{\beta_n}_1(t^{j}_i)}\1_{\{K(t^{j}_i)\geq1\}},\\{\zeta''}_{i,k,j}^n(1,2)=u_{n,m}ff'(X_{t_i^1}^n)\sum_{h=2}^{K(t_i^k)}\Delta Y_{T^{\beta_n}_h(t_i^k)}\sum_{h=2}^{K(t^{j}_i)}\Delta Y_{T^{\beta_n}_h(t^{j}_i)}.
	\end{array}\right.
\end{align*}
Instead of working with  ${\Gamma}_t^{n}(1,2)$, it is enough to prove that for each $k\in\{2,\dots,m\}$ and $j\in\{1,\dots,k-1\}$, the three triangular arrays with  generic terms ${\zeta}_{i,k,j}^n(1,2)$, ${\zeta'}_{i,k,j}^n(1,2)$ and ${\zeta''}_{i,k,j}^n(1,2)$ converge uniformly in probability to $0$ as $n\rightarrow\infty$.
Concerning ${\zeta}_{i,k,j}^n(1,2)$, on the one hand, by property \ref{I} and hypothesis \eqref{h3}, we have
		$
	\mathbb{E}({\zeta}_{i,k,j}^n(1,2)\1_{\{|{\zeta}_{i,k,j}^n(1,2)|\le1\}}|\mathcal{F}_{t^{1}_i})=0.$
On the other hand, as $ff'$ is bounded, using property \ref{I}, the inequalities 
$1-e^{\lambda_{n,m}}\le\lambda_{n,m}$ and    
$(\sum_{i=1}^n|x_i|)^{\alpha}\le\sum_{i=1}^n|x_i|^{\alpha}$ for $x_i\in\R$ and $\alpha\le 1$ and $c(\beta)\le C\beta^{2-\alpha}$, we have
	\begin{align*}
	&\mathbb{E}(|{\zeta}_{i,k,j}^n(1,2)|^2\1_{\{|{\zeta}_{i,k,j}^n(1,2)|\le1\}}|\mathcal{F}_{t^{1}_i})\\&= \frac{u_{n,m}^2(1-e^{-\lambda_{n,m}})}{\theta(\beta_n)}\mathbb{E}((\sum_{h=2}^{K(t^{j}_i)}\Delta Y_{T^{\beta_n}_h(t^{j}_i)})^2\int_{|x|>\beta_{n}}\hspace{-0.5cm}x^2\1_{\{|u_{n,m}ff'(X_{t_{i}^1})x\sum_{h=2}^{K(t^{j}_i)}\Delta Y_{T^{\beta_n}_h(t^{j}_i)}|\le1\}}F(dx)|\mathcal{F}_{t^{1}_i})\\
	&\le\frac{Cu_{n,m}^2}{n}\mathbb{E}((\sum_{h=2}^{K(t^{j}_i)}\Delta Y_{T^{\beta_n}_h(t^{j}_i)})^2c(\frac{1}{u_{n,m}|ff'(X_{t_{i}^1})||\sum_{h=2}^{K(t^{j}_i)}\Delta Y_{T^{\beta_n}_h(t^{j}_i)}|})|\mathcal{F}_{t^{1}_i})\\
	&\le \frac{Cu_{n,m}^\alpha}{n}\mathbb{E}((\sum_{h=2}^{K(t^{j}_i)}|\Delta Y_{T^{\beta_n}_h(t^{j}_i)}|)^{\alpha}|\mathcal{F}_{t^{1}_i})
	\le\frac{Cu_{n,m}^\alpha}{n}\mathbb{E}(\sum_{h=2}^{K(t^{j}_i)}|\Delta Y_{T^{\beta_n}_h(t^{j}_i)}|^{\alpha}|\mathcal{F}_{t^{1}_i})\\&\le\frac{Cu_{n,m}^\alpha}{n}\frac{\mathbb{E}(K(t^{j}_i)\1_{\{K(t^{j}_i)\ge2\}}|\mathcal{F}_{t^{1}_i})\rho_n}{\theta(\beta_n)}\le\frac{Cu_{n,m}^\alpha\rho_n\lambda_{n,m}}{n^2}.
\end{align*}
Now, using similar arguments as above, the inequality $\1_{\{|\sum_{j=1}^\ell a_j|> 1\}}\le \sum_{j=1}^\ell\1_{\{|a_j|> 1/\ell\}}$, $\theta(.)$ is decreasing and hypothesis \eqref{h1}, we obtain for all $y>1$
\begin{align*}
	&\mathbb{P}(|{\zeta}_{i,k,j}^n(1,2)|>y|\mathcal{F}_{t^{1}_i})\le \frac{C}{n}\int_{|x|>\beta_{n}}\mathbb{P}(|ff'(X_{t_{i}^1})u_{n,m}x\sum_{h=2}^{K(t^{j}_i)}\Delta Y_{T^{\beta_n}_h(t^{j}_i)}|>y|\mathcal{F}_{t^{1}_i})F(dx)\\
 &\le\frac{C}{n}\int_{|x|>\beta_{n}}\mathbb{E}(K(t^{j}_i)\1_{\{K(t^{j}_i)\ge2\}}\1_{\{|ff'(X_{t_{i}^1})u_{n,m}x\Delta Y_{T^{\beta_n}_2(t^{j}_i)}|>y/(K(t^{j}_i)-1)\}}|\mathcal{F}_{t^{1}_i})F(dx)\\
 &	\le\frac{C}{n}\int_{|x|>\beta_{n}}\frac{1}{\theta(\beta_n)}\mathbb{E}(K(t^{j}_i)\1_{\{K(t^{j}_i)\ge2\}}\int_{|z|>\beta_n}\1_{\{|z|>\frac{y}{|ff'(X_{t_{i}^1})|(K(t^{j}_i)-1) u_{n,m}|x|}\}}F(dz)|\mathcal{F}_{t^{1}_i})F(dx)\\
 	&\le\frac{C}{n}\int_{|x|>\beta_{n}}\frac{1}{\theta(\beta_n)}\mathbb{E}(K(t^{j}_i)\1_{\{K(t^{j}_i)\ge2\}}\theta(\frac{y}{|ff'(X_{t_{i}^1})|(K(t^{j}_i)-1)u_{n,m}|x|})|\mathcal{F}_{t^{1}_i})F(dx)\\&\le\frac{C}{n}\int_{|x|>\beta_{n}}\frac{1}{\theta(\beta_n)y^{\alpha}}\mathbb{E}(K(t^{j}_i)^2\1_{\{K(t^{j}_i)\ge2\}}|\mathcal{F}_{t^{1}_i}){u_{n,m}^{\alpha}|x|^{\alpha}}F(dx)\le\frac{Cu_{n,m}^{\alpha}\lambda_{n,m}\rho_n}{n^2y^{\alpha}}.
 \end{align*}
Then, we conclude the convergence of the triangular array with generic term ${\zeta}_{i,k,j}^n(1,2)$ using $u_{n,m}=(\frac{nm}{(m-1)\log{n}})^{1/\alpha}$, $\beta_n=(\frac{m}{m-1})^{1/\alpha}u_{n,m}^{-1}$,  $\rho_n\le C\log{1/\beta_n}$ (see \eqref{eq:3.13}), $\lambda_{n,m}\rightarrow0$ as $n\rightarrow\infty$, criteria \eqref{eq:2.9} and Lemma \ref{lem:tightrest}. Now, concerning the triangular array with the generic term ${\zeta'}_{i,k,j}^n(1,2)$, noticing that  $j$ and $k$ play a symmetric role in ${\zeta}_{i,k,j}^n(1,2)$ and ${\zeta'}_{i,k,j}^n(1,2)$, the same calculations yield the same bounds for the three conditions of criteria \eqref{eq:2.9} and therefore we obtain the convergence of the triangular array with the generic term ${\zeta'}_{i,k,j}^n(1,2)$ in the same way. Finally, concerning  ${\zeta''}_{i,k,j}^n(1,2)$, by similar arguments, we have $\mathbb{E}({\zeta''}_{i,k,j}^n(1,2)\1_{\{|{\zeta''}_{i,k,j}^n(1,2)|\le1\}}|\mathcal{F}_{t^{1}_i})=0$ and
	\begin{align*}
	&\mathbb{E}({\zeta''}_{i,k,j}^n(1,2)^2\1_{\{|{\zeta''}_{i,k,j}^n(1,2)|\le1\}}|\mathcal{F}_{t^{1}_i})	\le Cu_{n,m}^2\mathbb{E}((\sum_{h=2}^{K(t_i^k)}\Delta Y_{T^{\beta_n}_h(t^{k}_i)})^2(\sum_{h=2}^{K(t_i^j)}|\Delta Y_{T^{\beta_n}_h(t^{j}_i)}|)^{\alpha+(2-\alpha)}\\&\times\1_{\{|\sum_{h=2}^{K(t_i^j)}\Delta Y_{T^{\beta_n}_h(t^{j}_i)}|\le\frac{1}{u_{n,m}|ff'(X_{t_{i}^1})\sum_{h=2}^{K(t_i^k)}\Delta Y_{T^{\beta_n}_h(t^{k}_i)}|}\}}|\mathcal{F}_{t^{1}_i}),\\
	&\le Cu_{n,m}^{\alpha}\mathbb{E}((\sum_{h=2}^{K(t_i^k)}|\Delta Y_{T^{\beta_n}_h(t^{k}_i)}|)^{\alpha}(\sum_{h=2}^{K(t_i^j)}|\Delta Y_{T^{\beta_n}_h(t^{j}_i)}|)^{\alpha}|\mathcal{F}_{t^{1}_i})\\
	&\le\frac{Cu_{n,m}^{\alpha}\mathbb{E}(K(t_i^k) \1_{\{K(t_i^k)\ge2 \}}|\mathcal{F}_{t^{1}_i})\mathbb{E}( K(t_i^j)\1_{\{ K(t_i^j)\ge2\}}|\mathcal{F}_{t^{1}_i})\rho_n^2}{(\theta(\beta_{n}))^2}\le\frac{Cu_{n,m}^{\alpha}\lambda_{n,m}^2\rho_n^2}{n^2}.
\end{align*}
By the same calculations as for ${\zeta}_{i,k,j}^n(1,2)$, where we use the inequality $\1_{\{|\sum_{j'=1}^{\ell'}\sum_{j=1}^\ell a_{jj'}|> 1\}}\le \sum_{j'=1}^{\ell'}\sum_{j=1}^\ell\1_{\{|a_{jj'}|> 1/(\ell+\ell')\}}$, we have for  $y>1$
	\begin{align*}
&\mathbb{P}(|{\zeta''}_{i,k,j}^n(1,2)|>y|\mathcal{F}_{t^{1}_i})=\mathbb{E}(\1_{\{|ff'(X_{t_{i}^1})u_{n,m}\sum_{h=2}^{K(t_i^k)}\Delta Y_{T^{\beta_n}_h(t^{k}_i)}\sum_{h=2}^{K(t_i^j)}\Delta Y_{T^{\beta_n}_h(t^{j}_i)}|>y\}}|\mathcal{F}_{t^{1}_i})\\
&\le\mathbb{E}(K(t_i^k) K(t_i^j)\1_{\{K(t_i^k)\ge2, K(t_i^j)\ge2\}}\1_{\{|ff'(X_{t_{i}^1})u_{n,m}\Delta Y_{T^{\beta_n}_2(t^{k}_i)}\Delta Y_{T^{\beta_n}_2(t^{j}_i)}|>\frac{y}{(K(t_i^k)-1) (K(t_i^j)-1)}\}}|\mathcal{F}_{t^{1}_i})\\
&\le\mathbb{E}(\frac{K(t_i^k) K(t_i^j)\1_{\{K(t_i^k)\ge2, K(t_i^j)\ge2\}}}{(\theta(\beta_n))^2}\int_{|x|> \beta_n}\hspace{-0.5cm}\theta(\frac{y}{ff'(X_{t_{i}^1})u_{n,m}|x|(K(t_i^k)-1) (K(t_i^j)-1)})F(dx)|\mathcal{F}_{t^{1}_i})\\
&\le\frac{\mathbb{E}(K(t_i^k)^2 K(t_i^j)^2\1_{\{K(t_i^k)\ge2, K(t_i^j)\ge2\}}|\mathcal{F}_{t^{1}_i})}{(\theta(\beta_n))^2}\frac{u_{n,m}^{\alpha}\rho_n}{y^{\alpha}}\le\frac{Cu_{n,m}^\alpha\rho_n\lambda_{n,m}^2}{n^2y^{\alpha}}.
\end{align*}
Then, we conclude using $u_{n,m}=(\frac{nm}{(m-1)\log{n}})^{1/\alpha}$, $\beta_n=(\frac{m}{m-1})^{1/\alpha}u_{n,m}^{-1}$,  $\rho_n\le C\log{1/\beta_n}$ (see \eqref{eq:3.13}), $\lambda_{n,m}\rightarrow0$ as $n\rightarrow\infty$, criteria \eqref{eq:2.9} and Lemma \ref{lem:tightrest}.
 Therefore, we get ${\Gamma}^n(1,2)\stackrel{\mathbb{P}}{\rightarrow}0$.\\
\underline{\textbf{The term ${\Gamma}_t^n(2)$}}:
Let us recall that ${\Gamma}_t^n(2)=\sum_{k=2}^m\sum_{i=1}^{[nt]}\zeta_{i,k}^n(2)$ where
\begin{align*}
\zeta_{i,k}^n(2)=u_{n,m}ff'(X_{t_i^1}^n)\frac{d_n}{nm}[\frac{k-1}{nm}d_n+(N^{\beta_{n}}_{t_i^k}-N^{\beta_{n}}_{t^{1}_i})+(k-1)(N^{\beta_{n}}_{t^{k+1}_i}-N^{\beta_{n}}_{t_i^k})].
\end{align*}
In case \ref{2}, from hypotheses \eqref{h3} and \eqref{h4} $d_n=b+\int_{|x|>1}xF(dx)-{d'}_n=0$ then this ${\Gamma}_t^n(2)$ vanishes. In case \ref{4},  hypothesis \eqref{h3} yields ${d'}_n=0$ and then $d_n=b$.  Now, for $k$ fixed, as $ff'$ is bounded,  using Lemma \ref{lem:bigj}  and  $\int_{\R}x^2F(dx)<\infty$, we get $|\mathbb{E}(\zeta_{i,k}^n(2)|\mathcal{F}_{t^{1}_i})|\le \frac{Cu_{n,m}}{n^2}$ and
\begin{align*}
\mathbb{E}(|\zeta_{i,k}^n(2)|^2|\mathcal{F}_{t^{1}_i})&\le \frac{Cu_{n,m}^2}{n^2}(\frac{1}{n^2}+\mathbb{E}((N^{\beta_{n}}_{t_i^k}-N^{\beta_{n}}_{t^{1}_i})^2|\mathcal{F}_{t^{1}_i})+\mathbb{E}((N^{\beta_{n}}_{t^{k+1}_i}-N^{\beta_{n}}_{t_i^k})^2|\mathcal{F}_{t^{1}_i}))\\
&\le \frac{Cu_{n,m}^2}{n^2}(\frac{1}{n^2}+\frac{1}{n})\le \frac{Cu_{n,m}^2}{n^3}.
\end{align*}
Then, for case \ref{4}, we conclude using  $u_{n,m}=\frac{nm}{(m-1)\log{n}}$, criteria \eqref{eq:2.8} and Lemma \ref{lem:tightrest}. Therefore, we get $\Gamma^n(2)\stackrel{\mathbb{P}}{\rightarrow}0$. \\
\underline{\textbf{The term ${\Gamma}_t^n(3)$}}:
 Let us recall  that $\Gamma^n_t(3)=\sum_{k=2}^m\sum_{i=1}^{[nt]}\zeta_{i,k}^n(3)$ where
	$$
		\zeta_{i,k}^n(3)=u_{n,m}ff'(X_{t_i^1}^n)(M^{\beta_n}_{t_i^k}-M^{\beta_n}_{t^{1}_i})(Y_{t^{k+1}_i}-Y_{t_i^k}).	
	$$
	For $k$ fixed, as $ff'$ is bounded, by using property \ref{I}, lemmas \ref{lem:bigj} and \ref{lem:smallj} and  $\int_{\R}x^2 F(dx)<\infty$ (see Remark \ref{rmk:1}),  we have $\mathbb{E}(\zeta_{i,k}^n(3)|\mathcal{F}_{t^{1}_i})=0$ and 
	\begin{align*}
	\mathbb{E}(|\zeta_{i,k}^n(3)|^2|\mathcal{F}_{t^{1}_i})&\leq Cu_{n,m}^2\mathbb{E}((M^{\beta_n}_{t_i^k}-M^{\beta_n}_{t^{1}_i})^2|\mathcal{F}_{t^{1}_i})\left[\mathbb{E}((Y^{\beta_n}_{t^{k+1}_i}-Y^{\beta_n}_{t_i^k})^2|\mathcal{F}_{t^{1}_i})+\mathbb{E}((N^{\beta_n}_{t^{k+1}_i}-N^{\beta_n}_{t_i^k})^2|\mathcal{F}_{t^{1}_i})\right]\\
		&\leq \frac{Cu_{n,m}^2c_n}{n} \left(\frac{c_n}{n}+\frac{d_n^2}{n^2}+\frac{1}{n}\right).
	\end{align*}
	Then,  we conclude using \eqref{h3} that $d_n=b$, $u_{n,m}=(\frac{nm}{(m-1)\log{n}})^{1/\alpha}$, $\beta_n=(\frac{m}{m-1})^{1/\alpha}u_{n,m}^{-1}$,  $c_n\le C\beta_n^{2-\alpha}$ (see \eqref{eq:3.13}), criteria \eqref{eq:2.8} and Lemma \ref{lem:tightrest}. Therefore, we get $\Gamma^{n}(3)\stackrel{\mathbb{P}}{\rightarrow}0$.\\
	\underline{\textbf{The term ${\Gamma}_t^n(4)$}}:
Let us recall that $\Gamma^n_t(4)=\sum_{k=2}^m\sum_{i=1}^{[nt]}\zeta_{i,k}^n(4)$ where
	$$
	\zeta_{i,k}^n(4)=u_{n,m}ff'(X_{t_i^1}^n)[(A^{\beta_n}_{t_i^k}-A^{\beta_n}_{t^{1}_i})+(N^{\beta_n}_{t_i^k}-N^{\beta_n}_{t^{1}_i})](M^{\beta_n}_{t^{k+1}_i}-M^{\beta_n}_{t_i^k}).
$$
	For $k$ fixed, by similar arguments as for the term ${\Gamma}_t^n(3)$,  we have $\mathbb{E}(\zeta_{i,k}^n(4)|\mathcal{F}_{t^{1}_i})=0$ and
	$
		\mathbb{E}(|\zeta_{i,k}^n(4)|^2|\mathcal{F}_{t^{1}_i})\le \frac{C u_{n,m}^2c_n}{n}\left(\frac{d_n^2}{n^2}+\frac{1}{n}\right)$.
		Then,  we conclude similarly that $\Gamma^{n}(4)\stackrel{\mathbb{P}}{\rightarrow}0$.\\
	\underline{\textbf{The  term ${\Gamma}_t^n(5)$}}:
Since $k(x,y)$ is bounded on $\R^2$,  it is enough to prove that for $k\in\{1,\dots,m\}$, the following nine  triangular arrays with generic terms $\{{\zeta}_{i,k}^n(5,j),j\in\{1,\dots,9\}$  converge to $0$ as $n\rightarrow\infty$ with
\begin{align*}
	\left\{\begin{array}{ll}
		{\zeta}_{i,k}^n(5,1)=u_{n,m}(A^{\beta_n}_{t_i^k}-A^{\beta_n}_{t^{1}_i})^2|A^{\beta_n}_{t^{k+1}_i}-A^{\beta_n}_{t_i^k}|,&
		{\zeta}_{i,k}^n(5,2)=u_{n,m}(A^{\beta_n}_{t_i^k}-A^{\beta_n}_{t^{1}_i})^2	|M^{\beta_n}_{t^{k+1}_i}-M^{\beta_n}_{t_i^k}|\\
		{\zeta}_{i,k}^n(5,3)=u_{n,m}(A^{\beta_n}_{t_i^k}-A^{\beta_n}_{t^{1}_i})^2	|N^{\beta_n}_{t^{k+1}_i}-N^{\beta_n}_{t_i^k}|,&
	{\zeta}_{i,k}^n(5,4)=u_{n,m}(M^{\beta_n}_{t_i^k}-M^{\beta_n}_{t^{1}_i})^2	|A^{\beta_n}_{t^{k+1}_i}-A^{\beta_n}_{t_i^k}|\\
	{\zeta}_{i,k}^n(5,5)=u_{n,m}(N^{\beta_n}_{t_i^k}-N^{\beta_n}_{t^{1}_i})^2	|M^{\beta_n}_{t^{k+1}_i}-M^{\beta_n}_{t_i^k}|,&
	{\zeta}_{i,k}^n(5,6)=u_{n,m}(M^{\beta_n}_{t_i^k}-M^{\beta_n}_{t^{1}_i})^2	|N^{\beta_n}_{t^{k+1}_i}-N^{\beta_n}_{t_i^k}|\\
	{\zeta}_{i,k}^n(5,7)=u_{n,m}(M^{\beta_n}_{t_i^k}-M^{\beta_n}_{t^{1}_i})^2	|M^{\beta_n}_{t^{k+1}_i}-M^{\beta_n}_{t_i^k}|,&
	{\zeta}_{i,k}^n(5,8)=u_{n,m}(N^{\beta_n}_{t_i^k}-N^{\beta_n}_{t^{1}_i})^2	|N^{\beta_n}_{t^{k+1}_i}-N^{\beta_n}_{t_i^k}|\\{\zeta}_{i,k}^n(5,9)=u_{n,m}(N^{\beta_n}_{t_i^k}-N^{\beta_n}_{t^{1}_i})^2	|A^{\beta_n}_{t^{k+1}_i}-A^{\beta_n}_{t_i^k}|.&
	\end{array}\right.
\end{align*}
From \eqref{eq:RC1}, \eqref{eq:RC1.} and \eqref{eq:RC1.1} and by same calculations, the triangular arrays with generic terms ${\zeta}_{i,k}^n(5,i)$, $i\in\{1,\dots,7\}$ and ${\zeta}_{i,k}^n(5,9)$ are bounded as follows 
\begin{align}\label{eq:gamma5}
	\left\{\begin{array}{ll}
		\mathbb{E}(|{\zeta}_{i,k}^n(5,1)||\mathcal{F}_{t^{1}_i})\le C\frac{u_{n,m}|d_n|^3}{n^3},
		&\mathbb{E}(|{\zeta}_{i,k}^n(5,2)||\mathcal{F}_{t^{1}_i}){\le} C\frac{u_{n,m}|d_n|^2\sqrt{c_n}}{n^2\sqrt{n}}\\
		\mathbb{E}(|{\zeta}_{i,k}^n(5,3)||\mathcal{F}_{t^{1}_i}){\le} C\frac{u_{n,m}|d_n|^2\delta_n}{n^3},
		&\mathbb{E}(|{\zeta}_{i,k}^n(5,4)||\mathcal{F}_{t^{1}_i})\le C\frac{u_{n,m}|d_n|c_n}{n^2}\\
	\mathbb{E}(|{\zeta}_{i,k}^n(5,5)||\mathcal{F}_{t_i^1})\le C \frac{u_{n,m}\sqrt{c_n}}{n\sqrt{n}},
		&\mathbb{E}(|{\zeta}_{i,k}^n(5,6)||\mathcal{F}_{t_i^1})\le C \frac{u_{n,m}c_n\delta_n}{n^2}\\
	\mathbb{E}(|{\zeta}_{i,k}^n(5,7)||\mathcal{F}_{t^{1}_i})\le C \frac{u_{n,m}c_n^{3/2}}{n\sqrt{n}},
		&\mathbb{E}(|{\zeta}_{i,k}^n(5,9)||\mathcal{F}_{t^{1}_i})\le C\frac{u_{n,m}|d_n|}{n^2}.
	\end{array}\right.
\end{align}
	Then,  we conclude using  $u_{n,m}=(\frac{nm}{(m-1)\log{n}})^{1/\alpha}$, $\beta_n=(\frac{m}{m-1})^{1/\alpha}u_{n,m}^{-1}$,  $c_n\le C\beta_n^{2-\alpha}$ and $\delta_n\le C\log{1/\beta_n}$ (see \eqref{eq:3.13}),  $d_n=0$ for case \ref{2}, $d_n=b$ for case \ref{4}, criteria \eqref{eq:2.7} and Lemma \ref{lem:tightrest}. Therefore, we have the convergence to $0$ of the eight triangular arrays corresponding to these eight generic terms.
Now, concerning ${\zeta}_{i,k}^n(5,8)$, we rewrite  as $u_{n,m}(\sum_{h=1}^{k-1}(N^{\beta_n}_{t^{h+1}_i}-N^{\beta_n}_{t^{h}_i}))^2|N^{\beta_n}_{t^{k+1}_i}-N^{\beta_n}_{t_i^k}|$, then by Jensen's inequality we have 
$$
	{\zeta}_{i,k}^n(5,8)
	\le Cu_{n,m}\sum_{h=1}^{k-1}(\sum_{j=1}^{K(t^{h}_i)}\Delta Y_{T^{\beta_n}_j(t^{h}_i)})^2\sum_{j'=1}^{K(t_i^k)}|\Delta Y_{T^{\beta_n}_{j'}(t_i^k)}|.
$$
Thus, for $h<k$ fixed, thanks to the inequality $(a+b)^2\le2(a^2+b^2)$,  to prove the convergence of the last triangular array, it is enough to consider the two following generic terms \begin{align*}\left\{\begin{array}{l}\zeta_{i,k,h}^n(5,8)=u_{n,m}(\sum_{j=2}^{K(t^{h}_i)}\Delta Y_{T^{\beta_n}_j(t^{h}_i)})^2\sum_{j'=1}^{K(t^{k}_i)}|\Delta Y_{T^{\beta_n}_{j'}(t_i^k)}|,\\
{\zeta'}_{i,k,h}^n(5,8)=u_{n,m}(\Delta Y_{T^{\beta_n}_1(t^{h}_i)}\1_{\{K(t_i^h)\ge1\}})^2\sum_{j=1}^{K(t^{k}_i)}|\Delta Y_{T^{\beta_n}_{j}(t_i^k)}|.
\end{array}\right.
\end{align*}
 First, we consider $\zeta_{i,k,h}^n(5,8)$. By property \ref{I}, the inequality for $\alpha\le 1$,  $(\sum_{i=1}^n|x_i|)^{\alpha}\le\sum_{i=1}^n|x_i|^{\alpha}$ and $\int_{|x|>\beta_n}|x|^{\alpha/2}F(dx)\le\frac{C}{\beta_n^{\alpha/2}}$ (see \eqref{eq:3.3}), we have $\mathbb{E}(|\zeta_{i,k,h}^n(5,8)|\1_{\{|\zeta_{i,k,h}^n(5,8)|\le1\}}|\mathcal{F}_{t^{1}_i})$
\begin{align*}
	=&u_{n,m}\mathbb{E}(|\sum_{j=2}^{K(t^{h}_i)}\Delta Y_{T^{\beta_n}_j(t^{h}_i)}|^{\alpha+(2-\alpha)}|\sum_{j'=1}^{K(t^{k}_i)}\Delta Y_{T^{\beta_n}_{j'}(t^{k}_i)}|\1_{\{|\sum_{j=2}^{K(t^{h}_i)}\Delta Y_{T^{\beta_n}_j(t^{h}_i)}|\le \frac{1}{\sqrt{u_{n,m}\sum_{j'=1}^{K(t^{k}_i)}|\Delta Y_{T^{\beta_n}_{j'}(t^{k}_i)}|}}\}}|\mathcal{F}_{t^{1}_i})\\
	\le&u_{n,m}^{\alpha/2}\mathbb{E}((\sum_{j=2}^{K(t^{h}_i)}|\Delta Y_{T^{\beta_n}_j(t^{h}_i)}|)^{\alpha}(\sum_{j'=1}^{K(t^{k}_i)}|\Delta Y_{T^{\beta_n}_{j'}(t^{k}_i)}|)^{\alpha/2}|\mathcal{F}_{t^{1}_i})\\
	\le&u_{n,m}^{\alpha/2}\mathbb{E}(\sum_{j=2}^{K(t^{h}_i)}|\Delta Y_{T^{\beta_n}_j(t^{h}_i)}|^{\alpha}\sum_{j'=1}^{K(t^{k}_i)}|\Delta Y_{T^{\beta_n}_{j'}(t^{k}_i)}|^{\alpha/2}|\mathcal{F}_{t^{1}_i})\\
	\le&\frac{Cu_{n,m}^{\alpha/2}\rho_n}{(\theta(\beta_n))^2}\int_{|x|>\beta_n}|x|^{\alpha/2}F(dx)\mathbb{E}(K(t^{h}_i)K(t^{k}_i)\1_{\{K(t^{h}_i)\ge2,K(t^{k}_i)\ge1\}}|\mathcal{F}_{t^{1}_i})\le\frac{Cu_{n,m}^{\alpha/2}\lambda_{n,m}\rho_n}{n^2\beta_n^{\alpha/2}}.
\end{align*}
Now, using similar arguments as above,  Jensen's inequality, the inequality $\1_{\{|\sum_{j'=1}^{\ell'}\sum_{j=1}^\ell a_{jj'}|> 1\}}\le \sum_{j'=1}^{\ell'}\sum_{j=1}^\ell\1_{\{|a_{jj'}|> 1/(\ell+\ell')\}}$, $\theta(.)$ is decreasing and \eqref{h1}, for $y>1$, we have
\begin{align*}
	&\mathbb{P}(|\zeta_{i,k,h}^n(5,8)|>y|\mathcal{F}_{t^{1}_i})\le \mathbb{E}(\1_{\{u_{n,m}K(t^{h}_i)\sum_{j=2}^{K(t^{h}_i)}\Delta Y_{T^{\beta_n}_j(t^{h}_i)}^2\sum_{j'=1}^{K(t^{k}_i)}|\Delta Y_{T^{\beta_n}_{j'}(t_i^k)}|>y\}}|\mathcal{F}_{t^{1}_i})\\
	\le& \mathbb{E}(K(t^{h}_i)K(t^{k}_i)\1_{\{K(t^{h}_i)\ge2,K(t^{k}_i)\ge1\}}\1_{\{K(t^{h}_i)\Delta Y^2_{T^{\beta_n}_2(t^{h}_i)}|\Delta Y_{T^{\beta_n}_1(t_i^k)}|>\frac{y}{u_{n,m}K(t^{h}_i)K(t^{k}_i)}\}}|\mathcal{F}_{t^{1}_i})\\
	\le& \frac{1}{(\theta(\beta_n))^2}\mathbb{E}(K(t^{h}_i)K(t^{k}_i)\1_{\{K(t^{h}_i)\ge2,K(t^{k}_i)\ge1\}}\int_{|z|>\beta_n}\theta(\sqrt{\frac{y}{u_{n,m}K(t^{h}_i)^2K(t^{k}_i)|z|}})F(dz)|\mathcal{F}_{t^{1}_i})\\
	\le& \frac{1}{(\theta(\beta_n))^2}\int_{|z|>\beta_n}\frac{u_{n,m}^{\alpha/2}|z|^{\alpha/2}}{y^{\alpha/2}}F(dz)\mathbb{E}(K(t^{h}_i)^2K(t^{k}_i)^2\1_{\{K(t^{h}_i)\ge2,K(t^{k}_i)\ge1\}}|\mathcal{F}_{t^{1}_i})\le \frac{Cu_{n,m}^{\alpha/2}\lambda_{n,m}}{n^2y^{\alpha/2}\beta_{n}^{\alpha/2}}.
\end{align*}
We conclude using $u_{n,m}=(\frac{nm}{(m-1)\log{n}})^{1/\alpha}$, $\beta_n=(\frac{m}{m-1})^{1/\alpha}u_{n,m}^{-1}$,  $\rho_n\le C\log{1/\beta_n}$ (see \eqref{eq:3.13}), $\lambda_{n,m}\rightarrow0$ as $n\rightarrow\infty$, criteria \eqref{eq:2.9} and Lemma \ref{lem:tightrest}. Similarly for ${\zeta'}_{i,k,h}^n(5,8)$,  using $1-e^{-\lambda_{n,m}}\le\lambda_{n,m}$ and that $c(\beta)\le C\beta^{2-\alpha}$, $\mathbb{E}(|{\zeta'}_{i,k,h}^n(5,8)|\1_{\{|{\zeta'}_{i,k,h}^n(5,8)|\le1\}}|\mathcal{F}_{t^{1}_i})$ is equal to
\begin{align*}
	&u_{n,m}(1-e^{-\lambda_{n,m}})\mathbb{E}(\Delta Y_{T^{\beta_n}_1(t^{h}_i)}^2|\sum_{j=1}^{K(t_i^k)}\Delta Y_{T^{\beta_n}_j(t_i^k)}|\1_{\{|\Delta Y_{T^{\beta_n}_1(t^{h}_i)}|\le\frac{1}{\sqrt{u_{n,m}|\sum_{j=1}^{K(t_i^k)}\Delta Y_{T^{\beta_n}_j(t_i^k)}|}}\}}|\mathcal{F}_{t^{1}_i})\\
	&\le \frac{u_{n,m}}{n}\mathbb{E}(c(\frac{1}{\sqrt{u_{n,m}|\sum_{j=1}^{K(t_i^k)}\Delta Y_{T^{\beta_n}_j(t_i^k)}|}})|\sum_{j=1}^{K(t_i^k)}\Delta Y_{T^{\beta_n}_j(t_i^k)}||\mathcal{F}_{t^{1}_i})\\
	&\le \frac{Cu_{n,m}^{\alpha/2}}{n}\mathbb{E}(\sum_{j=1}^{K(t_i^k)}|\Delta Y_{T^{\beta_n}_j(t_i^k)}|^{\alpha/2}|\mathcal{F}_{t^{1}_i})\le \frac{Cu_{n,m}^{\alpha/2}}{n\theta(\beta_n)}\int_{|x|> \beta_n}\hspace{-0.5cm}|x|^{\alpha/2}F(dx)\mathbb{E}(K(t_i^k)\1_{\{K(t_i^k)\ge1\}}|\mathcal{F}_{t^{1}_i})
	\le \frac{Cu_{n,m}^{\alpha/2}}{n^2\beta_n^{\alpha/2}},
\end{align*}
and using the inequality $\1_{\{|\sum_{j=1}^\ell a_j|> 1\}}\le \sum_{j=1}^\ell\1_{\{|a_j|> 1/\ell\}}$, $\theta(.)$ is decreasing and \eqref{h1}, we get for $y>1$, 
\begin{align*}
	\mathbb{P}(|{\zeta'}_{i,k,h}^n(5,8)|>y|\mathcal{F}_{t^{1}_i})
	\le& \lambda_{n,m}\mathbb{E}(K(t^{k}_i)\1_{\{K(t^{k}_i)\ge1\}}\1_{\{\Delta Y^2_{T^{\beta_n}_1(t^{h}_i)}|\Delta Y_{T^{\beta_n}_1(t_i^k)}|>\frac{y}{u_{n,m}K(t^{k}_i)}\}}|\mathcal{F}_{t^{1}_i})\\
	\le& \frac{C\lambda_{n,m}}{(\theta(\beta_n))^2}\mathbb{E}(K(t^{k}_i)\1_{\{K(t^{k}_i)\ge1\}}\int_{|z|>\beta_n}\theta(\sqrt{\frac{y}{u_{n,m}K(t^{k}_i)|z|}})F(dz)|\mathcal{F}_{t^{1}_i})\\
\le&\frac{Cu_{n,m}^{\alpha/2}}{n\theta(\beta_n)y^{\alpha/2}}\int_{|z|> \beta_n}\hspace{-0.5cm}|z|^{\alpha/2}F(dz)\mathbb{E}(K(t_i^k)^2\1_{\{K(t_i^k)\ge1\}}|\mathcal{F}_{t^{1}_i})\le \frac{Cu_{n,m}^{\alpha/2}}{n^2y^{\alpha/2}\beta_{n}^{\alpha/2}}.
\end{align*}
We conclude using $u_{n,m}=(\frac{nm}{(m-1)\log{n}})^{1/\alpha}$, $\beta_n=(\frac{m}{m-1})^{1/\alpha}u_{n,m}^{-1}$, criteria \eqref{eq:2.9} and Lemma \ref{lem:tightrest}.
 Therefore, we get $\Gamma^{n}(5)\stackrel{\mathbb{P}}{\rightarrow}0$.
\subsection{Proof of Lemma \ref{lem:rest3}}	
From the formula of the rest term given by \eqref{eq:MR3}, we have
\begin{align*}
	\mathcal{R}^{n,m}_t={\Gamma}^n_t(1,2)+{\Gamma}^n_t(2,2)+\sum_{i=3}^5\Gamma^n_t(i).
\end{align*}
 The aim is to prove that each term converges uniformly in probability to $0$ when $n\rightarrow\infty$.\\
 \underline{\textbf{The term ${\Gamma}^n_t(1,2)$}}:
 Let us recall that ${\Gamma}^n_t(1,2)=\sum_{k=2}^m\sum_{i=1}^{[nt]}{\zeta}_{i,k}^n(1,2)$ where 
	\begin{displaymath}
	{\zeta}_{i,k}^n(1,2)=\frac{u_{n,m}d_{n}}{nm}ff'(X_{t_i^1}^n)\left[\sum_{j=1}^{k-1}\sum_{h=2}^{K(t_{j}^i)}\Delta Y_{T^{\beta_n}_h(t_{j}^i)}+(k-1)\sum_{h=2}^{K(t_{k}^i)}\Delta Y_{T^{\beta_n}_h(t_{k}^i)}\right].
	\end{displaymath}
	For $k\in\{2,\dots,m\}$, since $ff'$ is bounded, by property \ref{I}, Jensen's inequality and similar calculations as in \eqref{eq:y2case3}, we have
	\begin{align*}
		\mathbb{E}(|{\zeta}_{i,k}^n(1,2)||\mathcal{F}_{t^{1}_i})
		&\leq	\frac{Cu_{n,m}d_{n}}{n}\left[\sum_{j=1}^{k-1}\mathbb{E}(\sum_{h=2}^{K(t^{j}_i)}|\Delta Y_{T^{\beta_n}_h(t^{j}_i)}||\mathcal{F}_{t^{1}_i})+(k-1)\mathbb{E}(\sum_{h=2}^{K(t_i^k)}|\Delta Y_{T^{\beta_n}_h(t_i^k)}||\mathcal{F}_{t^{1}_i})\right]\\
		&\le \frac{Cu_{n,m}d_{n}\delta_n\lambda_{n,m}}{n^2}.
	\end{align*}
	As we choose $u_{n,m}=\frac{nm}{(m-1)(\log{n})^2}$, using $d_{n}\leq C\log{n}$, $\delta_{n}\leq C\log{n}$ (see \eqref{eq:3.13}) and $\lambda_{n,m}\to0$ as $n\to\infty$, we conclude by criteria \eqref{eq:2.7} and Lemma \ref{lem:tightrest} that
	$
	{\Gamma}^n(1,2)\stackrel{\mathbb{P}}{\rightarrow}0.
$\\
	 \underline{\textbf{The term ${\Gamma}^n_t(2,2)$}}:
	Let us recall that ${\Gamma}^n_t(2,2)=\sum_{k=2}^m\sum_{i=1}^{[nt]}{\zeta}_{i,k}^n(2,2)$ where 
	\begin{align*}
		{\zeta}_{i,k}^n(2,2)=&u_{n,m}ff'(X_{t_i^1}^n)\left(\Delta Y_{T^{\beta_n}_1(t_i^k)}\1_{\{K(t_i^k)\geq1\}}\sum_{j=1}^{k-1}\sum_{h=2}^{K(t^{j}_i)}\Delta Y_{T^{\beta_n}_h(t^{j}_i)}\right.\\
		&+\left.\sum_{h=2}^{K(t_i^k)}\Delta Y_{T^{\beta_n}_h(t_i^k)}\sum_{j=1}^{k-1}\Delta Y_{T^{\beta_n}_1(t^{j}_i)}\1_{\{K(t^{j}_i)\geq1\}}+\sum_{h=2}^{K(t_i^k)}\Delta Y_{T^{\beta_n}_h(t_i^k)}\sum_{j=1}^{k-1}\sum_{h=2}^{K(t^{j}_i)}\Delta Y_{T^{\beta_n}_h(t^{j}_i)}\right).
		\end{align*}
		For any fixed $k\in\{2,\dots, m\}$, as $ff'$ is bounded, by   \ref{I}, $1-e^{-\lambda_{n,m}}\le\lambda_{n,m}$ and  using the same calculations in \eqref{eq:y2case3}, we have
	\begin{align*}
		\mathbb{E}(|{\zeta}_{i,k}^n(2,2)||\mathcal{F}_{t^{1}_i})\leq Cu_{n,m}\left(2\times\frac{\delta_n}{n}\frac{\delta_n\lambda_{n,m}}{n}+\frac{\delta_n^2\lambda_{n,m}^2}{n^2}\right)
		\leq C\frac{u_{n,m}\delta_n^2\lambda_{n,m}}{n^2}.
	\end{align*}
	Then we conclude using the same arguments as above to get ${\Gamma}^n(2,2)\stackrel{\mathbb{P}}{\rightarrow}0$.\\
\underline{\textbf{The term $\Gamma_t^n(3)$}}:
 Let us recall that  $\Gamma_t^n(3)=\sum_{k=2}^m\sum_{i=1}^{[nt]}\zeta_{i,k}^n(3)$ with
$$	\zeta_{i,k}^n(3)
	=u_{n,m}ff'(X_{t_i^1}^n)(M_{t_i^k}^{\beta_{n}}-M_{t^{1}_i}^{\beta_{n}})(Y_{t_i^{k+1}}-Y_{t_i^k}).$$
For $k\in\{2,\dots,m\}$, since $ff'$ is bounded, by property \ref{I}, the independence between $M_{t_i^k}^{\beta_{n}}-M_{t^{1}_i}^{\beta_{n}}$ and $Y_{t_i^{k+1}}-Y_{t_i^k}$, decomposition \eqref{eq:3.7},  Lemma \ref{lem:bigj} and Lemma \ref{lem:smallj}, $\int_{\R}x^2F(dx)<\infty$ (see Remark \ref{rmk:1}),we get $	\mathbb{E}(\zeta_{i,k}^n(3)|\mathcal{F}_{t^{1}_i})=0$ and
$$\mathbb{E}(|\zeta_{i,k}^n(3)|^2|\mathcal{F}_{t^{1}_i})\leq Cu_{n,m}^2\mathbb{E}(|M_{t^{k}_i}^{\beta_{n}}-M_{t_i^1}^{\beta_{n}}|^2|\mathcal{F}_{t^{1}_i})\mathbb{E}(|Y_{t_i^{k+1}}-Y_{t^{k}_i}|^2|\mathcal{F}_{t^{1}_i})
	\leq \frac{Cu_{n,m}^2c_{n}}{n}(\frac{d_{n}^2}{n^2}+\frac{c_{n}}{n}+\frac{1}{n}).$$
Then we conclude using $u_{n,m}=\frac{nm}{(m-1)(\log{n})^2}$ and $\beta_{n}=\frac{\log{n}}{n}$, $d_n\le C\log{1/\beta_n}$ and $c_n\le C\beta_n$ from \eqref{eq:3.13}, criteria \eqref{eq:2.8} and Lemma \ref{lem:tightrest}. Therefore, we get $\Gamma^n(3)\stackrel{\mathbb{P}}{\rightarrow}0$.\\
 \underline{\textbf{The term $\Gamma_t^n(4)$}}: Let us recall that $\Gamma_t^n(4)=\sum_{k=2}^m\sum_{i=1}^{[nt]}\zeta_{i,k}^n(4)$ with
$$
	\zeta_{i,k}^n(4)
	=u_{n,m}ff'(X_{t_i^1}^n)(A^{\beta_n}_{t_i^k}-A^{\beta_n}_{t^{1}_i}+N^{\beta_n}_{t_i^k}-N^{\beta_n}_{t^{1}_i})(M^{\beta_{n}}_{t^{k+1}_i}-M^{\beta_{n}}_{t_i^k}).
$$
For $k$ fixed, using the same arguments as for $\Gamma^n(3)$, we get $	\mathbb{E}(\zeta_{i,k}^n(4)|\mathcal{F}_{t^{1}_i})=0$,
$ 
	\mathbb{E}((\zeta_{i,k}^n(4))^2|\mathcal{F}_{t^{1}_i})
	\leq \frac{Cu_{n,m}^2c_{n}}{n}(\frac{d_{n}^2}{n^2}+\frac{1}{n})
$ and therefore, we get $\Gamma^n(4)\stackrel{\mathbb{P}}{\rightarrow}0$.\\
\underline{\textbf{The term $\Gamma_t^n(5)$}}:
Since $k(x,y)$ is bounded on $\R^2$,  it is enough to prove that for $k\in\{1,\dots,m\}$, the following nine  triangular arrays with generic terms $\{{\zeta}_{i,k}^n(5,j),j\in\{1,\dots,9\}$  converge to $0$ as $n\rightarrow\infty$ with
\begin{align*}
	\left\{\begin{array}{ll}
		{\zeta}_{i,k}^n(5,1)=u_{n,m}(A^{\beta_n}_{t_i^k}-A^{\beta_n}_{t^{1}_i})^2|A^{\beta_n}_{t^{k+1}_i}-A^{\beta_n}_{t_i^k}|,&
		{\zeta}_{i,k}^n(5,2)=u_{n,m}(A^{\beta_n}_{t_i^k}-A^{\beta_n}_{t^{1}_i})^2	|M^{\beta_n}_{t^{k+1}_i}-M^{\beta_n}_{t_i^k}|\\
		{\zeta}_{i,k}^n(5,3)=u_{n,m}(A^{\beta_n}_{t_i^k}-A^{\beta_n}_{t^{1}_i})^2	|N^{\beta_n}_{t^{k+1}_i}-N^{\beta_n}_{t_i^k}|,&
		{\zeta}_{i,k}^n(5,4)=u_{n,m}(M^{\beta_n}_{t_i^k}-M^{\beta_n}_{t^{1}_i})^2	|A^{\beta_n}_{t^{k+1}_i}-A^{\beta_n}_{t_i^k}|\\
		{\zeta}_{i,k}^n(5,5)=u_{n,m}(N^{\beta_n}_{t_i^k}-N^{\beta_n}_{t^{1}_i})^2	|M^{\beta_n}_{t^{k+1}_i}-M^{\beta_n}_{t_i^k}|,&
		{\zeta}_{i,k}^n(5,6)=u_{n,m}(M^{\beta_n}_{t_i^k}-M^{\beta_n}_{t^{1}_i})^2	|N^{\beta_n}_{t^{k+1}_i}-N^{\beta_n}_{t_i^k}|\\
		{\zeta}_{i,k}^n(5,7)=u_{n,m}(M^{\beta_n}_{t_i^k}-M^{\beta_n}_{t^{1}_i})^2	|M^{\beta_n}_{t^{k+1}_i}-M^{\beta_n}_{t_i^k}|,&
		{\zeta}_{i,k}^n(5,8)=u_{n,m}(N^{\beta_n}_{t_i^k}-N^{\beta_n}_{t^{1}_i})^2	|N^{\beta_n}_{t^{k+1}_i}-N^{\beta_n}_{t_i^k}|\\{\zeta}_{i,k}^n(5,9)=u_{n,m}(N^{\beta_n}_{t_i^k}-N^{\beta_n}_{t^{1}_i})^2	|A^{\beta_n}_{t^{k+1}_i}-A^{\beta_n}_{t_i^k}|.&
	\end{array}\right.
\end{align*}
From \eqref{eq:RC1}, \eqref{eq:RC1.},  \eqref{eq:RC1.1} and \eqref{eq:gamma5}, the triangular arrays with generic terms ${\zeta}_{i,k}^n(5,i)$, $i\in\{1,\dots,9\}$ are bounded as follows 
\begin{align}\label{eq:gamma5i}
	\left\{\begin{array}{ll}
		\mathbb{E}(|{\zeta}_{i,k}^n(5,1)||\mathcal{F}_{t^{1}_i})\le C\frac{u_{n,m}|d_n|^3}{n^3},
		&\mathbb{E}(|{\zeta}_{i,k}^n(5,2)||\mathcal{F}_{t^{1}_i}){\le} C\frac{u_{n,m}|d_n|^2\sqrt{c_n}}{n^2\sqrt{n}}\\
		\mathbb{E}(|{\zeta}_{i,k}^n(5,3)||\mathcal{F}_{t^{1}_i}){\le} C\frac{u_{n,m}|d_n|^2\delta_n}{n^3},
		&\mathbb{E}(|{\zeta}_{i,k}^n(5,4)||\mathcal{F}_{t^{1}_i})\le C\frac{u_{n,m}|d_n|c_n}{n^2}\\
		\mathbb{E}(|{\zeta}_{i,k}^n(5,5)||\mathcal{F}_{t_i^1})\le C \frac{u_{n,m}\sqrt{c_n}}{n\sqrt{n}},
		&\mathbb{E}(|{\zeta}_{i,k}^n(5,6)||\mathcal{F}_{t_i^1})\le C \frac{u_{n,m}c_n\delta_n}{n^2}\\
		\mathbb{E}(|{\zeta}_{i,k}^n(5,7)||\mathcal{F}_{t^{1}_i})\le C \frac{u_{n,m}c_n^{3/2}}{n\sqrt{n}},
		&\mathbb{E}(|{\zeta}_{i,k}^n(5,8)||\mathcal{F}_{t^{1}_i})\le C \frac{u_{n,m}\delta_n}{n^2},\\
		\mathbb{E}(|{\zeta}_{i,k}^n(5,9)||\mathcal{F}_{t^{1}_i})\le C\frac{u_{n,m}|d_n|}{n^2}.&
	\end{array}\right.
\end{align}
Then, by $u_{n,m}=\frac{nm}{(m-1)(\log{n})^2}$ , $d_n\le C\log{1/\beta_n}$, $\delta_n\le C\log{1/\beta_n}$ and $c_n\le C\beta_n$ (see \eqref{eq:3.13}) with $\beta_{n}=\frac{\log{n}}{n}$, criteria \eqref{eq:2.7} and Lemma \ref{lem:tightrest}, we conclude the convergence of these triangular arrays.
Therefore, we get $\Gamma^n(5)\stackrel{\mathbb{P}}{\rightarrow}0$.
\subsection{Proof of Lemma \ref{lem:rest5}}Here, we prove that 
the sequences of processes $(\overline{Y}^n(1))_{n\ge0}$ and  $(\mathcal{R}^{n,m})_{n\ge0}$ converge uniformly in probability to $0$ as $n\rightarrow\infty$. First, instead of considering the form $\overline{Y}^{n}_t(1)$ given in \eqref{Y}, it is enough to prove that for each $k\in\{1,\dots,m\}$ the triangular arrays with generic terms $y^{n,m}_{i,k}(1,1)=M^{\beta_n}_{t_i^k,t_i^{k+1}}$ and $y^{n,m}_{i,k}(1,2)=\sum_{j\geq2}\Delta Y_{T^{\beta_n}_{ \s j}( t_i^k)}\1_{\{K({\s t_i^k})\geq j\}}$ converge uniformly in probability to $0$ as $n\rightarrow\infty$. On the one hand, from the above \eqref{eq:tighty1}, we have
$
\mathbb{E}(y^{n,m}_{i,k}(1,1)|\mathcal{F}_{t^{1}_i})
=0,\quad
\mathbb{E}((y^{n,m}_{i,k}(1,1))^2|\mathcal{F}_{t^{1}_i})=\frac{c_{n}}{nm}.
$ Then, $y^{n,m}_{i,k}(1,1)$ satisfies $\eqref{eq:2.8}$ and we conclude using $c_{n}\le C\beta_n^{2-\alpha}$ (see \eqref{eq:3.13}) and Lemma \ref{lem:tightrest}. On the other hand,  by similar calculations as in \eqref{eq:tighty2}, property \ref{I} and  $\int_{\R}x^2F(dx)<\infty$,
we have $	\mathbb{E}(|y^{n,m}_{i,k}(1,2)||\mathcal{F}_{t^{1}_i})\le \frac{\delta_n\lambda_{n,m}}{nm}$. Then we conclude using  $\delta_n\le C\log{1/\beta_n}$ (see \eqref{eq:3.13}), $\lambda_{n,m}\le\frac{C}{n\beta_n^{\alpha}}$ (see \eqref{h1}), with the choice $\beta_n=\frac{\log{n}}{n^{1/(2\alpha)}}$, criteria \eqref{eq:2.7} and Lemma \ref{lem:tightrest}. Therefore, we have $\overline{Y}^n(1)\stackrel{\mathbb{P}}{\rightarrow}0$.
Now, from the formula of the rest term given by \eqref{eq:MR5}, we have 
 \begin{align*}
 \hspace{-0.4cm}\mathcal{R}^{n,m}_t={\Gamma}^n_t(1,2)+\sum_{i=2}^5\Gamma^n_t(i).
 \end{align*}
Now, we will prove that each term converges uniformly in probability to $0$ when $n\rightarrow\infty$.\\
 \underline{\textbf{The term ${\Gamma}^n_t(1,2)$}}:
 Let us recall that  ${\Gamma}^n_t(1,2)=\sum_{k=1}^{m}\sum_{i=1}^{[nt]}{\zeta}_{i,k}^n(1,2)$, where  $${\zeta}_{i,k}^n(1,2)=ff'(X_{t_i^1})\sum_{j=2}^{K(t_{k}^i)}\Delta Y_{T_j(t_{k}^i)}\tilde{M}^{n,m}_{i,k}\quad\textrm{with}\quad \tilde{M}^{n,m}_{i,k}=(M_{t^{m+1}_i}^{\beta_{n}}-M_{t^{1}_i}^{\beta_{n}})-(M_{t^{k+1}_i}^{\beta_{n}}-M_{t^{k}_i}^{\beta_{n}}).$$ 
 For $k\in\{1,\dots,m\}$ fixed, first, by Lemma \ref{lem:smallj} we have
 $	\mathbb{E}((\tilde{M}^{n,m}_{i,k})^2|\mathcal{F}_{t^{1}_i})
 	\leq \frac{Cc_{n}}{n}.
 $
As $ff'$ is bounded, using Jensen's inequality, the independence of $\{\Delta Y_{T_j(t_{k}^i)},j\ge1, K(t_{k}^i)\}$ and $\tilde{M}^{n,m}_{i,k}$, property \ref{I} and $\int_{\R}x^2F(dx)<\infty$, we get that $\mathbb{E}({\zeta}_{i,k}^n(1,2)|\mathcal{F}_{t^{1}_i})=0$ and
 \begin{align*}
 &	\mathbb{E}(({\zeta}_{i,k}^n(1,2))^2|\mathcal{F}_{t^{1}_i}){\leq}Cu_{n,m}^2\mathbb{E}(K(t_i^k)\sum_{j=2}^{K(t_i^k)}(\Delta Y_{T^{\beta_n}_j(t_i^k)})^2|\mathcal{F}_{t^{1}_i})\mathbb{E}(({\tilde{M}}^{n,m}_{i,k})^2|\mathcal{F}_{t^{1}_i})\\
 	&{\leq} \frac{Cu_{n,m}^2c_n}{n\theta(\beta_{n})}\mathbb{E}(K(t_i^k)(K(t_i^k)-1)\1_{\{K(t_i^k)\ge2\}}|\mathcal{F}_{t^{1}_i})\int_{|x|>\beta_{n}}x^2F(dx){\leq}\frac{Cu_{n,m}^2c_n\lambda_{n,m}}{n^2}.
 \end{align*}
  Then, we conclude by $c_n\le C\beta_n^{2-\alpha}$ (see $\eqref{eq:3.13}$), $\lambda_{n,m}\le \frac{C}{n\beta_n^{\alpha}}$ (see \eqref{h1}), the choices $u_{n,m}=\left[\frac{mn}{(m-1)\log{n}}\right]^{1/\alpha}$ and $\beta_n=\frac{\log{n}}{n^{1/(2\alpha)}}$ with $\alpha>1$, criteria \eqref{eq:2.8} and Lemma \ref{lem:tightrest}. Therefore, we get  ${\Gamma}^n(1,2)\stackrel{\mathbb{P}}{\rightarrow}0$.\\
\underline{\textbf{The term $\Gamma^n_t(2)$}}:
Let us  recall that    $\Gamma^n_t(2)=\sum_{k=2}^{m}\sum_{i=1}^{[nt]}\zeta_{i,k}^n(2)$, with   $$\zeta_{i,k}^n(2)=u_{n,m}[\frac{(k-1)d_{n}}{nm}+M_{t_i^k}^{\beta_{n}}-M_{t^{1}_i}^{\beta_{n}}](\frac{d_{n}}{nm}+M_{t^{k+1}_i}^{\beta_{n}}-M_{t_i^k}^{\beta_{n}}).$$
 For $k\in\{2,\dots,m\}$ fixed, by the independence of the increments and Lemma \ref{lem:smallj}, we get $\mathbb{E}(\zeta_{i,k}^n(2)|\mathcal{F}_{t^{1}_i})	
 =\frac{(k-1)u_{n,m} d_{n}^2}{n^2m^2}$ and
\begin{align*}
	&\mathbb{E}(|\zeta_{i,k}^n(2)|^2|\mathcal{F}_{t^{1}_i})	\leq C u_{n,m}^2(\mathbb{E}(M_{t_i^k}^{\beta_{n}}-M_{t^{1}_i}^{\beta_{n}})^2+\frac{d_n^2}{n^2})(\mathbb{E}(M_{t^{k+1}_i}^{\beta_{n}}-M_{t_i^k}^{\beta_{n}})^2+\frac{d_{n}^2}{n^2})\leq Cu_{n,m}^2\left(\frac{c_{n}^2}{n^2}+\frac{d_{n}^4}{n^4}\right).
\end{align*}
Then, we conclude by $c_n\le C\beta_n^{2-\alpha}$, $d_n\le C\beta_n^{1-\alpha}$ (see $\eqref{eq:3.13}$),  the choices $u_{n,m}=\left[\frac{mn}{(m-1)\log{n}}\right]^{1/\alpha}$ and $\beta_n=\frac{\log{n}}{n^{1/(2\alpha)}}$ with $\alpha>1$, criteria \eqref{eq:2.8} and Lemma \ref{lem:tightrest}. Therefore, we get  $\Gamma^n(2)\stackrel{\mathbb{P}}{\rightarrow}0$.\\
\underline{\textbf{The term $\Gamma^n_t(3)$}}:
Let us recall that $\Gamma^n_t(3)=\sum_{k=2}^m\sum_{i=1}^{[nt]}\zeta_{i,k}^n(3)$ with $$\zeta_{i,k}^n(3)=\frac{u_{n,m}d_n(k-1)}{nm}(N_{t^{k+1}_i}^{\beta_{n}}-N_{t_i^k}^{\beta_{n}})+\frac{u_{n,m}d_n}{nm}(N_{t_i^k}^{\beta_{n}}-N_{t^{1}_i}^{\beta_{n}}).$$ For $k\in\{2,\dots,m\}$ fixed, by the independence of the increments and Lemma \ref{lem:bigj}, we get
\begin{align*}
	\mathbb{E}(|\zeta_{i,k}^n(3)||\mathcal{F}_{t^{1}_i})\leq \frac{Cu_{n,m}|d_n|}{n}(\mathbb{E}|N_{t^{k+1}_i}^{\beta_{n}}-N_{t_i^k}^{\beta_{n}}|+\mathbb{E}|N_{t_i^k}^{\beta_{n}}-N_{t^{1}_i}^{\beta_{n}}|)
	\leq
	\frac{Cu_{n,m}|d_{n}|\delta_{n}}{n^2}.
\end{align*}
Then, we conclude using $\delta_n$ and $d_n$ are bounded by $C\beta_n^{1-\alpha}$ (see $\eqref{eq:3.13}$),  the choices $u_{n,m}=\left[\frac{mn}{(m-1)\log{n}}\right]^{1/\alpha}$ and $\beta_n=\frac{\log{n}}{n^{1/(2\alpha)}}$, criteria \eqref{eq:2.7} and Lemma \ref{lem:tightrest}. Therefore, $\Gamma^n(3)\stackrel{\mathbb{P}}{\rightarrow}0$. \\
\underline{\textbf{The term $\Gamma^n_t(4)$}}:
Let us recall that  $\Gamma^n_t(4)=\sum_{k=2}^m\sum_{i=1}^{[nt]}\zeta_{i,k}^n(4)$ where $$\zeta_{i,k}^n(4)=u_{n,m}(N_{t_i^k}^{\beta_{n}}-N_{t^{1}_i}^{\beta_{n}})(N_{t^{k+1}_i}^{\beta_{n}}-N_{t_i^k}^{\beta_{n}}).$$ For $k\in\{2,\dots,m\}$ fixed, by the independence of the increments and Lemma \ref{lem:bigj},  we have
\begin{align*}
	\mathbb{E}(|\zeta_{i,k}^n(4)||\mathcal{F}_{t^{1}_i})\le Cu_{n,m}\mathbb{E}(|N_{t_i^k}^{\beta_{n}}-N_{t^{1}_i}^{\beta_{n}}|)\mathbb{E}(|N_{t^{k+1}_i}^{\beta_{n}}-N_{t_i^k}^{\beta_{n}}|)\le C\frac{u_{n,m}\delta_{n}^2}{n^2}.
\end{align*}
Then, we conclude using $\delta_n\le C\beta_n^{1-\alpha}$ (see $\eqref{eq:3.13}$), with our choice of $u_{n,m}$ and $\beta_n$, criteria \eqref{eq:2.7} and Lemma \ref{lem:tightrest}. Therefore, we get $\Gamma^n(4)\stackrel{\mathbb{P}}{\rightarrow}0$. \\
\underline{\textbf{The term $\Gamma_t^n(5)$}}:
Since $k(x,y)$ is bounded on $\R^2$,  it is enough to prove that for $k\in\{1,\dots,m\}$, the following nine  triangular arrays with generic terms $\{{\zeta}_{i,k}^n(5,j),j\in\{1,\dots,9\}$  converge to $0$ as $n\rightarrow\infty$ with
\begin{align*}
	\left\{\begin{array}{ll}
		{\zeta}_{i,k}^n(5,1)=u_{n,m}(A^{\beta_n}_{t_i^k}-A^{\beta_n}_{t^{1}_i})^2|A^{\beta_n}_{t^{k+1}_i}-A^{\beta_n}_{t_i^k}|,&
		{\zeta}_{i,k}^n(5,2)=u_{n,m}(A^{\beta_n}_{t_i^k}-A^{\beta_n}_{t^{1}_i})^2	|M^{\beta_n}_{t^{k+1}_i}-M^{\beta_n}_{t_i^k}|\\
		{\zeta}_{i,k}^n(5,3)=u_{n,m}(A^{\beta_n}_{t_i^k}-A^{\beta_n}_{t^{1}_i})^2	|N^{\beta_n}_{t^{k+1}_i}-N^{\beta_n}_{t_i^k}|,&
		{\zeta}_{i,k}^n(5,4)=u_{n,m}(M^{\beta_n}_{t_i^k}-M^{\beta_n}_{t^{1}_i})^2	|A^{\beta_n}_{t^{k+1}_i}-A^{\beta_n}_{t_i^k}|\\
		{\zeta}_{i,k}^n(5,5)=u_{n,m}(N^{\beta_n}_{t_i^k}-N^{\beta_n}_{t^{1}_i})^2	|M^{\beta_n}_{t^{k+1}_i}-M^{\beta_n}_{t_i^k}|,&
		{\zeta}_{i,k}^n(5,6)=u_{n,m}(M^{\beta_n}_{t_i^k}-M^{\beta_n}_{t^{1}_i})^2	|N^{\beta_n}_{t^{k+1}_i}-N^{\beta_n}_{t_i^k}|\\
		{\zeta}_{i,k}^n(5,7)=u_{n,m}(M^{\beta_n}_{t_i^k}-M^{\beta_n}_{t^{1}_i})^2	|M^{\beta_n}_{t^{k+1}_i}-M^{\beta_n}_{t_i^k}|,&
		{\zeta}_{i,k}^n(5,8)=u_{n,m}(N^{\beta_n}_{t_i^k}-N^{\beta_n}_{t^{1}_i})^2	|N^{\beta_n}_{t^{k+1}_i}-N^{\beta_n}_{t_i^k}|\\{\zeta}_{i,k}^n(5,9)=u_{n,m}(N^{\beta_n}_{t_i^k}-N^{\beta_n}_{t^{1}_i})^2	|A^{\beta_n}_{t^{k+1}_i}-A^{\beta_n}_{t_i^k}|.&
	\end{array}\right.
\end{align*}
From  \eqref{eq:gamma5i}, the triangular arrays with generic terms ${\zeta}_{i,k}^n(5,i)$, $i\in\{1,\dots,9\}\backslash\{5\}$ are bounded as follows
\begin{align*}
	\left\{\begin{array}{ll}
		\mathbb{E}(|{\zeta}_{i,k}^n(5,1)||\mathcal{F}_{t^{1}_i})\le C\frac{u_{n,m}|d_n|^3}{n^3},
		&\mathbb{E}(|{\zeta}_{i,k}^n(5,2)||\mathcal{F}_{t^{1}_i}){\le} C\frac{u_{n,m}|d_n|^2\sqrt{c_n}}{n^2\sqrt{n}}\\
		\mathbb{E}(|{\zeta}_{i,k}^n(5,3)||\mathcal{F}_{t^{1}_i}){\le} C\frac{u_{n,m}|d_n|^2\delta_n}{n^3},
		&\mathbb{E}(|{\zeta}_{i,k}^n(5,4)||\mathcal{F}_{t^{1}_i})\le C\frac{u_{n,m}|d_n|c_n}{n^2}\\
		\mathbb{E}(|{\zeta}_{i,k}^n(5,6)||\mathcal{F}_{t_i^1})\le C \frac{u_{n,m}c_n\delta_n}{n^2},
		&\mathbb{E}(|{\zeta}_{i,k}^n(5,7)||\mathcal{F}_{t^{1}_i})\le C \frac{u_{n,m}c_n^{3/2}}{n\sqrt{n}}\\
		\mathbb{E}(|{\zeta}_{i,k}^n(5,8)||\mathcal{F}_{t^{1}_i})\le C \frac{u_{n,m}\delta_n}{n^2},
		&\mathbb{E}(|{\zeta}_{i,k}^n(5,9)||\mathcal{F}_{t^{1}_i})\le C\frac{u_{n,m}|d_n|}{n^2}.
	\end{array}\right.
\end{align*}
Then, by $u_{n,m}=\left[\frac{nm}{(m-1)\log{n}}\right]^{1/\alpha}$ , $d_n\le C\beta_n^{1-\alpha}$, $\delta_n\le C\beta_n^{1-\alpha}$ and $c_n\le C\beta_n^{2-\alpha}$ (see \eqref{eq:3.13}) with $\beta_{n}=\frac{\log{n}}{n^{1/(2\alpha)}}$, criteria \eqref{eq:2.7} and Lemma \ref{lem:tightrest}, we conclude that these triangular arrays with generic terms ${\zeta}_{i,k}^n(5,i)$, $i\in\{1,\dots,9\}\backslash\{5\}$ converge to $0$. Next, concerning  ${\zeta}_{i,k}^n(5,5)$, we reuse the notations in section \ref{section1} and rewrite ${\zeta}_{i,k}^n(5,5)={\zeta'}_{i,k}^n(5,5)+{\zeta''}_{i,k}^n(5,5)+{\zeta'''}_{i,k}^n(5,5)$ where
\begin{align*}\left\{\begin{array}{l}
{\zeta'}_{i,k}^n(5,5)=u_{n,m}\Delta Y_{T^{\beta_n}_1(t_i^1,t_i^k)}^2\1_{\{K(t_i^1,t_i^k)\ge 1\}}|M^{\beta_n}_{t^{k+1}_i}-M^{\beta_n}_{t_i^k}|,\\
{\zeta''}_{i,k}^n(5,5)=u_{n,m}\sum_{j=2}^{K(t_i^1,t_i^k)}\Delta Y_{T^{\beta_n}_j(t_i^1,t_i^k)}^2|M^{\beta_n}_{t^{k+1}_i}-M^{\beta_n}_{t_i^k}|,\\
{\zeta'''}_{i,k}^n(5,5)=u_{n,m}\sum_{\substack{j,j'=1\\j\neq j'}}^{K(t_i^1,t_i^k)}|\Delta Y_{T^{\beta_n}_j(t^{1}_i,t_i^k)}||\Delta Y_{T^{\beta_n}_{j'}(t^{1}_i,t_i^k)}||M^{\beta_n}_{t^{k+1}_i}-M^{\beta_n}_{t_i^k}|.\end{array}\right.
\end{align*}
Concerning  ${\zeta''}_{i,k}^n(5,5)$ and ${\zeta'''}_{i,k}^n(5,5)$, by   the independence between the martingale increment and the jumps with size bigger than $\beta_n$, Cauchy-Schwarz's inequality and Lemma \ref{lem:smallj} for term $|M^{\beta_n}_{t^{k+1}_i}-M^{\beta_n}_{t_i^k}|$ and the same estimates as done for the treatment of  $\overline{\Gamma}^n(1,2)$ in the proof of Lemma \ref{lem:rest1}, we have 
\begin{align*}
	\left\{\begin{array}{l}
		\mathbb{E}(	|{\zeta''}_{i,k}^n(5,5)||\mathcal{F}_{t^{1}_i})\le Cu_{n,m}\sqrt{\frac{c_n}{n}}\mathbb{E}\left(\sum_{h=2}^{K(t^{1}_i,t_i^k)}\Delta Y_{T^{\beta_n}_h(t^{1}_i,t_i^k)}^2|\mathcal{F}_{t^{1}_i}\right)\le  \frac{Cu_{n,m}\sqrt{c_n}\lambda_{n,m}}{n\sqrt{n}}\\
		\mathbb{E}(|{\zeta'''}_{i,k}^n(5,5)||\mathcal{F}_{t^{1}_i})\le  \frac{Cu_{n,m}\sqrt{c_n}}{\sqrt{n}}\mathbb{E}\left(\sum_{\substack{j,j'=1\\j\neq j'}}^{K(t^{1}_i,t_i^k)}|\Delta Y_{T^{\beta_n}_j(t^{1}_i,t_i^k)}\Delta Y_{T^{\beta_n}_{j'}(t^{1}_i,t_i^k)}||\mathcal{F}_{t^{1}_i}\right)\le \frac{Cu_{n,m}\sqrt{c_n}\delta_{n}^2}{n^2\sqrt{n}}\end{array}\right.
\end{align*} 
Then, for our choice of $u_{n,m}$ and $\beta_{n}$, we use $\delta_n\le C\beta_n^{1-\alpha}$ and $c_n\le C\beta_n^{2-\alpha}$ (see \eqref{eq:3.13}), $\lambda_{n,m}\le \frac{C}{n\beta_n^{\alpha}}$ (see \eqref{h1}), criteria \eqref{eq:2.7} and Lemma \ref{lem:tightrest} to get the convergence to $0$ of the triangular arrays with generic terms ${\zeta''}_{i,k}^n(5,5)$, ${\zeta'''}_{i,k}^n(5,5)$.
Concerning the term ${\zeta'}_{i,k}^n(5,5)$, using the independence structure, properties \ref{I} and \ref{M}, we see that 
\begin{align*}
	\mathbb{E}(e^{\texttt{i}v{\zeta'}_{i,k}^n(5,5)}|\mathcal{F}_{t^{1}_i})=e^{-(k-1)\lambda_{n,m}}+\frac{1-e^{-(k-1)\lambda_{n,m}}}{\theta(\beta_n)}\int_{|x|>\beta_n}	\mathbb{E}(e^{\texttt{i}vu_{n,m}x^2|M^{\beta_n}_{t^{k+1}_i}-M^{\beta_n}_{t^{k}_i}|})F(dx)
\end{align*}
Let us denote ${z'}_{n,m}(x,v)=\frac{1}{nm}\int_{|y|\le\beta_n}(e^{\texttt{i}vu_{n,m}x^2|y|}-1-\texttt{i}vu_{n,m}x^2|y|)F(dy)$. Then we have
\begin{align*}
	\mathbb{E}(e^{\texttt{i}v{\zeta'}_{i,k}^n(5,5)}|\mathcal{F}_{t^{1}_i})&=e^{-(k-1)\lambda_{n,m}}+\frac{1-e^{-(k-1)\lambda_{n,m}}}{\theta(\beta_n)}\int_{|x|>\beta_n}	e^{{z'}_{n,m}(x,v)}F(dx)\\
	&=1+\frac{1-e^{-(k-1)\lambda_{n,m}}}{\theta(\beta_n)}\int_{|x|>\beta_n}	(e^{{z'}_{n,m}(x,v)}-1)F(dx).
\end{align*} 
By similar calculations as in \eqref{eq:tool}, we easily get $|{z'}_{n,m}(x,v)|\le \frac{C}{n}|vu_{n,m}x^2|^{\alpha}$ and the suprema of $|{z'}_{n,m}(x,v)|$ over all  $|x|\le p$ and $|v|\le 1$ goes to $0$ as $n$ tends to $0$. Now, since  $|e^{{z'}_{n,m}(x,v)}-1|\le C |{z'}_{n,m}(x,v)|$ and $x\mapsto|x|^{2\alpha}$ is $F$-integrable (see Remark \ref{rmk:1}), we have 
\begin{align*}
	|\mathbb{E}(e^{\texttt{i}v{\zeta'}_{i,k}^n(5,5)}|\mathcal{F}_{t^{1}_i})-1|\le \frac{C|v|^{\alpha}u_{n,m}^{\alpha}}{n^2}.	
\end{align*}
	Then, for our choice of $u_{n,m}$, ${\zeta'}_{i,k}^n(5,5)$ satisfies \eqref{eq:2.12} with $\xi{'''}_{n,v}=\frac{C|v|^{\alpha}u_{n,m}^{\alpha}}{n}$ which converges to $0$ as $n\rightarrow\infty$  for all $v\le1$. Then, by Lemma \ref{lem:add} and Lemma \ref{lem:tightrest}, we get the convergence for the last triangular array.	Therefore, we get $\Gamma^n(5)\stackrel{\mathbb{P}}{\rightarrow}0$.
\section{Proof of some equalities and inequalities}\label{app:B}
First, we observe from the L\'evy measure two relations followed: \\
\paragraph*{$\bullet$ \textit{On the positive side}} For all $0\leq a<b\leq 1$ and $\gamma>0$, we have
\begin{equation}
	\int_{a<x\leq b}|x|^{\gamma}F(dx)=\gamma\int_0^by^{\gamma-1}(\theta_+(y\vee a)-\theta_+(b))dy\label{eq:3.2.1}
\end{equation}
\begin{equation}\label{eq:3.6.1}
	\textrm{and}\quad\int_{a<x\leq b}x\log{x}F(dx)=\int_0^b(1+\log{y})(\theta_+(y\vee a)-\theta_+(b))dy.
\end{equation}
\begin{proof}
	The proofs  are simply by Fubini. Concerning \eqref{eq:3.2.1}, it r.h.s. is equal to
	\begin{align*}
		&\hskip 0.5cm \gamma\int_0^by^{\gamma-1}\left(\int_{x>y\vee a}F(dx)-\int_{x>b}F(dx)\right)dy=\gamma\int_0^by^{\gamma-1}\int_{y\vee a<x\le b}F(dx)dy\\&=\gamma\left(\int_0^ay^{\gamma-1}\int_{a<x\leq b}F(dx)dy+\int_a^by^{\gamma-1}\int_{y<x\leq b}F(dx)dy\right)\\
		&=\gamma\left(\int_{a<x\leq b}F(dx)\int_0^ay^{\gamma-1}dy+\int_{a<x\leq b}F(dx)\int_a^xy^{\gamma-1}dy\right)\\
		&=\int_{a<x\leq b}a^{\gamma}F(dx)+\int_a^b(x^{\gamma}-a^{\gamma})F(dx)=\int_{a<x\leq b}x^{\gamma}F(dx).
	\end{align*}	 
Concerning \eqref{eq:3.6.1}, its r.h.s is equal to 
\begin{align*}
&\hskip 0.5cm\int_0^a(1+\log{y})(\theta_+( a)-\theta_+(b))dy+\int_a^b(1+\log{y})(\theta_+(y)-\theta_+(b))dy\\
&=(\theta_+( a)-\theta_+(b))a\log{a}+\int_a^b\int_{x>y}(1+\log{y})F(dx)dy-\theta_+(b)(b\log{b}-a\log{a})\\
&=\theta_+( a)a\log{a}-\theta_+(b)b\log{b}+\int_{x>a}\int_{a<y\le x\wedge b}(1+\log{y})dyF(dx)\\
&=-\theta_+(b)b\log{b}+\int_{x>a}(x\wedge b)\log{(x\wedge b)}F(dx)=\int_{a<x\leq b}x\log{x}F(dx).
\end{align*}	
\end{proof}	
\paragraph*{$\bullet$ \textit{On the negative side}} For all $-1\leq a<b\leq 0$ and $\gamma>0$, we have
\begin{equation}
	\int_{a\le x< b}|x|^{\gamma}F(dx)=\gamma\int_a^0(-y)^{\gamma-1}(\theta_-((-y)\vee (-b))-\theta_-(-a))dy.\label{eq:3.2.2}
\end{equation}
\begin{equation}\label{eq:3.6.2}
	\textrm{and}\quad\int_{a\le x< b}|x|\log{|x|}F(dx)=\int_a^0(1+\log{|y|})(\theta_-((-y)\vee (-b))-\theta_-(-a))dy.
\end{equation}
\begin{proof}
Here, also, the proofs  are simply by Fubini. Concerning \eqref{eq:3.2.2}, it r.h.s. is equal to
	\begin{align*}
		&\hskip 0.5cm \gamma\int_a^0(-y)^{\gamma-1}\left(\int_{x<y\wedge b}F(dx)-\int_{x<a}F(dx)\right)dy
		=\gamma\int_a^0(-y)^{\gamma-1}\int_{a\le|x|<y\wedge b}F(dx)dy\\&=\gamma\left(\int_a^b(-y)^{\gamma-1}\int_{a\le|x|<y}F(dx)dy+\int_b^0(-y)^{\gamma-1}\int_{a\le |x|<b}F(dx)dy\right)\\
		&=\gamma\left(\int_{a\le x< b}F(dx)\int_x^b(-y)^{\gamma-1}dy+\int_{a\le x< b}F(dx)\int_b^0(-y)^{\gamma-1}dy\right)\\
		&=\int_{a\le x< b}((-x)^{\gamma}-(-b)^{\gamma})F(dx)+\int_{a\le x< b}F(dx)(-b)^{\gamma}=\int_{a\le x< b}(-x)^{\gamma}F(dx).
	\end{align*}	 
Concerning \eqref{eq:3.6.2}, its r.h.s is equal to 
\begin{align*}
	&\hskip 0.5cm\int_a^b(1+\log{y})(\theta_-( -y)-\theta_-(-a))dy+\int_b^0(1+\log{y})(\theta_-(-b)-\theta_-(-a))dy\\
	&=\int_a^b\int_{x<y}(1+\log{y})F(dx)dy-\theta_-(-a)(b\log{(-b)}-a\log{(-a)})-(\theta_-( -b)-\theta_-(-a))b\log{(-b)}\\
	&=\int_{x< b}\int_{x\vee a<y\le  b}(1+\log{y})dyF(dx)+\theta_-( -a)a\log{(-a)}-\theta_-(-b)b\log{(-b)}\\
	&=\int_{x<b}(x\vee a)\log{(-(x\vee a))}F(dx)+\theta_-( -a)a\log{(-a)}=\int_{a\le x< b}|x|\log{|x|}F(dx).
\end{align*}	
\end{proof}	
\begin{proof}[Proof of Lemma \ref{lem:estimation1}]
	Taking advantage of $\eqref{eq:3.2.1}$ and $\eqref{eq:3.2.2}$, we have
	\begin{align*}
		&c(\beta)=\int_{0<x\leq\beta}x^2F(dx)+\int_{-\beta<x \leq 0}x^2F(dx)\\
		&=2\int_0^{\beta}y(\theta_+(y)-\theta_+(\beta))dy+2\int_{-\beta}^0(-y)(\theta_-(-y)-\theta_-(\beta))dy\\
		&=2\int_0^{\beta}y(\theta_+(y)-\theta_+(\beta))dy+2\int_0^{\beta}y(\theta_-(y)-\theta_-(\beta))dy=2\int_0^{\beta}y(\theta(y)-\theta(\beta))dy.\label{*}\tag{*}
		\end{align*}
	In other hand, by \eqref{h1}
$
		|\theta(y)-\theta(\beta)|\leq|\theta(y)|+|\theta(\beta)|\leq Cy^{-\alpha}+C\beta^{-\alpha}.
$
	By $(\ref{*})$, 
$
	c(\beta)\leq\left(\frac{2C}{2-\alpha}+C\right)\beta^{2-\alpha}.
$
		Similar proofs are easily deduced from $\eqref{eq:3.2.1}$ and $\eqref{eq:3.2.2}$ for the other formulas.
\end{proof}
\begin{proof}[Proof of Lemma \ref{lem:estimation2}]
	By $(\ref{*})$
$
		c(\beta)\sim 2\int_0^{\beta}y\left(\frac{\theta}{y^{\alpha}}-\frac{\theta}{\beta^{\alpha}}\right)dy=\frac{\alpha\theta}{2-\alpha}\beta^{2-\alpha}.
	$
			Here, we can also obtain similar proofs for the other formulas.
	\end{proof}
\begin{proof}[Proof of Lemma \ref{lem:estimation3}]
	Applying \eqref{eq:3.6.1}  we have
	\begin{align*}
	&\frac{1}{(\log{(1/\beta)})^2}\int_{\beta<x\le b}(x\log{x})F(dx)= 	\frac{1}{(\log{(1/\beta)})^2}\int_0^ b(1+\log{y})(\theta_+(y\vee \beta)-\theta_+(b))dy\\
	&\sim \frac{1}{(\log{(1/\beta)})^2}\theta_+( \beta)\int_0^{\beta}(1+\log{|y|})dy+\frac{1}{(\log{(1/\beta)})^2}\int_{\beta}^b(1+\log{y})\theta_+(y)dy.
	\end{align*}
Considering the first term, for $\beta\to0$ we deduce from \eqref{h2}  with $\alpha=1$ that  $\frac{1}{(\log{(1/\beta)})^2}\theta_+( \beta)\int_0^{\beta}(1+\log{|y|})dy\sim \frac{\log{\beta}}{(\log{(1/\beta)})^2}\theta_+\underset{n\rightarrow\infty}{\longrightarrow}0$. Now, let $\varepsilon>0$, there exists a $\varepsilon'\in(0,1)$ such that $\beta\in(0,\varepsilon')$ and we have $|\frac{\beta\theta_+(\beta)}{\theta_+}-1|\le \varepsilon$. Considering the second term, we rewrite $\frac{1}{(\log{(1/\beta)})^2}\int_{\beta}^b(1+\log{y})\theta_+(y)dy=x_n+y_n$ where
\begin{align*}
	x_n=\frac{1}{(\log{(1/\beta)})^2}\int_{\varepsilon'}^b(1+\log{y})\theta_+(y)dy,\quad y_n=\frac{1}{(\log{(1/\beta)})^2}\int_{\beta}^{\varepsilon'}(1+\log{y})\theta_+(y)dy.
	\end{align*}
On the one hand, $x_n$ is bounded by $\frac{\theta_+(\varepsilon')}{(\log{(1/\beta)})^2}\int_{\varepsilon'}^b(1+\log{y})dy$ which converges to $0$ as $n\to\infty$. On the other hand, if we denote ${y'}_n=\frac{\theta_+}{(\log{(1/\beta)})^2}\int_{\beta}^{\varepsilon'}(1+\log{y})y^{-1}dy$, we have ${y'}_n(1-\varepsilon)\le y_n\le {y'}_n(1+\varepsilon)$ for any $\varepsilon$ arbitrarily small, then $y_n\sim{y'}_n\sim \frac{-(\log{\beta})^2}{2(\log{(1/\beta)})^2}\theta_+\underset{n\rightarrow\infty}{\longrightarrow}-\frac{\theta_+}{2}$. Therefore, it is clear that $\frac{1}{(\log{(1/\beta)})^2}\int_{\beta<x\le \varepsilon'}(x\log{x})F(dx)\underset{n\rightarrow\infty}{\longrightarrow}-\frac{\theta_+}{2}$. Similarly, applying \eqref{eq:3.6.2}, we get $\frac{1}{(\log{(1/\beta)})^2}\int_{-b\le x< \beta}((-x)\log{(-x)})F(dx)\underset{n\rightarrow\infty}{\longrightarrow}-\frac{\theta_-}{2}$ which completes the proof.		
\end{proof}

	
	\section{Some general tools}\label{appC}
	\subsection{Uniformly tight processes}
	We  recall the definition of uniformly tight property $(UT)$ defined in Jakubowski, M\'emin and Pag\`es (1989) \cite{f}. Let $Z^n$ be a sequence of semimartingale, with the canonical decompositions 
	\begin{equation}
	Z_t^n=A_t^{n,a}+M_t^{n,a}+\sum_{s\leq t}\Delta Z_s^n1_{\{|\Delta Z_s^n|>a\}},\label{eq:2.3}
	\end{equation} 
	where $a>0$ and $A^{n,a}$ is a predictable process with locally bounded variation and $M^{n,a}$ is a (locally bounded) local martingale. Then we say that the sequence $(Z^n)$ satisfies $(UT)$ if for any $t<\infty$, the sequence of real-valued random variables
	\begin{displaymath}
	Var(A^{n,a})_t+\langle M^{n,a},M^{n,a}\rangle_t+\sum_{s\leq t}|\Delta Z_t^n|1_{\{|\Delta Z_s^n|>a\}}
	\end{displaymath}
	is tight. This property does not depend on the choice of $a\in(0,\infty)$.\\	
	If a sequence is $(UT)$ then it has some other important properties as in the theorem below, which can also be found in \cite{g} or in Theorem 2.3 of \cite{b}. 
	\begin{theorem} \label{thm:ut}
		Let $X^n$ and $Y^n$ be two sequences of semi-martingales,
		\begin{enumerate}[(i)]
			\item If both sequences $X^n$ and $Y^n$ are $(UT)$, then so has the sequence $X^n+Y^n$.
			\item Let $H^n$ be a sequence of predictable processes such that the sequence $\sup_{s\leq t}|H^n_s|$ is tight. If the sequence $X^n$ is $(UT)$, so is the sequence $\int_0^.H^n_sdX^n_s$.\\
			\item  Suppose that $X^n$ weakly converges. Then $(UT)$ is necessary and sufficient for the following property:\\
			For any sequence of adapted c\`adl\`ag  $H^n$ processes such that the sequence $(H^n,X^n)$ weakly converges to $(H,X)$, then $X$ is a semi-martingale with respect to the filtration generated by the process $(H,X)$, and we have $(H^n, X^n,\int_0^.H^n_{s-}dX^n_s)\stackrel{\mathcal{L}}{\longrightarrow}(H,X,\int_0^.H_{s-}dX_s)$.
		\end{enumerate} 	
	\end{theorem}
		
	\subsection{Triangular arrays }
	
	Now concerning sums of triangular arrays of the form
	\begin{displaymath}
	\Gamma_t^n=\sum_{i=1}^{[nt]}\zeta_i^n,
	\end{displaymath}
	where for each $n$ we have $\R^d$-valued random variables $(\zeta_i^n)_{i\geq1}$ such that each $\zeta_i^n$ is $\mathcal{F}_{i/n}$-measurable. Below we give various conditions recalled in \cite{a} insuring tightness or convergence of the sequence $(\Gamma^n)$.
	\begin{align}
		&\hskip 0.5cm\mathds{E}(|\zeta_i^n||\mathcal{F}_{(i-1)/n})\leq\frac{\xi_n}{n},\label{eq:2.7}\\
		&\left\{\begin{array}{l}	\label{eq:2.8}	|\mathds{E}(\zeta_i^n|\mathcal{F}_{(i-1)/n})|\leq\frac{\xi_n}{n},\\
			\mathds{E}(|\zeta_i^n|^2|\mathcal{F}_{(i-1)/n})\leq\frac{{\xi'}_n}{n},\end{array}\right.\\
		&\left\{\begin{array}{l}\label{eq:2.9}|\mathds{E}(\zeta_i^n\1_{\{|\zeta_i^n|\leq 1\}}|\mathcal{F}_{(i-1)/n})|\leq\frac{\xi_n}{n},\\	
			\mathds{E}(|\zeta_i^n|^2\1_{\{|\zeta_i^n|\leq 1\}}|\mathcal{F}_{(i-1)/n})\leq\frac{{\xi'}_n}{n},\\\mathds{P}(|\zeta_i^n|>y|\mathcal{F}_{(i-1)/n})\leq\frac{{\xi''}_{n,y}}{y},\hskip 1cm \forall y>1.\end{array}\right.
	\end{align}
	Note that $\eqref{eq:2.8}$ with $\hat{\xi}_n$ and $\hat{\xi'}_n$ implies $\eqref{eq:2.9}$ with $\xi_n=\hat{\xi}_n+\hat{\xi'}_n$ and ${\xi'}_n=\hat{\xi'}_n$ and ${\xi"}_{n,y}=\hat{\xi'}_n/y^2$ (the last is from extended version of Markov inequality for monotonically increasing functions). Also, $\eqref{eq:2.7}$ with $\hat{\xi}_n$ implies $\eqref{eq:2.9}$ with $\xi_n={\xi'}_n=\hat{\xi}_n$ and ${\xi"}_{n,y}=\hat{\xi}_n/y$.\\
	By $\Gamma^n\stackrel{\mathds{P}}{\longrightarrow}0$, we mean that $\sup_{s\leq t}|\Gamma_s^n|$ goes to $0$ in probability for all $t$.	
	\begin{lemma}\label{lem:tightrest}(Lemma 2.5 in \cite{a})
		\begin{enumerate}[(a)]
			\item For $\Gamma^n\stackrel{\mathds{P}}{\rightarrow}0$, it is enough that either $\eqref{eq:2.7}$ or $\eqref{eq:2.8}$ or $\eqref{eq:2.9}$ hold with 
			\begin{align}\label{eq:2.10}
				\lim_n\xi_n=0,\hskip 1cm \lim_n{\xi'}_n=0,\hskip1cm \lim_n{\xi''}_{n,y}=0 \hskip 1cm\forall y>1.
			\end{align}
			\item For the sequence $(\Gamma^n)$ to be tight for the Skorokhod topology, it is enough that the sequence of each of the $d$ components of $\zeta_i^n$ satisfies either $\eqref{eq:2.7}$ or $\eqref{eq:2.8}$ or $\eqref{eq:2.9}$ with 
			\begin{align}
			\limsup_n\xi_n<\infty,\hskip0.5cm\limsup_n{\xi'}_n<\infty,\hskip0.5cm\limsup_n{\xi''}_{n,y}<\infty,\hskip0.5cm\lim_{y\uparrow\infty}\limsup_n{\xi''}_{n,y}=0.\label{eq:2.11}
			\end{align}	
		\end{enumerate}	
	\end{lemma}
	\begin{lemma}\label{lem:add}{(Lemma 2.6 in \cite{a})}\\
		Suppose that one can find constants $\xi_{n,v}^{'''}$ such that
		\begin{equation}\label{eq:2.12}
		\sup_{u:|u|\leq v}|1-\mathbb{E}(e^{iu.\zeta_i^n}|\mathcal{F}_{(i-1)/n})|\leq\frac{\xi_{n,v}^{'''}}{n},\quad |v|\le1
		\end{equation}	
		then $\eqref{eq:2.9}$ holds with $\xi_n=\xi_n^{'}= C\xi_{n,1}^{'''}$ and $\xi_{n,y}^{''}=C \xi_{n,1/y}^{'''}$.
	\end{lemma}
	We recall some important lemmas.
	The following theorems are from Theorem 2.3.7. and Theorem 4.2.3. in  \cite{App}.
	\begin{lemma}\label{lem:bigj}
	Let $N$ be a Poisson process with intensity function $\mu$ and $A$ be bounded below. Then if $f\in L^1(A,\mu(A))$, we have $\mathbb{E}(\int_{A}f(x)N(t,dx))=t\int_Af(x)\mu(dx)$. 
	\end{lemma}
	Let us denote  $$\mathcal{H}_2(T,E)=\left\{F:[0,T]\times E\times\Omega\rightarrow \mathbb{R}\mid F\textrm{ is predictable and }\int_0^T\int_E\mathbb{E}(| F(t,x)|^2)\rho(dt,dx)\right\}$$ and $I_T(F)=\int_0^T\int_EF(t,x)M(dt,dx)$ with $M$ is a martingale satisfying $M(\{0\},A)=0$ a.s. and there exists a $\sigma$-finite measure $\rho$ on $\mathbb{R}^+\times E$ for which $\mathbb{E}(M(t,A)^2)=\rho(t,A)$ for any $A\in\mathcal{B}(E)$.  
	\begin{lemma}\label{lem:smallj}
	If $F\in\mathcal{H}_2(T,E)$ then $\mathbb{E}(I_T(F)^2)=\int_0^T\int_E\mathbb{E}(|F(t,x)|^2)\rho(dt,dx)$.
	\end{lemma}
	The following theorem is about the convergence of infinitely divisible distributions. Justified by the one-to-one correspondence between infinitely divisible distributions $\mu$ and their characteristics $(a,b,\nu)$, we may write $\mu=id(a,b,\nu)$. This can be found from Theorem 15.14 (i) of Kallenberg \cite{e} or equivalently Theorem VII.3.4 of Jacod and Shiryaev \cite{g}.
	\begin{theorem}\label{thm:cv}
		Let $\mu=id(a,b,\nu)$ and $\mu_n=id(a_n,b_n,\nu_n)$ on $\mathbb{R}^d$, and fix any $h>0$ with $\nu\{|x|=h\}=0$. Define 
		\begin{align*}
			a^h=a+\int_{|x|\leq h}xx^{\top}\nu(dx),\quad
			b^h=b-\int_{h<|x|\leq1}x\nu(dx),
		\end{align*}  
		where $\int_{h<|x|\leq1}x\nu(dx)=-\int_{1<|x|\leq h}x\nu(dx)$ when $h>1$.\\
		Then $\mu_n\stackrel{w}{\rightarrow}\mu$ iff $a_n^h\rightarrow a^h$, $b_n^h
		\rightarrow b^h$, and $\nu_n\stackrel{v}{\rightarrow}\nu$ on $\bar{\mathbb{R}^d}\backslash \{0\}$.
	\end{theorem}
	The following Lemma shows a way to prove the convergence in distribution of a triangular array. This can be found as Corollary 15.16 of Kallenberg \cite{e} or equivalently  Theorem VIII.2.29 of Jacod and Shiryaev \cite{g} which 
	\begin{lemma}\label{lem:cv}
		Consider in $\mathbb{R}^d$ an i.i.d. array $(\zeta^n_i)$ and let $\Gamma$ be $id(a,b,\nu)$. For any $h>0$ with $\nu\{|x|=h\}=0$, $\sum_{i=1}^{[n.]}\zeta_i^n\stackrel{\mathcal{L}}{\rightarrow}\Gamma$ iff
		\begin{enumerate}[(i)]
			\item $n\mathcal{L}(\zeta_1^n)\stackrel{v}{\rightarrow}\nu$ on $\bar{\mathbb{R}^d}\backslash \{0\}$ 
			\item  $n\mathbb{E}(\zeta_1^n;|\zeta_1^n|\leq h)\rightarrow b^h$
			\item $n\mathbb{E}(\zeta_1^n{\zeta_1^n}^{\top};|\zeta_1^n|\leq h)\rightarrow a^h$.
		\end{enumerate} 
	\end{lemma}
	\begin{remark}
		To check  $(i)$, we prove that
		$n\mathbb{E}(1_{|\zeta^n_1|>\omega})\rightarrow \nu\{|x|>\omega\},$	for any $\omega>0$.
	\end{remark}
	\begin{lemma}\label{lem:t1}(Lemma 2.1 in \cite{a})
	If for each $n$ $\zeta_i^n$, $i=1,2,\hdots$ are i.i.d. random variables and $\Gamma^n_1$ converges in law to a limit $U$, then there is a L\'evy process $\Gamma$ such that $\Gamma_1=U$. This process $\Gamma$ is unique in law and $\Gamma^n$ converges in law to $\Gamma$ (for the Skorokhod topology). Further, the sequence $(\Gamma^n)$ has $(UT)$. 	 	
\end{lemma}
Let $Z^n_t=\sum_{i=1}^{[nt]}\eta_i^n$,  $\Gamma^n_t=\sum_{i=1}^{[nt]}\zeta^n_i$ and ${\Gamma'}^n_t=\sum_{i=1}^{[nt]}{\zeta'}^n_i$ with $\zeta^n_i=g(X^n_{\frac{i-1}{n}}){\zeta'}^n_i$. 
 For each $n$, if the sequence $(\eta_i^n,{\zeta'}^n_i)$, $i=1,2,\hdots,$ is i.i.d., combining Lemma \ref{lem:t1} with Theorem \ref{thm:ut} (iii), we get the following lemma which is very similar to Lemma 2.8 in \cite{a}.
\begin{lemma}\label{lem:cv1tot}
	We suppose that the sequence $(Z^n,\Gamma^n)$ is tight. If the pair $(Z^n_1,{\Gamma'}^n_1)$ of random variables converges in law to $(Z_1,{\gamma'})$ with ${\gamma'}$ a random variable independent of $Z_1$ and that $g$ is a  Lipschitz-continuous function, then there is a L\'evy process $\Gamma'$, independent of $Y$ and unique in law, such that the processes $(Z^n,{\Gamma'}^n,{\Gamma}^n)$ converge in law to $(Z,{\Gamma'},{\Gamma})$, where ${\Gamma}_t=\int_0^tg(X_{s-})d{\Gamma'}_s$. If further ${\gamma'}$ is a constant, then we get ${\Gamma}_t=\int_0^tg(X_{s-}){\gamma'}ds$, and the convergence of $(Z^n,{\Gamma'}^n,{\Gamma}^n)$ takes place in probability.
\end{lemma}
\begin{proof}
 We rewrite $\Gamma^n=\Gamma^{n,1}+\Gamma^{n,2}$ where $\Gamma^{n,1}_t=\sum_{i=1}^{[nt]}(g(X^n_{\frac{i-1}{n}})-g(X_{\frac{i-1}{n}})){\zeta'}^n_i$ and $\Gamma^{n,2}_t=\sum_{i=1}^{[nt]}g(X_{\frac{i-1}{n}})){\zeta'}^n_i$.
	First, since $(Z^n_1,{\Gamma'}^n_1)$ converges in law to $(Z_1,{\gamma'})$, then by Lemma 2.8 in \cite{a}, we have $(Z^n,{\Gamma'}^n,{\Gamma}^{n,2})\stackrel{\mathcal{L}}{\rightarrow} (Z,{\Gamma'},{\Gamma})$.
	Second, by \cite[Theorem 1.2]{a} and using that $g$ is Lipschitz, we easily deduce that $g(X_{\eta_n(.)}^n)-g(X_{\eta_n(.)})\stackrel{\mathbb{P}}{\rightarrow}0$. Then, since  ${\Gamma}^{n,2}\stackrel{\mathcal{L}}{\rightarrow}{\Gamma}$, we apply Lemma \ref{lem:t1} to get the $(UT)$ property of ${\Gamma}^{n,2}$ and   Theorem \ref{thm:ut} (iii) to obtain ${\Gamma}^{n,1}\stackrel{\mathbb{P}}{\rightarrow}0$.
	Therefore, we get $(Z^n,{\Gamma'}^n,{\Gamma}^n)=(Z^n,{\Gamma'}^n,{\Gamma}^{n,1}+{\Gamma}^{n,2})\stackrel{\mathcal{L}}{\rightarrow} (Z,{\Gamma'},{\Gamma})$. 
	\end{proof}
	\subsection{Evaluation of logarithm and exponential functions}\label{eval}
Here, we use the power series expansions for both  functions   $\log{(1+z)}$ and $e^z$ for $z\in\mathbb{C}$ (see e.g.\ \cite{Gronwall}). We know that if $z\in\mathbb{C}$ and $|z|<\frac{1}{2}$, we have 
$
\log{(1+z)}-z=-\sum_{n\geq2}(-1)^n\frac{z^n}{n}.
$
Then 
$
|\log{(1+z)}-z|\leq\frac{1}{2}\sum_{n\geq2}|z|^n.
$
By applying the formula for convergent geometric sum, we have
$
|\log{(1+z)}-z|\leq\frac{1}{2}|z|^2\frac{1}{1-|z|}.
$
Now, since $|z|<\frac{1}{2}$, then
$
\frac{1}{1-|z|}<2
$
then $\forall z\in \mathbb{C}$ such that $|z|<\frac{1}{2}$ we have 
$\boxed{
|\log{(1+z)}-z|\leq|z|^2}.
$ We can proceed in the same way to prove that $\forall z\in \mathbb{C}$ such that $|z|<\frac{1}{2}$ we also have 
$\boxed{
	|e^z-1-z|\leq|z|^2}$ and by consequence, we get $\boxed{
	|e^z-1|\leq \frac{3}{2}|z|}$.


\begin{thebibliography}{10}

\bibitem{App}
D.~Applebaum.
\newblock {\em L\'{e}vy processes and stochastic calculus}, volume 116 of {\em
  Cambridge Studies in Advanced Mathematics}.
\newblock Cambridge University Press, Cambridge, second edition, 2009.

\bibitem{k}
M.~Ben~Alaya and A.~Kebaier.
\newblock Central limit theorem for the multilevel {M}onte {C}arlo {E}uler
  method.
\newblock {\em Ann. Appl. Probab.}, 25(1):211--234, 2015.

\bibitem{BAKN}
M.~Ben~Alaya, A.~Kebaier, and T.~B.~T. Ngo.
\newblock Central limit theorem for the $\sigma$-antithetic multilevel monte
  carlo method, 2020.

\bibitem{DereichLi}
S.~Dereich and S.~Li.
\newblock {Multilevel Monte Carlo for Lévy-driven SDEs: Central limit theorems
  for adaptive Euler schemes}.
\newblock {\em The Annals of Applied Probability}, 26(1):136 -- 185, 2016.

\bibitem{Gil}
M.~B. {Giles}.
\newblock Multilevel monte carlo path simulation.
\newblock {\em Oper. Res.}, 56:607--617, 2008.

\bibitem{Gior}
D.~Giorgi, V.~Lemaire, and G.~Pag\`es.
\newblock Limit theorems for weighted and regular multilevel estimators.
\newblock {\em Monte Carlo Methods Appl.}, 23(1):43--70, 2017.

\bibitem{h}
C.~Graham, T.~G. Kurtz, S.~M\'{e}l\'{e}ard, P.~E. Protter, M.~Pulvirenti, and
  D.~Talay.
\newblock {\em Probabilistic models for nonlinear partial differential
  equations}, volume 1627 of {\em Lecture Notes in Mathematics}.
\newblock Springer-Verlag, Berlin; Centro Internazionale Matematico Estivo
  (C.I.M.E.), Florence, 1996.
\newblock Lectures given at the 1st Session and Summer School held in
  Montecatini Terme, May 22--30, 1995, Edited by Talay and L. Tubaro,
  Fondazione CIME/CIME Foundation Subseries.

\bibitem{Gronwall}
T.~H. Gronwall.
\newblock On the power series for $log (1 + z)$.
\newblock {\em Annals of Mathematics}, 18(2):70--73, 1916.

\bibitem{a}
J.~Jacod.
\newblock The {E}uler scheme for {L}\'evy driven stochastic differential
  equations: limit theorems.
\newblock {\em Ann. Probab.}, 32(3):1830--1872, 07 2004.

\bibitem{b}
J.~Jacod and P.~Protter.
\newblock Asymptotic error distributions for the {E}uler method for stochastic
  differential equations.
\newblock {\em Ann. Probab.}, 26(1):267--307, 1998.

\bibitem{JacodProtter2012}
J.~Jacod and P.~Protter.
\newblock {\em Discretization of processes}, volume~67 of {\em Stochastic
  Modelling and Applied Probability}.
\newblock Springer, Heidelberg, 2012.

\bibitem{g}
J.~Jacod and A.~Shiryaev.
\newblock {\em Limit theorems for stochastic processes}, volume 288.
\newblock A serie of comprehensive studies in Mathematics, 2nd edition, 2003.

\bibitem{f}
A.~Jakubowski, J.~M{\'e}min, and G.~Pag\`es.
\newblock Convergence en loi des suites d'int{\'e}grales stochastiques sur
  l'espace {$\mathbb{D}^1$} de skorokhod.
\newblock {\em Probability Theory and Related Fields}, 81(1):111--137, Feb
  1989.

\bibitem{e}
O.~Kallenberg.
\newblock {\em Foundations of modern probability}.
\newblock Probability and its Applications (New York). Springer-Verlag, New
  York, second edition, 2002.

\bibitem{d}
K.~I. Sato.
\newblock L\'evy processes and infinitely divisible distributions.
\newblock In {\em L\'evy Processes and Infinitely Divisible Distributions},
  1999.

\bibitem{m}
H.~Wang.
\newblock The {E}uler scheme for a stochastic differential equation driven by
  pure jump semimartingales.
\newblock {\em Journal of Applied Probability}, 52:149--166, 03 2015.
\end{thebibliography}
\end{document}